\documentclass[reqno]{amsart}
\usepackage{mathtools}
\usepackage{tkz-euclide}
\usepackage{xfrac}
\usepackage{hyperref}
\usepackage{amsmath}  
\usepackage{amsfonts}
\usepackage{amssymb} 
\usepackage{mathrsfs} 
\usepackage{graphicx}  
\usepackage{bm}
\usepackage{tikz} 
\usetikzlibrary{svg.path} 
\usepackage{amssymb}
\usepackage{pgfplots}
\usepackage{amsmath}
\usepackage[toc]{appendix}
\usetikzlibrary{calc} 
\usepackage{xcolor}  
\usetikzlibrary{arrows.meta} 
\usepackage{amsthm} 
\usepackage{float}
\usepackage{enumitem}
\usepackage{comment}

\usepackage[english]{babel}
\usepackage[font=footnotesize, labelfont=bf]{ caption}
\usepackage{ esint }
\usepackage{stackengine,scalerel,graphicx}
\stackMath

\definecolor{trueblue}{rgb}{0.0, 0.45, 0.81}
\definecolor{truegreen}{rgb}{0.13, 0.55, 0.13}

\newcommand{\eop}{\nopagebreak\hspace*{\fill}$\Box$\smallskip}
\newcommand{\NNN}{\color{black}} 
\newcommand{\MMM}{\color{black}}

\newcommand{\EEE}{\color{black}}

\newcommand{\QQQ}{\color{black}}

\newcommand{\eps}{\varepsilon}

\theoremstyle{plain}
\begingroup
\newtheorem{theorem}{Theorem}[section]
\newtheorem{lemma}[theorem]{Lemma}

\newtheorem{remark}[theorem]{Remark}
\newtheorem{proposition}[theorem]{Proposition}
\newtheorem{corollary}[theorem]{Corollary}
\endgroup

\usepackage{todonotes}

\theoremstyle{definition}
\begingroup
\newtheorem{definition}[theorem]{Definition}
\endgroup

\renewcommand{\tilde}{\widetilde}

\usepackage{tikz} 

\usepackage{adjustbox}

\definecolor{color0}{HTML}{000000}
\definecolor{color1}{HTML}{000000}
\definecolor{color2}{HTML}{f80707}
\definecolor{bgColor}{HTML}{ffffff}
\usetikzlibrary {arrows.meta,bending}

\numberwithin{equation}{section}
\newcommand{\N}{\mathbb{N}}
\newcommand{\Z}{\mathbb{Z}}
\newcommand{\R}{\mathbb{R}}

\newcommand{\defas}{:=}

\newcommand{\Q}{\mathcal{Q}}

\usepackage[normalem]{ulem}

\begin{document}
\author[S. Almi]{Stefano Almi}
\address[Stefano Almi]{Dipartimento di Matematica e Applicazioni “R. Caccioppoli”, Università
di Napoli Federico II, via Cintia, 80126 Napoli, Italy}
\email{stefano.almi@unina.it}

\author[A. Donnarumma]{Antonio Flavio Donnarumma}
\address[Antonio Flavio Donnarumma]{Department of Mathematics, Friedrich-Alexander Universit\"at Erlangen-N\"urnberg. Cauerstr.~11,
D-91058 Erlangen, Germany}
\email{antonio.flavio.donnarumma@fau.de}

\author[M. Friedrich]{Manuel Friedrich}
\address[Manuel Friedrich]{Department of Mathematics, Johannes Kepler Universität Linz. Altenbergerstrasse 66, 4040 Linz, Austria}
\email{manuel.friedrich@jku.at}

\title[Asymptotic analysis  for heterogeneous elastic energies   with material voids]{Asymptotic analysis  for heterogeneous \\ elastic energies   with material voids}
\maketitle
\begin{abstract}

We study the effective behavior of heterogeneous energies arising in the modeling  of material voids in geometrically linear elastic materials. Specifically, we consider functionals featuring bulk terms depending  on the symmetrized gradient of the displacement and terms comparable to the surface area of the material voids inside the material. Under suitable growth conditions for the bulk and surface densities we prove that, as the microscale $\varepsilon$ tends to zero,  the  $\Gamma$-limit  admits an integral representation that contains an additional surface term expressed by  jump discontinuities of the displacement  outside of the void region. This term  is related to the phenomenon of  collapsing of voids in the effective limit. Under a continuity assumption of the surface density at the $\varepsilon$-scale, we show that the limiting density related to jumps is  twice the  energy density for voids. 
 
\end{abstract}

\EEE
 
\section{Introduction}
The derivation of homogenization and general $\Gamma$-convergence results for integral functionals has been a thriving field over the last decades. We refer to \cite{BraidesDefranceschi1998} for an overview in the Sobolev setting and mention, without claiming exhaustiveness, various extensions to $BV$-type frameworks featuring bulk and surface terms \cite{BLZ,BraDefVit96,caterina-bv, cagnetti2018gammaconvergence, cagnetti2017stochastic, cagnettiscardia, donati,  dalmaso-toader, DonnarummaFriedrich,  FriPerSol20a, giapon04a, Marziani_2023}. The analysis generally relies on the localization method for $\Gamma$-convergence \cite{DalMaso:93}, combined with a global method \cite{Bouchitt2002AGM,Bouchitt1998} that allows   to represent limits as integral functionals. The goal of the present article is to extend this  approach to energies in elasticity with material voids.

The formation of material voids inside elastically stressed solids can be formulated in the context of \emph{stress-driven rearrangement instabilities} (SDRI), see \cite{AsaroTiller1972, Grin86, Grinfeld1993, SuoWang1994}. Energy functionals describing SDRI are characterized by the competition between stored elastic bulk energy and surface contributions. Problems of this type have been intensively studied in recent years using variational methods,  including results on existence, regularity, and relaxation \cite{brachasol05, Crismale_2020,fonfusleo11a_VOID, KholmatovPiovano2020, Siegel2004Evolution_VOID}, as well as on linearization \cite{KFZ:2021,  KFZ:2022}, and dimension reduction \cite{ FKZ_elasticrodswithvoids, friedrich2025derivationkirchhofftypeplatetheories, SantilliSchmidt2023BZK, SantilliSchmidt2023}. Despite this recent progress, homogenization in this setting remains largely unexplored and appears to be limited to a periodic homogenization result for a scalar model \cite{solci}.

\textbf{Setting and background:}
We now describe in more detail the energy functionals modeling material voids in elastically stressed solids. In the setting of linearized elasticity, \NNN a prototypical energy \EEE takes the form
\begin{equation}
\label{simplified functional}
\mathcal{I}(u,E)=\int_{\Omega \setminus E}\mathbb{C}e(u)\colon e(u) \, \mathrm{d}x+\int_{\Omega \cap \partial E}g(\nu_E) \, \mathrm{d}\mathcal{H}^{d-1},
\end{equation}
where $\Omega \subset \mathbb{R}^d$ denotes the reference domain, $E$ is the (sufficiently smooth) void set, and $u \in H^1(\Omega \setminus \overline{E}; \mathbb{R}^d)$ denotes the displacement field. The first term corresponds to the elastic bulk energy, which is quadratic in the symmetrized gradient $e(u)\defas\frac{\nabla u + \nabla u^T}{2}$ and depends on the fourth-order, positive semi-definite elasticity tensor $\mathbb{C}$. The surface term accounts for the presence of the void set $E$, where $g$ is a norm and $\nu_E$ denotes the outer unit normal to $\partial E$. If $u$ is of the form $u=(0,\ldots,0,v)$ for a scalar function $v$, then \eqref{simplified functional} reduces to a scalar problem.

Functionals of the type \eqref{simplified functional} are separately lower semicontinuous: for fixed $E$ with Lipschitz boundary, the map $u \mapsto \mathcal{I}(u,E)$ is weakly lower semicontinuous in $H^1$, while for fixed $u$ the map $E \mapsto \mathcal{I}(u,E)$ can be regarded as a lower semicontinuous functional on sets of finite perimeter. However, as pointed out by {\sc Braides, Chambolle, and Solci}    \cite{brachasol05}, the functional \eqref{simplified functional} is not lower semicontinuous with respect to the pair $(u,E)$. More precisely, {\sc Crismale} and the third author showed in \cite{Crismale_2020} that the lower semicontinuous envelope of \eqref{simplified functional} coincides with
\begin{equation}
\label{simplified relaxed functional}
\overline{\mathcal{I}}(u,E)= \int_{\Omega \setminus E}\mathbb{C}e(u)\colon e(u) \, \mathrm{d}x+\int_{\Omega \cap \partial^* E}g(\nu_E) \, \mathrm{d}\mathcal{H}^{d-1}+\int_{J_u \setminus \partial^* E}2g(\nu_u) \, \mathrm{d}\mathcal{H}^{d-1}.
\end{equation}
We refer to \cite{brachasol05} for analogous results in the scalar case and in nonlinear elasticity. In the relaxed formulation, regular void sets are replaced by sets of finite perimeter with essential boundary $\partial^*E$. The interaction between the displacement field and the void set gives rise to an additional surface term involving the jump set $J_u$ with normal $\nu_u$. Indeed, during relaxation, voids may collapse into discontinuities of $u$, and such collapsed interfaces contribute twice to the surface energy, yielding the density $2g$. Note that \eqref{simplified relaxed functional} is technically more involved than its scalar or nonlinear counterparts, since $u$ is no longer a $(G)SBV$ function \cite{ambrosio2000fbv}, but rather belongs to the space of generalized special functions of bounded deformation ($GSBD$) \cite{DALMASO_GBD}. A key achievement of \cite{Crismale_2020} lies in overcoming the lack of Korn's inequality and dealing with the resulting analytical difficulties in this setting.

In the scalar framework, the result of \cite{brachasol05} was extended by {\sc Solci} \cite{solci} to heterogeneous materials with periodic microstructure. Specifically, energies of the form
\begin{equation}
\label{solci non homo}
\int_{\Omega \setminus E}f\Big(\frac{x}{\varepsilon},\nabla y\Big) \, \mathrm{d}x+\int_{\Omega \cap \partial E}g\Big(\frac{x}{\varepsilon},\nu_E \Big) \, \mathrm{d}\mathcal{H}^{d-1}
\end{equation}
are considered, where $\varepsilon>0$  represents the size of the microstructure, and $f$ and $g$ are densities periodic in $x$ and satisfying suitable growth conditions.  Under the additional assumption that $g$ is a norm in the $\nu$-variable, the energies \eqref{solci non homo} $\Gamma$-converge to
\begin{equation*}
\int_{\Omega \setminus E}f_\mathrm{hom}(\nabla y) \, \mathrm{d}x+\int_{\Omega \cap \partial^* E}g_\mathrm{hom}(\nu_E) \, \mathrm{d}\mathcal{H}^{d-1} + \int_{J_u \setminus \partial^*E}2g_\mathrm{hom}(\nu_u) \, \mathrm{d}\mathcal{H}^{d-1},
\end{equation*}
for suitable homogenized densities $f_{\rm hom}$ and $g_{\rm hom}$ which are independent of the $x$-variable.

The aim of this work is to derive effective models for heterogeneous materials with material voids in the setting of linearized elasticity. In this way, we generalize \cite{solci} to (a) the vectorial, geometrically linear setting and to (b) materials with possibly nonperiodic microstructures. In this sense, our work can be understood as a counterpart of \cite{cagnetti2018gammaconvergence, FriPerSol20a} where we consider material voids in place of fracture sets. The main novelty of our contribution lies in the analysis of collapsing material voids along the homogenization process. 
\NNN Specifically, \EEE we study the asymptotic behavior of energy functionals of the form
\begin{equation}
\label{non homo energies linear elasticity}
\mathcal{E}_\varepsilon(u,E)=\int_{\Omega \setminus E}f_\varepsilon(x,e(u))\, \mathrm{d}x+\int_{\partial^* E \cap \Omega}g_\varepsilon(x,\nu_E)\, \mathrm{d}\mathcal{H}^{d-1},
\end{equation}
for suitable bulk and surface densities $f_\varepsilon$ and $g_\varepsilon$, see $(f_1)$–$(f_3)$ and $(g_1)$–$(g_3)$ in Section \ref{setting and main results} for the precise assumptions. We emphasize that \eqref{non homo energies linear elasticity} allows for  more general microstructures than the purely periodic case, which is recovered by the specific choice $f_\varepsilon(x,\xi)=f(\frac{x}{\varepsilon},\xi)$ and $g_\varepsilon(x,\nu)=g(\frac{x}{\varepsilon},\nu)$, for functions $f$ and $g$ being periodic in $x$. Further examples of energies involving the competition between bulk and surface terms in nonperiodic settings can be found, for instance, in \cite{ caterina-bv, cagnetti2018gammaconvergence, cagnetti2017stochastic} within a generalized $BV$ framework and in \cite{DonnarummaFriedrich, FriPerSol20a, Marziani_2023} for brittle fracture in linear elasticity.

For reasons discussed below, we adopt a modeling framework in which voids are sets of finite perimeter, in contrast to the setting of \eqref{simplified relaxed functional}. Correspondingly, the displacement field $u$ is regarded as an $SBV$ function defined on $\Omega$, satisfying $u=0$ on $E$ and $J_u \subset \partial^* E$, see \NNN \eqref{newenergies}--\eqref{SBVdefW} \EEE for details.

\textbf{Main results:}
 In our first main result (Theorem \ref{first gamma convergence result}), we  show that  functionals of the form \eqref{non homo energies linear elasticity} $\Gamma$-converge, up to subsequences, to a limiting functional $\mathcal{E}_0$ of the form
\begin{equation}
\label{non homo limit functional}
\mathcal{E}_0(u,E)=\int_{\Omega \setminus E}f_0(x,e(u))\, \mathrm{d}x+\int_{\partial^* E \cap \Omega}g_0(x,\nu_E)\, \mathrm{d}\mathcal{H}^{d-1}+\int_{J_u \setminus \partial^*E}h_0(x,[u],\nu_u)\, \mathrm{d}\mathcal{H}^{d-1},
\end{equation}
where the densities $f_0$, $g_0$, and $h_0$ are defined via suitable asymptotic cell formulas, see \eqref{def f_0}–\eqref{def h_0}, and   $[u]$ denotes the jump height. Here,   $\Gamma$-convergence is understood with respect to convergence in measure for $u$ and $L^1$-convergence of the characteristic function of $E$. In contrast to \eqref{simplified relaxed functional}, the surface densities satisfy only the inequality $h_0 \ge 2g_0$, and equality does not hold in general, see Remark~\ref{counterexample}. To the best of our knowledge, this phenomenon, namely starting from a single surface density $g_\varepsilon$ and obtaining two independent surface densities in the limit, has not previously been observed in the literature. However, if $g_\varepsilon$ is continuous with respect to the spatial variable, we can indeed prove that $h_0=2g_0$, see Proposition~\ref{limsupliminf}. \NNN The idea is that, in \EEE a blow-up, collapsing voids correspond to two parallel and infinitesimally close hypersurfaces, whose energies are comparable only under suitable continuity assumptions.

Our second main result (Theorem \ref{Identification of Gamma-limit}) concerns the identification of the $\Gamma$-limit $\mathcal{E}_0$ by analyzing the relation between the sequences $(f_\varepsilon)_\varepsilon$, $(g_\varepsilon)_\varepsilon$ and the limiting densities $f_0$, $g_0$, and $h_0$. In line with other $\Gamma$-convergence results for free-discontinuity problems (see, e.g., \cite{cagnetti2018gammaconvergence, FriPerSol20a}), we show that bulk and surface densities do not interact in the limit. More precisely, the bulk density $f_0$ is obtained as the limit of infimum problems defined on Sobolev functions without voids, namely
\begin{equation*}
    f_0(x, \xi) \EEE =\limsup_{\rho \to 0}\frac{1}{\gamma_d \rho^d}\limsup\limits_{\varepsilon \to 0}\inf_{v}\int_{B_{\rho}(x)}f_{\varepsilon}(y,e(v))\, \mathrm{d}y
\end{equation*}
for $x \in \Omega$ and $\xi \in \R^{d \times d}$, where $B_{\rho}(x)$ denotes the ball of radius $\rho$ centered in $x$, $\gamma_d$ is the volume of the $d$-dimensional unit ball, and the infimum is taken on Sobolev functions $v$ satisfying  $v(x) =\xi x$ near $\partial B_{\rho}(x)$.  Instead, the surface density $g_0$ is given by
\begin{equation}\label{g-density}
      g_0(x,\nu)=\limsup_{\rho \to 0}\frac{1}{\gamma_{d-1} \rho^{d-1}}\limsup\limits_{\varepsilon \to 0}\inf_{F}\int_{B_{\rho}(x)\cap \partial^* F}g_{\varepsilon}(y,\nu_F)\, \mathrm{d}y
\end{equation}
for $x \in \Omega$ and $\nu \in \mathbb{S}^{d-1}$, where $\gamma_{d-1}$ is the volume of the $(d-1)$-dimensional unit ball and the infimum problem is defined on sets of finite perimeter coinciding with the half space  $\lbrace y \in \R^d\colon   (y-x) \cdot \nu \le 0 \rbrace $ near $\partial B_{\rho}(x)$.  These formulas coincide with those in \cite{cagnetti2018gammaconvergence, FriPerSol20a}. In particular, the densities decouple in the sense that $f_0$ is not influenced by $(g_\eps)_\eps$ and   $g_0$ is not influenced by $(f_\eps)_\eps$.

The characterization of $h_0$ is more subtle. Restricting to the planar case, we show that  $h_0(x,\nu)$ \NNN is independent of $[u]$ and \EEE admits a representation similar to \eqref{g-density},  but in place of a half plane as boundary condition, we assume that the competitor set $F$  is close to the strip  $\lbrace y \in \R^d\colon   |(y-x) \cdot \nu| \le \eps \rbrace $ near $\partial B_{\rho}(x)$, and needs to `separate' $B_\rho(x)$ in the sense that $B_\rho(x) \setminus F$ consists of two connected components. This formulation is tailor-made to capture the collapse of voids into displacement discontinuities. The restriction to $d=2$ is due to the availability of approximation results for $GSBD$ functions by piecewise rigid motions only in the planar setting.

As a corollary, again in dimension $d=2$, we obtain a (periodic) homogenization result using techniques analogous to those in \cite{BraDefVit96, cagnetti2018gammaconvergence, cagnetti2017stochastic}, see Corollary \ref{homogenization}. In  Remark \ref{periodic-homo-remark},  we observe that in the periodic case the identity $h_0=2g_0$ holds even without assuming continuity of the surface density. \NNN Finally, in any space dimension, we consider the special case $f_\eps = f$ and $g_\eps = g$ for all $\eps>0$, and show that the relaxation can be represented as an integral of the form \eqref{simplified relaxed functional}, see Theorem \ref{relaxation}. \EEE

\textbf{Proof strategy:}
Due to the possible presence of nonperiodic microstructures, our proof strategy follows the localization method for $\Gamma$-convergence \cite[Chapters 14–20]{DalMaso:93}, combined with an integral representation result. We  comment on some of the main technical challenges.

(a) The identification of $f_0$ and $g_0$ follows the approach of \cite[Theorem 2.4]{FriPerSol20a}, which consists in restricting the asymptotic minimization problems for pairs of function-set $(u,E)$  either to Sobolev functions without voids or a problem exclusively defined on sets. Korn's inequality for functions with small jump sets \cite{CagChaSca20} plays a crucial role for the bulk density. The identification of $h_0$, by contrast, requires a different Korn-type inequality available only in the planar setting, allowing the approximation by piecewise rigid motions \cite[Theorem 4.1]{FriedrichSolombrino2018}. Following \cite[Lemma 5.2]{FriPerSol20a}, we first construct one-dimensional separating sets, which are then replaced via covering arguments by boundaries of two-dimensional voids,   still separating $B_\rho(x)$ and thus being admissible competitors for the asymptotic cell formula of $h_0$. 

(b) A key ingredient both for modifying boundary data in the cell formulas and for the localization argument is the construction of suitable interpolations between pairs $(u,E)$ and $(v,F)$, commonly referred to as the \emph{fundamental estimate}. In our setting (Proposition \ref{fundamentalestimate}), this is achieved by a careful combination of corresponding results \EEE for functionals defined on $GSBD$-functions (see \cite[Proposition~4.1]{FriPerSol20a}) and  Caccioppoli partitions  (see \cite[Lemma 4.4]{AmbBra90}) which in turn crucially hinge on Korn inequalities in $SBD$ to interpolate two functions and the  Fleming-Rishel-formula to join two sets.  Throughout the construction, it is essential to preserve the constraints $J_u\subset\partial^*E$ and $u=0$ on $E$. Moreover, for the application to the asymptotic cell formula for $h_0$, it is essential that the fundamental estimate preserves the separation property if two sets $E$ and $F$ separate $B_\rho(x)$, \NNN cf.\ Corollary \ref{set.corollary}. \EEE

In this regard, let us we highlight that it seems out of reach to obtain a   fundamental estimate for functionals defined on partitions of Lipschitz sets. This observation is the reason why  we relax the setting of   \eqref{simplified relaxed functional} and consider a model where voids are sets of finite perimeter rather than Lipschitz sets. For relaxation or $\Gamma$-convergence results, this different modeling assumption is irrelevant since effective limits necessarily involve relaxation to sets of finite perimeter.

%
%
%
%

(c) The proof of the identity $h_0=2g_0$ under continuity assumptions on $g_\varepsilon$ is based on replacing the boundary $\Gamma = \partial^* F \cap B_\rho(x)$ of an admissible competitors $F$ for $g_0$ by a separating void set whose boundary essentially consists of two closely shifted copies of $\Gamma$. Continuity allows  to control the associated surface energy, while additional approximation by polyhedral sets \cite{BraConGar16} is used to estimate the remaining boundary contributions related to the shifting of $\Gamma$.

(d) Finally, the integral representation result relies on the \emph{global method for relaxation} developed in \cite{Bouchitt2002AGM,Bouchitt1998}, following closely the strategy for functionals on $GSBD$ spaces, see \cite{SolFriCri20} for details. To our knowledge, this provides the first general integral representation result for functionals defined on pairs of function-set. The main necessary adaptations    consist in showing that, in the equivalence of asymptotic cell formulas with different boundary conditions, see \cite[Sections 5 and 6]{SolFriCri20}, the  presence of the void set can be  handled. This is essentially achieved by  considering $(u,E)$ as an $GSBD$-function with $J_u \subset \partial^* E$.

%

\textbf{Further perspectives:}
Similar results could be obtained in the context of nonlinear elasticity, where several technical difficulties could be avoided, such as the lack of control on the skew-symmetric part of $\nabla u$ and the use of $GSBD$. In that setting, the identification of the $\Gamma$-limit is expected to hold in arbitrary space dimensions. In the present work, the restriction to $d=2$ arises solely from the approximation result \cite[Lemma 5.2]{FriPerSol20a}. In a $GSBV$-framework, analogous approximations can be obtained in any dimension using the coarea formula, see, e.g., \cite{BraDefVit96, cagnetti2018gammaconvergence, Ambrosio94}.

 In this work, we do not consider Dirichlet boundary conditions, but note that they could in principle be incorporated, leading to convergence results for (almost) minimizers with prescribed boundary data. We refer to \cite{Fri19;08} and \cite{FriPerSol20a} for corresponding results in the geometrically nonlinear and linearized setting, respectively,  dealing with surface discontinuities in place of voids.

\textbf{Outline:}
The paper is organized as follows. In Section \ref{setting and main results}, we introduce the setting and state the main results. Section \ref{section proof 1} is devoted to compactness of $\Gamma$-convergence and integral representation of the resulting limiting functional. Sections \ref{section proof of proposition limsupliminf}–\ref{section proof 2} address the identification of the $\Gamma$-limit. As some proofs on the bulk densities are very similar to the ones in \cite{SolFriCri20, FriPerSol20a}, they are omitted, but for the convenience of the reader we include the essential steps in Appendix \ref{otherproof-sec}. Eventually,  in Appendix \ref{section preliminaries}, we  collect basic properties on the function space $GSBD$ used throughout  the paper.

%
%

\NNN \textbf{Notation:} \EEE We close the introduction by introducing    notation which we use  throughout the paper: \EEE

\begin{itemize}
    \item[(i)] We denote by $\mathcal{A}$   the family of all open, bounded subsets of $\mathbb{R}^d$, \MMM and let  $\mathcal{A}_0 \subset \mathcal{A}$  be the bounded sets \EEE   with Lipschitz boundary. Given $\Omega \in \mathcal{A}_0$, $\mathcal{A}_0(\Omega)$ denotes the family of sets $A \in \mathcal{A}_0$ satisfying $A \subset \Omega$. \MMM Similarly, by $\mathcal{B}(\Omega)$ we denote the Borel sets contained in $\Omega$. For $A_1$, $A_2$, $A_3 \in \mathcal{A}$, we denote by $A_1 \triangle A_2$ the symmetric difference, and we say that `$A_1 = A_2$ on $A_3$' if $A_1 \cap A_3  = A_2 \cap A_3 $. \EEE   
\item[(ii)] \MMM Given \EEE two vectors $x$, $y \in \mathbb{R}^d$ we  denote by $x \cdot y$ their inner product. \MMM We set $\mathbb{S}^{d-1} := \lbrace x \in \R^d \colon |x| =1 \rbrace$. By $\R^{d\times d}_{\rm sym}$ and $\R^{d\times d}_{\rm skew}$ we denote the set of symmetric and skew-symmetric matrices, respectively. \EEE 
     \item[(iii)] Given $x \in \mathbb{R}^d$, $\rho>0$, we denote by $Q_{\rho}(x)$ the $d$-dimensional open cube with center in $x$ and sidelength $\rho$, oriented according to the canonical  basis.  \MMM We   often drop the dependence on $x$ if $x=0$. \EEE   
    Moreover,  we denote by $B_\rho(x)$ the ball centered in $x$   with radius $\rho$.  If $x=0$, we just write $B_\rho$. 
     \item[(iv)] By $\mathcal{L}^k$ and $\mathcal{H}^k$ we denote the $k$-dimensional Lebesgue and Hausdorff measure, \MMM respectively. We  \EEE  set $\gamma_k\defas \mathcal{L}^k(B^k_1)$, where $B^k_1$ denotes the $k$-th dimensional ball of radius one and center in zero. 
    \item[(v)] \MMM  Given $\Omega \in \mathcal{A}_0$, we denote by $L^0(\Omega;\R^d)$ the space of $\mathcal{L}^d$-measurable functions on $\Omega$ with values in $\R^d$, endowed with the topology of  convergence in measure.  Moreover, \EEE   we let  \EEE $\mathcal{M}(\Omega)\defas \{ A \subset \Omega \colon A\:\: \text{\NNN Lebesgue \EEE measurable}\}$. We equip  \MMM $\mathcal{M}(\Omega)$ \EEE with the metric given by the convergence in measure (or analogously by the convergence in $L^1$ of the characteristic functions),  \MMM  i.e., we say  $A_n$ \NNN converges \EEE  to $A$ if $\Vert \chi_{A_n}-\chi_A\Vert_{L^1(\Omega)}\to 0$.    In particular,  $A$, $B \in \mathcal{M}(\Omega)$ \MMM are identified \EEE  if $\mathcal{L}^d(A \triangle B)=0.$

     \item[(vi)]    Let $E \in \mathcal{M}(\Omega)$. For every $t \in [0,1]$, we denote by $E^t$ the \MMM set \EEE of points where $E$ has density $t$ i.e.,
    \begin{equation*}
        E^t \defas \Big \{ x \in \Omega \colon \lim\limits_{\rho \to 0}\frac{\mathcal{L}^d(B_\rho(x)\cap E)}{\gamma_d\rho^d}=t \Big\}.
    \end{equation*}
    \item[(vii)]  \MMM By $\mathcal{P}(\Omega)$ we denote all sets in $\mathcal{M}(\Omega)$ with  finite perimeter. For such sets,   \EEE we denote by $\partial^* E$ the essential boundary and by $\nu_E$ the approximate \MMM outer unit \EEE normal vector to $\partial^* E$. 
    \item[(viii)] Given $x \in \mathbb{R}^d$, $\varepsilon>0$, and a Borel set $F \subset \mathbb{R}^d$, we define $F_{\varepsilon,x}=x+\varepsilon(F-x)$.  
    \item[(ix)] Given $x \in \mathbb{R}^d$ and $\nu \in \mathbb{S}^{d-1}$, we \MMM define the hyperplane and halfspaces \EEE
    \begin{equation*}
        \Pi_{x}^\nu\defas \{ y\in \mathbb{R}^d \colon (y-x)\cdot \nu =0\}, \quad \quad  \Pi_{x}^{\nu, \pm}\defas \{ y\in \mathbb{R}^d \colon \pm (y-x)\cdot \nu >0\},
    \end{equation*}
\MMM where here and in the sequel $\pm$ stands as a placeholder for $+$ and $-$.  Moreover, \EEE given $\nu \in \mathbb{S}^{d-1}$ and a ball $B_\rho(x)$, we denote by $B_{\rho}^{\nu,\pm}(x)\defas \{ y \in B_\rho(x) \colon \pm(y-x)\cdot \nu >0\}$.
      \item[(x)] \MMM We use standard notation for $SBV$-functions. In particular, for $u \in SBV(\Omega)$,  we denote by $\nabla u$ the approximate gradient  and by  $J_u$ the set of jump points of $u$ in $\Omega$. \EEE By $\nu_u$ we denote the normal vector to $J_u$, and by $u^+$, $u^-$ the traces of $u$ on $J_u$ according to the orientation induced by $\nu_u$.  \MMM We say $u \in SBV^p(\Omega)$ if $\nabla u \in L^p(\NNN \Omega, \EEE \R^d)$ and $\mathcal{H}^{d-1}(J_u) < \infty$.
    \item[(xi)]  We use similar notation for  $BD$ and $GSBD$, see \cite{DALMASO_GBD}, \cite{Temam1980FunctionsOB}. \MMM In particular, we define $e(u) = \frac{1}{2}(\nabla u^T + \nabla u)$. \EEE  The main properties of these spaces are recalled in Appendix \ref{section preliminaries}. \EEE 
    \end{itemize}

\section{Setting and main results}
\label{setting and main results}
\MMM In this section we introduce the setting and present our main results. \EEE Let $\Omega \in \mathcal{A}_0$. Consider the family of energies $\mathcal{E}_{\varepsilon}\colon L^0(\Omega;\mathbb{R}^d)\times  \mathcal{M}(\Omega)\times \mathcal{A}(\Omega)\to [0,\infty]$ defined by
\begin{equation}
\label{newenergies}
    \mathcal{E}_{\varepsilon}(u,E,A)\defas \begin{cases}
        \int_{A \setminus {E}}f_{\varepsilon}(x,e(u))\, \mathrm{d}x+\int_{\partial^* E \cap A}g_{\varepsilon}(x,\nu_{ E})\, \mathrm{d}\mathcal{H}^{d-1}\:\: \text{if}\:\: \MMM (u,E) \in \mathcal{W}^{1,p}(A), \EEE \\
        +\infty \:\: \text{otherwise},
    \end{cases}
\end{equation}
where  \MMM
\begin{equation}\label{SBVdefW}
    \mathcal{W}^{1,p}(A)\defas \big\{(u, E) \colon \,   u \in SBV^p(A), \,  E \in \mathcal{P}(A) \text{ with }   J_u \subset \partial^* E\cap A  \:\: \text{and}\:\: u=0 \MMM \text{ on } E \EEE \big\}.
\end{equation}
The energy is determined by the void set  $E$ and the values of the displacement  $u$ on $A \setminus E$. The condition
$u = 0$ on $E$ is for definiteness only, in order to define functions on the entire set $A$. In particular, we use the notation $   \mathcal{W}^{1,p}$ similar to the one of the Sobolev space since, in the case that $E$ is smooth, the corresponding functions in $  \mathcal{W}^{1,p}(A)$ correspond to \NNN $  W^{1,p}(A \setminus \overline{E} ;\R^d)$-functions extended by \EEE $u = 0$ on $E$. \EEE

We assume that the functions $f_{\varepsilon}\colon \mathbb{R}^d \times \mathbb{R}^{d \times d}\to [0,\infty)$ and $g_{\varepsilon}\colon \mathbb{R}^d \times \mathbb{S}^{d-1}\to (0,\infty)$ are Borel measurable and satisfy
\begin{itemize}
    \item[$(f_1)$] $f_{\varepsilon}(x,0)=0$ for every $x \in \mathbb{R}^d$,
    \item[$(f_2)$] $\alpha\vert \frac{\xi+\xi^T}{2}\vert^p\leq f_{\varepsilon}(x,\xi)$ for every $x \in \mathbb{R}^d$ and $\xi \in \mathbb{R}^{d \times d}$,
    \item[$(f_3)$] $f_{\varepsilon}(x,\xi)\leq \beta(1+\vert \frac{\xi+\xi^T}{2}\vert^p)$ for every $x \in \mathbb{R}^d$ and $\xi \in \mathbb{R}^{d \times d}$,
    \item[$(g_1)$] \NNN $g_{\varepsilon}(x,\nu) = g_{\varepsilon}(x,-\nu)$ for every $x \in \mathbb{R}^d$ and $\nu \in \mathbb{S}^{d-1}$, \EEE
     \item[$(g_2)$] $\alpha\leq g_{\varepsilon}(x,\nu)$ for every $x \in \mathbb{R}^d$ and $\nu \in \mathbb{S}^{d-1}$,
     \item[$(g_3)$] $g_{\varepsilon}(x,\nu)\leq \beta$ for every $x \in \mathbb{R}^d$ and $\nu \in \mathbb{S}^{d-1}$,
\end{itemize}
with $\alpha \in (0,1)$, $\beta>\alpha$ \MMM and  $1<p<\infty$. Without further notice, we mention that the bulk integral of \eqref{newenergies} can also be taken over the set $A$, see \eqref{SBVdefW} and $(f_1)$. \EEE 
 
  Our goal is to identify the $\Gamma$-limit of $(\mathcal{E}_\varepsilon)_\eps$ and to show \MMM that \EEE it admits an integral representation. \EEE Before doing so, we introduce the definitions and the notation for the limit functional.  Due to \MMM the growth conditions \EEE $(f_2)$, $(g_2)$, and Theorem \ref{compactness in GSBD^p}, we expect that, whenever a sequence $(u_{\varepsilon},E_{\varepsilon})_{\varepsilon}$ with $\sup_{\varepsilon>0}\mathcal{E}_{\varepsilon}(u_{\varepsilon},E_{\varepsilon},\Omega)<\infty$ converges in $L^0(\Omega;\mathbb{R}^d)\times L^1(\Omega)$ to a pair $(u,E)$, then \EEE $u \in GSBD^p(\Omega)$ and $E \in \MMM\mathcal{P}\EEE(\Omega)$. Motivated also by \cite[Theorem 5.1 and Proposition 5.4]{Crismale_2020},   we define \MMM
\begin{equation}
\label{newspace}
\mathcal{G}^{p}(\Omega)\defas \big\{ (u, E)\colon \, u \in GSBD^p(\Omega), \ E \in \mathcal{P}(\Omega) \text{ with }  u=u\chi_{\NNN E^0}\big\}.
\end{equation} \EEE
For every $x \in \mathbb{R}^d$, $\zeta \in \mathbb{R}^d$, \MMM and \EEE $\xi \in \mathbb{R}^{d \times d}$ we denote by $\overline{l}_{x,\zeta,\xi}$ the affine function $\overline{l}_{x,\zeta,\xi}(y)\defas \zeta+\xi(y-x)$. To simplify the notation, we write ${l}_{\xi}$ instead of $\overline{l}_{0,0,\xi}$. For every $x \in \mathbb{R}^d$, $\zeta^+, \zeta^- \in \mathbb{R}^d$, \MMM and \EEE $\nu \in \mathbb{S}^{d-1}$ we denote by \NNN $\overline{u}_{x,\zeta^-,\zeta^+}^\nu$ \EEE the piecewise constant function defined by
\begin{equation*}
    \overline{u}^\nu_{x,\zeta^-,\zeta^+}(y)\defas \begin{cases}
        \zeta^+ \:\: \text{if}\:\: (y-x)\cdot \nu >0, \\
        \zeta^- \:\: \text{if}\:\: (y-x)\cdot \nu<0.
    \end{cases}
\end{equation*}
To shorten the notation, \MMM we let \EEE ${u}_{x,\zeta}^\nu\defas \overline{u}^\nu_{x,0,\zeta}$.

Given a functional $\mathcal{E}_0 \colon L^0(\Omega;\mathbb{R}^d)\times \mathcal{M}(\Omega)\times \mathcal{A}(\Omega)\to [0,\infty]$, \EEE we define the infimum problem
\begin{equation}
\label{infimumproblem1}
m_{\mathcal{E}_0}(u,E,A)\defas \inf \big\{ \mathcal{E}_0(v,F,A)\colon   \MMM (v,F) \in  \EEE \mathcal{G}^{p}(A)\:\: \text{such that} \:\:  (v,F)=(u,E)  \:\:  \text{near}\:\: \partial A\big\},
\end{equation}
\MMM where `near' refers to a neighborhood of $\partial A$ inside $A$. \EEE
In addition, for every $x\in \Omega$, $\xi \in \MMM \mathbb{R}^{d\times d} \EEE$, $\zeta \in \mathbb{R}^d \setminus \{0\}$, and $\nu \in \mathbb{S}^{d-1}$, we consider the limits
\begin{align}
 \label{def f_0}
f_0(x,\xi) &\defas \limsup\limits_{\rho \to 0}\frac{m_{\mathcal{E}_0}({l}_{\xi},\emptyset, B_{\rho}(x))}{\gamma_d \rho^d},   
\\ \label{def g_0}
   \MMM g_0(x,\nu) \EEE &\defas \limsup\limits_{\rho \to 0} \frac{m_{\mathcal{E}_0}( \MMM 0 , \EEE \Pi^{\nu,-}_{x},B_{\rho}(x))}{\gamma_{d-1} \rho^{d-1}},
\\ \label{def h_0}
  h_0(x,\zeta,\nu) &\defas \limsup\limits_{\rho \to 0} \frac{m_{\mathcal{E}_0}({u}_{x,\zeta}^{\nu},\emptyset,B_{\rho}\MMM (x) \EEE )}{\gamma_{d-1} \rho^{d-1}}.
\end{align}
\MMM Concerning the boundary conditions in  \eqref{def f_0}--\eqref{def h_0}, we observe that the void is only present in $g_0$. The formulation of $f_0$ and $h_0$ features affine or piecewise constant boundary conditions for $u$, as costumary in free-discontinuity problems. \EEE  
We \EEE are ready to present the first main result of this work.
\begin{theorem}[Compactness and Integral representation]
\label{first gamma convergence result}
Let $\Omega \in \mathcal{A}_0$. Let $(f_\varepsilon)_\varepsilon$, $(g_\varepsilon)_\varepsilon$ be \NNN collections \EEE of functions satisfying $(f_1)$--$(f_3)$ \MMM and \EEE $(g_1)$--$(g_3)$, respectively. Let $\mathcal{E}_\varepsilon \colon L^0(\Omega;\mathbb{R}^d)\times \mathcal{M}(\Omega)\times \mathcal{A}(\Omega)\to [0,\infty]$ be the corresponding sequence of functionals given in \eqref{newenergies}. Then, there exist $\mathcal{E}_0\colon L^0(\Omega;\mathbb{R}^d)\times \mathcal{M}(\Omega)\times \mathcal{A}(\Omega)\to [0,\infty]$ and a subsequence (not relabeled) such that
\begin{align}\label{htegammalimit}
    \mathcal{E}_0(\cdot,\cdot,A)=\Gamma-\lim\limits_{\varepsilon \to 0} \mathcal{E}_\varepsilon(\cdot,\cdot,A)\:\: \text{with respect to the}\:\: L^0(\Omega;\mathbb{R}^d)\times L^1(\Omega)\:\: \text{convergence}  
\end{align}
for all $A \in \mathcal{A}(\Omega)$. Moreover, \MMM for all $A \in \mathcal{A}(\Omega)$ and \EEE for every $(u,E) \in  \mathcal{G}^p(\Omega)$ we have that
\begin{equation}\label{the e-0 energy}
    \mathcal{E}_0(u,E,A)=\int_{A}f_0(x,e( u))\, \mathrm{d}x+\int_{\partial^* E\cap A} \MMM g_0(x,\nu_{E}) \EEE \, \mathrm{d}\mathcal{H}^{d-1}+\int_{(J_u \cap E^0)\cap A}h_0(x,[u],\nu_u)\, \mathrm{d}\mathcal{H}^{d-1},
\end{equation}
and $\mathcal{E}_0(u,E,A)=\infty$ otherwise in $L^0(\Omega;\mathbb{R}^d)\times \mathcal{M}(\Omega)$, where $f_0 \colon \Omega \times \mathbb{R}^{d \times d}\to [0,\infty)$, $g_0 \colon \Omega \times \mathbb{S}^{d-1}\to (0,\infty)$, and $h_0 \colon \Omega \times (\mathbb{R}^d \setminus \{0\})  \times \mathbb{S}^{d-1}\to (0,\infty)$ are given by \eqref{def f_0}--\eqref{def h_0}, respectively. 
\end{theorem}
The proof of Theorem \ref{first gamma convergence result}  \MMM  relies on the localization method of $\Gamma$-convergence and  will be given in \EEE  Section \ref{section proof 1}. \MMM  An important ingredient   is a \EEE  fundamental estimate for the energies \eqref{newenergies}, \MMM see \EEE Proposition \ref{fundamentalestimate}.

We continue by introducing definitions that we \MMM will \EEE use for the identification of the $\Gamma$-limit. We consider the following \MMM cell formulas \EEE for the energies \eqref{newenergies}:
\begin{align}
\label{simplified_cellformula_bulk}
    m^{1,p}_{\mathcal{E}_\varepsilon}(u,A)&\defas \inf \{ \mathcal{E}_\varepsilon(v,\emptyset, A)\colon v \in W^{1,p}(A;\mathbb{R}^d) \:\: \text{and}\:\: v=u \:\: \text{near}\:\: \partial A\}, \\ 
\label{simplified_cellformula_voids}
m_{\mathcal{E}_{\varepsilon}}^{\mathrm{voids}}(E,A) &\defas \inf\{ \mathcal{E}_{\varepsilon}( \MMM 0 , \EEE F,A)\colon F \in \NNN \mathcal{P}(A) \EEE \colon \MMM F= E\EEE \:\: \text{near}\:\: \partial A\}.    
\end{align}
(We will use this only for $u = \ell_\xi$, $E = \Pi^{\nu,-}_x$, and $A= B_\rho(x)$.) In addition, we introduce the Lipschitz set $E^{x,\nu}_{\varepsilon}\subset B_{1}(x)$ defined as
\begin{equation}
\label{thin layer}
    E_{\varepsilon}^{x,\nu} \defas \{ y \in B_1(x)\colon \mathrm{dist}(y,\Pi^\nu_x\NNN ) \EEE <\varepsilon\},
\end{equation}
 \MMM and  for \EEE every ball $B_{\rho}(x)$ we define
\begin{align}
\label{set of competitors}
   \MMM  \mathcal{D}^\nu_{\varepsilon}(B_{\rho}(x)) \EEE \defas \Big\{ &(u,E)\in L^0(B_{\rho}(x);\mathbb{R}^d)\times\mathcal{M}(B_{\rho}(x))\colon \exists\:\: S^\pm \in \mathcal{M}(B_{\rho}(x))\:\: \text{such that}\nonumber
    \\&  S^+\cap S^-=\emptyset,\:\: \mathcal{L}^d((S^+\cup S^-\cup E)\triangle B_{\rho}(x))=0,\:\: \mathcal{H}^{d-1} \MMM ( \partial^* S^+\cap \partial^* S^- ) \EEE =0 \nonumber
    \\& u= \MMM e_1 \EEE \chi_{S^+}\:\: \text{in}\:\: B_\rho(x), \:\:    {S^\pm}= \MMM B^{\nu,\pm}_\rho(x)\setminus  E^{x,\nu}_{\varepsilon} \EEE  \:\: \text{near}\:\: \partial B_\rho(x)\Big\},
\end{align}
\MMM where $e_1 =(1,0,\ldots,0)$.   \EEE 
\begin{center}
\begin{figure}[htp]
\begin{tikzpicture}[scale=0.8, line cap=round, line join=round]

\def\R{3}

\draw[line width=1.2pt] (0,0) circle (\R);

\draw[dashed, line width=0.9pt] (-\R,0) -- (\R,0);

\draw[line width=1.2pt]
(-\R,0.6)
-- (-1.4,0.6)
.. controls (-0.9,0.6) and (-0.6,1.4) ..
(-0.1,1.25)
.. controls (0.6,1.1) and (0.9,0.7) ..
(1.4,0.6)
-- (\R,0.6);

\draw[line width=1.2pt]
(-\R,-0.55)
-- (-1.3,-0.55)
.. controls (-0.8,-0.55) and (-0.6,-1.0) ..
(-0.1,-1.15)
.. controls (0.7,-1.4) and (1.0,-0.9) ..
(1.5,-0.6)
-- (\R,-0.6);

\draw[<->, line width=1pt]
(-3.4,0.55) -- (-3.4,-0.55)
node[midway,left] {$2\varepsilon$};

\node at (0.5,0.28) {$E$};
\node at (1.0,1.68) {$S^+$};
\node at (1.0,-1.48) {$S^-$};
\node at (-1,1.9) {$u = e_1$};
\node at (-1,-2.0) {$u = 0$};

\end{tikzpicture}
 \captionof{figure}{An example of \NNN a \EEE competitor in $\mathcal{D}^{e_2}_{\varepsilon}(B_1)$.}
    \label{figure1}
   \end{figure}
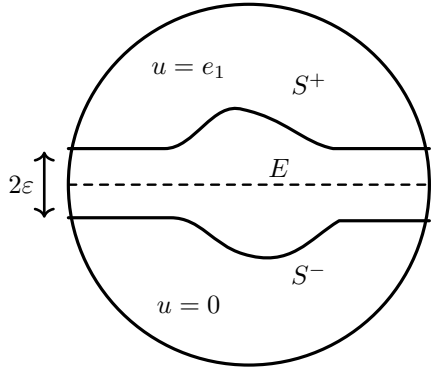
\end{center}

We highlight that the conditions $S^+\cap S^-=\emptyset$ and $\mathcal{H}^{d-1}\MMM (\partial^* S^+\cap \partial^* S^- ) \EEE =0$ means that the void set $E$ must disconnect $S^+$ and $S^-$, \MMM see Figure \ref{figure1}. \EEE

\NNN Based on this definition, \EEE we introduce   the cell formula for the jump part \NNN by \EEE
\begin{align}
\label{simpilfied_cellformula_jump}
m^{\text{jump}}_{\mathcal{E}_{\varepsilon}}(u_{x,e_1}^{\nu},B_{\rho}(x))\QQQ \defas \EEE \inf\big\{& \mathcal{E}_{\varepsilon}(u,E,B_{\rho}(x))\colon  (u,E)\in \mathcal{D}^{\nu}_{\varepsilon}(B_{\rho}(x))\big\}.   
\end{align}

\begin{proposition}
\label{limsupliminf}
Let $\Omega \in \mathcal{A}_0$.
Let $(f_{\varepsilon})_\varepsilon$, $(g_\varepsilon)_\varepsilon$ be collections of functions satisfying $(f_1)$--$(f_3)$ and $(g_1)$--$(g_3)$, respectively. Correspondingly, define $\mathcal{E}_{\varepsilon}$ as in \eqref{newenergies}. Then, by passing to a subsequence (not relabeled) the following holds: 

\noindent 
\MMM  {\rm (i)} \EEE For \QQQ every \EEE $x \in \Omega$ and $\xi \in \mathbb{R}^{d \times d}$ we have
\begin{equation}
\label{f1 = f2}
\MMM \hat{f} \EEE (x,\xi) \defas \limsup\limits_{\rho \to 0^+}\liminf\limits_{\varepsilon \to 0}\frac{m_{\mathcal{E}_{\varepsilon}}^{1,p}({l_{\xi}}, B_{\rho}(x))}{\gamma_d\rho^d}=\limsup\limits_{\rho \to 0^+}\limsup\limits_{\varepsilon \to 0}\frac{m_{\mathcal{E}_{\varepsilon}}^{1,p}({l_{\xi}}, B_{\rho}(x))}{\gamma_d\rho^d},
\end{equation}
and for every $x \in \NNN \Omega \EEE$ and $\nu \in \mathbb{S}^{d-1}$ we have
\begin{equation}
\label{g1 = g2}
\MMM \hat{g} \EEE (x,\nu)\defas \limsup\limits_{\rho \to 0^+}\liminf\limits_{\varepsilon \to 0}\frac{m_{\mathcal{E}_{\varepsilon}}^{\mathrm{voids}}(\Pi^{\nu,-}_x, B_{\rho}(x))}{\gamma_{d-1}\rho^{d-1}}=\limsup\limits_{\rho \to 0^+}\limsup\limits_{\varepsilon \to 0}\frac{m_{\mathcal{E}_{\varepsilon}}^{\mathrm{voids}}(\Pi_x^{\nu,-}, B_{\rho}(x))}{\gamma_{d-1}\rho^{d-1}} .
\end{equation}
\MMM {\rm (ii)}  If $g_{\varepsilon}$ is continuous \NNN for all $\eps>0$, \EEE  then for \MMM every $x \in \NNN \Omega \EEE $ and $\nu \in \mathbb{S}^{d-1}$, \QQQ it also holds that \EEE
\begin{equation}
\label{h1 = h2}
\MMM\hat{h} \EEE(x,\nu)\defas \limsup\limits_{\rho \to 0}\liminf\limits_{\varepsilon \to 0}\frac{m^{\mathrm{jump}}_{\mathcal{E}_{\varepsilon}}(u^\nu_{x,e_1},B_{\rho}(x))}{\gamma_{d-1}\rho^{d-1}}=\limsup\limits_{\rho \to 0}\limsup\limits_{\varepsilon \to 0}\frac{m^{\mathrm{jump}}_{\mathcal{E}_{\varepsilon}}(u^\nu_{x,e_1},B_{\rho}(x))}{\gamma_{d-1}\rho^{d-1}} ,
\end{equation}
and 
\begin{equation}
\label{2g=h}
\MMM \hat{h}(x,\nu)= 2\hat{g}(x,\nu). \EEE 
\end{equation}
\end{proposition}

The proof of Proposition \ref{limsupliminf} \MMM will be given in \EEE Section \ref{section proof of proposition limsupliminf}. \QQQ Let us briefly comment on \eqref{h1 = h2} and~\eqref{2g=h}. \EEE

\begin{remark}\label{h-remark}
Considering a subsequence as in Proposition \ref{limsupliminf} and defining
\begin{align}\label{h-remark3}
h^\prime(x,\nu)\defas \limsup\limits_{\rho \to 0}\liminf\limits_{\varepsilon \to 0}\frac{m^{\mathrm{jump}}_{\mathcal{E}_{\varepsilon}}(u^\nu_{x,e_1},B_{\rho}(x))}{\gamma_{d-1}\rho^{d-1}}, \quad \quad h^{\prime \prime}(x,\nu) \defas \limsup\limits_{\rho \to 0}\limsup\limits_{\varepsilon \to 0}\frac{m^{\mathrm{jump}}_{\mathcal{E}_{\varepsilon}}(u^\nu_{x,e_1},B_{\rho}(x))}{\gamma_{d-1}\rho^{d-1}}, 
\end{align}
the following holds: without assuming   continuity \NNN of $g_\eps$, \EEE  we always  have $h' \ge 2\hat{g}$, i.e., the surface density related to jumps is always at least twice the one related to voids. The continuity condition is only needed to show $ 2\hat{g} \ge h''$ which then gives \eqref{h1 = h2}--\eqref{2g=h} see Section \ref{section proof of proposition limsupliminf} for details. Recall that the relation between $\hat{g}$ and $\hat{h}$ is due to the collapsing of voids, as explained in the introduction. 

Note that in general \eqref{2g=h} is false, as we exemplify below in Remark \ref{counterexample}. Instead, we expect that, along a suitable subsequence, \eqref{h1 = h2} also holds without the continuity condition, but this is beyond our scope. In our setting, \eqref{h1 = h2} simply follows by combining  the two inequality $h' \ge 2\hat{g}$ and $ 2\hat{g} \ge h''$, i.e., there is no need to show  \eqref{h1 = h2} independently. 
\end{remark}
We proceed with the identification of the $\Gamma$-limit.
\begin{theorem}[Identification of the $\Gamma$-limit]
\label{Identification of Gamma-limit}
Let $\Omega \in \mathcal{A}_0$.
Let $(f_{\varepsilon})_\varepsilon$ and \EEE $(g_\varepsilon)_\varepsilon$ be collections of functions satisfying $(f_1)$--$(f_3)$ and $(g_1)$--$(g_3)$ respectively.  Assume that  \eqref{htegammalimit}--\eqref{the e-0 energy}  and   \eqref{f1 = f2}--\eqref{h1 = h2} \NNN  hold for a suitably chosen subsequence, \EEE where  $f_0$, $g_0$, $h_0$ are defined as in  \eqref{def f_0}--\eqref{def h_0}. \EEE Then the following holds \MMM for all  $(u,E) \in \mathcal{G}^p(\Omega)$: \EEE
\begin{itemize}
    \item[{\rm (i)}] \MMM We \EEE  have that
    \begin{equation}
    \label{identification gamma limit volume}
        f_0(x,e(u)(x))=\hat{f}(x,e(u)(x))\:\: \text{for}\:\: \mathcal{L}^d\text{-a.e.}\:\: x \in \Omega,
    \end{equation}
    \item[{\rm (ii)}]  \MMM It holds \EEE that
    \begin{equation}
\label{identification gamma limit voids}
    \MMM  g_0(x,\nu_{ E}(x)) \EEE =\hat{g}(x,\nu_{E}(x))\:\: \text{for}\:\: \mathcal{H}^{d-1}\text{-a.e.}\:\: x \in  \partial^*E\cap \Omega,
\end{equation} 
    \item[{\rm (iii)}] If $d=2$, $p \geq 2$, we have  that
\begin{equation}
\label{h0 pc}
    h_0(x,[u](x),\nu_u  (x) \EEE )=\hat{h}(x,\nu_u (x) \EEE )\:\: \text{for}\:\: \mathcal{H}^{1}\text{-a.e.}\:\: x \in J_u \NNN \cap E^0. \EEE
\end{equation}

\item[{\rm (iv)}] If $d=2$, $p\geq 2$,  and  \NNN all  $g_\eps$ are continuous for $\eps >0$, \EEE  it holds that  
    \begin{equation}
    \label{identification gamma limit jump parts}
        h_0(x,[u](x),\nu_{u}(x))=2\hat{g}(x,\nu_u(x))\:\: \text{for all}\:\: \mathcal{H}^{1}\text{-a.e.}\:\: x \in J_u \cap \NNN E^0. \EEE
    \end{equation}
\end{itemize}
In particular, \NNN under the assumption of  {\rm (iv)}, \EEE  for all $A \in \mathcal{A}(\Omega)$ the functionals  $\mathcal{E}_{\varepsilon} ( \cdot, \cdot, A)$, \EEE defined as in \eqref{newenergies}, $\Gamma$-converge with respect to the $L^0(\Omega;\mathbb{R}^d)\times L^1(\Omega)$-convergence to the functional $\QQQ \hat{\mathcal{E}} (\cdot, \cdot, A) \EEE \colon L^0(\Omega;\mathbb{R}^d)\times \mathcal{M}(\Omega)\to [0,\infty]$, where 
\begin{equation*}
    \hat{\mathcal{E}}(u,E,A)\defas \int_{A\setminus E}\hat{f}(x,\QQQ e(u) \EEE ) \, \mathrm{d}x +\int_{A \cap \partial^* E}\hat{g}(x, \QQQ \nu_{E})\, \mathrm{d}\mathcal{H}^{d-1}  \EEE +\int_{J_u \cap \NNN E^0 \cap A}2\hat{g}(x,\QQQ \nu_u )\, \mathrm{d}\mathcal{H}^{d-1}  
\end{equation*}
for $(u,E) \in \mathcal{G}^p(A)$, and $\hat{\mathcal{E}}(u,E,A)\QQQ \defas \EEE \infty$ otherwise in $L^0( \QQQ \Omega \EEE ;\mathbb{R}^d)\times \mathcal{M}(\QQQ \Omega \EEE)$.
\end{theorem}

\NNN Note that in the statement we \emph{assumed} \eqref{f1 = f2}--\eqref{h1 = h2}. For \eqref{f1 = f2}--\eqref{g1 = g2} this can always be guaranteed by Proposition \ref{limsupliminf}, see Remark \ref{h-remark} for a discussion on \eqref{h1 = h2}. The restriction  to $d =2$, $p \ge 2$ in (iii)-(iv) is due to the approximation result in Lemma \ref{Lemma 5.2} below.  \EEE   The proof of Theorem \ref{Identification of Gamma-limit} will be given in Section \ref{section proof 2}.
\EEE

\begin{remark}
\label{counterexample}
The assumption of continuity \QQQ of \EEE $g_\varepsilon$ in Theorem \ref{Identification of Gamma-limit}{\rm (iv)} is important. Indeed, in general the identity  $h_0=2g_0$ is not valid which we show by \EEE the following example: let $\Omega=(0,1)\times (0,1)$  and \EEE $g_{\varepsilon}(x,\nu)=g(x)$ \MMM be \EEE defined as 
\begin{equation*}
    g(x)\defas \begin{cases}
        1 \:\: \text{if}\:\: x \in \{\frac{1}{2}\}\times (0,1), \\
        2 \:\: \text{otherwise}.
    \end{cases}
\end{equation*}
Assume by contradiction that \MMM $h_0=2g_0$. \EEE Let $\MMM \bar{x} \EEE =(\frac{1}{2},0)$ and  $\MMM \bar{u}  =u_{\bar{x},e_1}^{e_1}\EEE$. Let $(u_{\varepsilon},E_{\varepsilon})_{\varepsilon}$ be a recovery sequence for $(\bar{u},\emptyset)$. \MMM It is elementary to check that, up to some asymptotically vanishing perturbation,  $\partial^* E_\eps \cap \Omega$  consists of two vertical lines, one of them being $\{\frac{1}{2}\}\times (0,1)$. We thus get \EEE  $\MMM \lim_{\eps\to 0} \EEE \mathcal{E}_{\varepsilon}(u_{\varepsilon},E_{\varepsilon},\Omega)=3$. By virtue of the $\Gamma$-$\liminf$ inequality, we have
\begin{align*}
\mathcal{E}_0(\bar{u},\emptyset,\Omega)=3> 2\liminf\limits_{\varepsilon\to 0}\int_{\Pi_{\bar{x}}^{e_1} \cap \Omega} \MMM g_{\varepsilon}(x,e_1) \EEE \, \mathrm{d}\mathcal{H}^{1}\geq 2\int_{\Pi_{\bar{x}}^{e_1} \cap \Omega }g_0(x,e_1)\, \mathrm{d}\mathcal{H}^1=\int_{\Pi_{\bar{x}}^{e_1} \cap \Omega }h_0(x,e_1)\, \mathrm{d}\mathcal{H}^1,
\end{align*}    
which leads to $\mathcal{E}_0(\bar{u},\emptyset,\Omega)>\mathcal{E}_0(\bar{u},\emptyset,\Omega)$, and hence to a contradiction. 
\end{remark}
\MMM As a \NNN first \EEE  application \EEE of Theorem \ref{first gamma convergence result} and Theorem \ref{Identification of Gamma-limit}, \MMM we present \EEE  the following homogenization result, which is stated for $d=2$ and $p\geq 2$ by virtue of Theorem \ref{Identification of Gamma-limit}{\rm (iii)}--{\rm (iv)}.
\begin{theorem}[Homogenization]
\label{homogenization}
Let $d=2$ and $p\geq 2$. Let $f$ and  $g$ satisfy {\rm $(f_1)$}--{\rm $(f_3)$} \MMM and {\rm $(g_1)$}--{\rm $(g_3)$}, \EEE respectively. In addition, let \MMM $g$ be continuous.    Let $f_\varepsilon(x,\xi)\defas f(\frac{x}{\varepsilon},\xi)$, $g_\varepsilon(x,\nu)\defas g(\frac{x}{\varepsilon},\nu)$, and let $\mathcal{E}_\varepsilon$ be the corresponding energy defined as in \eqref{newenergies}. 
Assume that for every $x \in   \mathbb{R}^2 $, $\xi \in \MMM \mathbb{R}^{2\times 2} \EEE $, and $\nu \in \mathbb{S}^{1}$ the limits
\begin{align}
 \label{homogenization volume}
\lim\limits_{r \to \infty}\frac{m^{1,p}_{\mathcal{E}}(l_\xi,Q_r(rx))}{r^2}=: & \ f_{\mathrm{hom}}(\xi),  
\\ \label{homogenization voids}
\lim\limits_{r \to \infty}\frac{m^{\mathrm{voids}}_\mathcal{E}(\Pi^{\nu,-}_{rx},Q_{rx}^\nu(rx))}{r}=: & \ g_{\mathrm{hom}}(\nu)
\end{align}
exist and are independent of $x$, \MMM where $\mathcal{E}$ denotes the functional \eqref{newenergies} with densities $f$ and $g$. \EEE  
Let $\mathcal{E}_{\mathrm{hom}}\colon L^0(\Omega;\mathbb{R}^d)\times \mathcal{M}(\Omega)\times \mathcal{A}(\Omega)\to [0,\infty]$ be defined as
\begin{equation*}
     \mathcal{E}_{\mathrm{hom}}(u,E,A) \QQQ \defas \EEE
   \int_{A \setminus E}f_\mathrm{hom}(\QQQ e(u) \EEE )\, \mathrm{d}x+\int_{\partial^* E \cap A}g_\mathrm{hom}( \QQQ \nu_E \EEE)\,\mathrm{d}\mathcal{H}^1  \EEE+\int_{J_u \cap \NNN E^0 \cap A}2g_\mathrm{hom}(\QQQ \nu_u \EEE )\, \mathrm{d}\mathcal{H}^1 \EEE
\end{equation*}
if \MMM  $(u,E) \in   \mathcal{G}^p(A)$, \EEE  and $\mathcal{E}_{\mathrm{hom}}(u,E,A)\QQQ \defas \EEE \infty$ otherwise. Then, 
\begin{equation*}
    \mathcal{E}_\varepsilon(\cdot,\cdot,A)\:\: \Gamma\text{-converges to}\:\: \mathcal{E}_{\mathrm{hom}}(\cdot,\cdot,A)\:\: \text{with respect to the}\:\: L^0(A;\mathbb{R}^2)\times L^1(A)\text{-convergence}.
\end{equation*}
\end{theorem}
\begin{proof}
The proof is standard and follows the same lines of \cite[Theorem 3.11]{cagnetti2018gammaconvergence}. \MMM Let us just mention the essential idea. The \EEE proof consists in showing that the limits \eqref{homogenization volume} and \eqref{homogenization voids} coincide with $\hat{f}$ and $\hat{g}$, respectively. \MMM One essential point here is that in the definition of  $\hat{f}$ and $\hat{g}$ in \eqref{f1 = f2}--\eqref{g1 = g2} we \emph{do not} pass to a subsequence. \EEE Notice that from this, the fact \MMM that $g$ is continuous, \EEE and \eqref{2g=h}, we also have $h_\mathrm{hom}=2g_\mathrm{hom}$. Finally, we apply Theorem \ref{first gamma convergence result} and Theorem \ref{Identification of Gamma-limit} together with the Urysohn property of $\Gamma$-convergence (see \cite[Proposition 8.3]{DalMaso:93}). \EEE
\end{proof}
\begin{remark}[\MMM Periodic homogenization\EEE]

With a slight adaptation of \cite[Proposition 2.1 and Proposition 2.2]{BraDefVit96}, one can show that \eqref{homogenization volume}--\eqref{homogenization voids} hold, for instance, when $f$ and $g$ are periodic of period $1$ with respect to the coordinates $e_1$ and $e_2$. Similarly, one can show the same \MMM property \EEE under the assumption of stationarity, \MMM in the setting of \EEE stochastic homogenization, see \cite[Proposition 4.1 and Theorem 6.1]{cagnetti2017stochastic}. \EEE

In the case that $g$ is periodic with period $1$, the assumption that $g$ is continuous in the $x$-variable can be dropped, see  Remark \ref{periodic-homo-remark} \EEE below for details.  Note that, in this particular case, our result corresponds to the analog of \cite{solci} in linearized elasticity. 

\end{remark}

 \NNN  As a second application, we give a relaxation result which holds without restriction on the dimension.    \EEE
\begin{theorem}[Relaxation]
\label{relaxation}
Let $\Omega \in \mathcal{A}_0$.
Let $f$ and $g$ be densities satisfying $(f_1)$--$(f_3)$ and $(g_1)$--$(g_3)$ respectively. In addition, assume $g$ to be continuous on $\Omega \times \mathbb{S}^{d-1}$. Let $\mathcal{E}\colon L^0(\Omega;\mathbb{R}^d)\times \mathcal{M}(\Omega)\times \mathcal{A}(\Omega)\to [0,\infty]$ be the functional given in \eqref{newenergies} \NNN with $f$ and $g$. \EEE For all $A \in \mathcal{A}(\Omega)$ and for every $(u,E)\in L^0(A;\mathbb{R}^d)\times \mathcal{M}(\Omega) $ let
\begin{equation}\label{recovii}
    \overline{\mathcal{E}}(u,E,A)\defas \inf\Big \{\liminf\limits_{n \to \infty}\mathcal{E}(u_n,E_n,A)\colon u_n \to u \:\: \text{in}\:\: L^0(A;\mathbb{R}^d)\:\: \text{and}\:\: \chi_{E_n}\to \chi_E \:\: \text{in}\:\: L^1(A)\Big\}.
\end{equation}
Then we have
\begin{equation*}
 \overline{\mathcal{E}}(u,E,A)=\int_A \NNN f^{\rm qc} \EEE (x,e(u))\, \mathrm{d}x+\int_{\partial^* E \cap A}  \NNN g^{\rm BV} \EEE (x,\nu_E)\,\mathrm{d}\mathcal{H}^{d-1}+2\int_{J_u \cap \NNN E^0 \cap A} \NNN g^{\rm BV} \EEE (x,\nu_u)\, \mathrm{d} \mathcal{H}^{d-1} 
\end{equation*}
for every $(u,E)\in \mathcal{G}^p(A)$ and $+\infty$ otherwise, \NNN where  for  $x \in \Omega$,  $\xi \in \mathbb{R}^{d \times d}$, and $\nu \in \mathbb{S}^{d-1}$ we set 
\begin{equation}
\label{relaxationxxxxxxxx000}
   f^{\rm qc}(x,\xi)=\inf\Big\{ \frac{1}{\gamma_d}\int_{B_1}f(x,\nabla u(y))\, \mathrm{d}y \colon u \in W^{1,p}(B_1;\R^d)\:\: \text{with}\:\:u=l_\xi\:\: \,  \text{near} \, \:\: \partial B_1\Big\}  
\end{equation}
and  \EEE
\begin{equation}
\label{relaxationxxxxxxxx}
 g^{BV}(x,\nu)=\inf\Big\{ \frac{1}{\NNN \gamma_{d-1}}\int_{B_1}g(x,\nu_E(y))\, \mathrm{d}\mathcal{H}^{d-1}(y)\colon E \in \mathcal{P}(B_1)\:\: \text{with}\:\:E=\Pi^{\nu,-}_{x}\:\: \text{near}\:\: \partial B_1\Big\}.
\end{equation}

\end{theorem}
\NNN Here,   $f^{\rm qc}$ is the quasiconvex envelope (with respect to
the second variable) of $f$, see \cite[Theorem~9.8]{Dacorogna2008}, and $g^{\rm BV} $ is the $BV$-elliptic envelope of $g$, see \cite[Theorem 3.1]{AmbrosioBraides1990b}. \EEE  The proof of Theorem \ref{relaxation} will be given in Section \ref{section proof 2}.
\EEE

\section{Integral representation and compactness of $\Gamma$-convergence: \MMM Proof of Theorem \ref{first gamma convergence result}\EEE}
\label{section proof 1}

This section is devoted to the proof of Theorem \ref{first gamma convergence result}. The proof is divided \MMM into \EEE two parts. We first \NNN show \EEE an abstract \EEE integral representation for functionals \MMM defined on pairs of function-set, see  Subsections~\ref{sec3.1}--\ref{sec3.2}.  \MMM Afterwards, in Subsection \ref{sec3.3} \EEE  we prove that, up to subsequences, $\mathcal{E}_\varepsilon(\cdot,\cdot,A)$ \MMM given in \eqref{newenergies} \EEE  $\Gamma$-converges to a functional $\mathcal{E}_0(\cdot,\cdot,A)$ satisfying  \MMM all assumptions of the integral representation  result.  \EEE

\subsection{Integral representation for a class of functionals \MMM defined on pairs of function-set\EEE}\label{sec3.1}
Consider \QQQ a functional \EEE $\mathcal{E}_0 \colon L^0(\Omega;\mathbb{R}^d)\times \mathcal{M}(\Omega)\times \MMM \mathcal{B}(\Omega) \EEE \to [0.\infty]$ satisfying the following properties:
\begin{itemize}
\label{set of hypotheses on functionals}
    \item[$\mathrm{(H1)}$] $\mathcal{E}_0(u,E,\cdot)$ is \MMM  a Borel measure  \EEE for every $(u,E) \in \mathcal{G}^{p}(\Omega)$,
    \item [$\mathrm{(H2)}$] $\mathcal{E}_0(\cdot,\cdot,A)$ is lower semicontinuous with respect to the $L^0(A;\mathbb{R}^d)\times L^1(A)$-convergence \MMM for any $A \in \mathcal{A}(\Omega)$, \EEE
    \item[$(\mathrm{H3})$] $\mathcal{E}_0(\cdot,\cdot,A)$ is local for any $A \in \mathcal{A}(\Omega)$ i.e., if $u=v$ and $E=F$ $\mathcal{L}^d$-a.e.\  in $A$, then $\mathcal{E}_0(u,E,A)=\mathcal{E}_0(v,F,A)$,
    \item[$(\mathrm{H4})$] there \NNN exist \EEE $\alpha$, $\beta>0$ with $\alpha<\beta$ such that
    \begin{equation*}
        \alpha\big( \MMM \Vert e(u)\Vert^p_{L^p(A)} \EEE +\mathcal{H}^{d-1}(\partial^* E \cap A)+2\mathcal{H}^{d-1}(J_u  \cap \NNN E^0 \cap A \EEE )\big)\leq \mathcal{E}_0(u,E,A),
    \end{equation*}
    and 
    \begin{equation*}
       \mathcal{E}_0(u,E,A) \leq \beta\big(\mathcal{L}^d(A)+\Vert e(u)\Vert_{L^p(A)}^p+\mathcal{H}^{d-1}(\partial^* E \cap A)+2\mathcal{H}^{d-1}(J_u \cap \NNN E^0 \cap A \EEE )\big),
    \end{equation*}
    for every \MMM  $A \in \mathcal{A}(\Omega)$ and \EEE $(u,E) \in   \mathcal{G}^{p}(A)$.
        \item[$(\mathrm{H5})$] \MMM $\mathcal{E}_0$ is invariant under rigid motions, i.e., $\mathcal{E}_0(u+a \chi_{\Omega \setminus E},E,A)=\mathcal{E}_0(u,E,A)$ for every   $A \in \mathcal{A}(\Omega),$  $(u,E) \in   \mathcal{G}^{p}(A)$, and  affine function $a \colon \mathbb{R}^d \to \mathbb{R}^d$   with $e(a)=0$. \EEE
\end{itemize}

Our goal is to show the following result.   

\begin{theorem}[Integral representation]
\label{Theorem2.1 CriFrieSol}
Let $\Omega \in \mathcal{A}_0$. Suppose that $\mathcal{E}_0$ satisfies $(\mathrm{H1})$--$(\mathrm{H5})$.  Then, \EEE  for each  $A \in \mathcal{A}(\Omega)$  and  $(u,E) \in \mathcal{G}^p(\Omega)$ it holds that 
\begin{equation}
    \mathcal{E}_0(u,E,B)=\int_{\NNN A}f_0(x,e( u))\, \mathrm{d}x+\int_{\partial^* E\cap \NNN A} \MMM g_0(x, \NNN \nu_E) \EEE \, \mathrm{d}\mathcal{H}^{d-1}+ \NNN \int_{J_u \cap E^0\cap A} \EEE h_0(x,[u],\nu_u)\, \mathrm{d}\mathcal{H}^{d-1},
\end{equation}
where $f_0$, $g_0$, and  $h_0$ \MMM are given in \EEE \eqref{def f_0}--\eqref{def h_0}, respectively. 
\end{theorem}

\begin{remark}
\label{REMARK invariance under rigid motions}
Without  details we mention that, in the same spirit of \cite[Theorem 2.1]{SolFriCri20}, one could show an integral representation  when \MMM $(\mathrm{H5})$ \EEE is removed. In that case, $f_0$ can depend on $x$, $u(x)$, and $\nabla u(x)$, while $g_0$ and $h_0$ can both depend on $x$, $u^-(x)$, $u^+(x)$, and $\nu_u(x)$. 
\end{remark}

The proof follows by the global method for relaxation \cite{Bouchitt1998, Bouchitt2002AGM}, in particular \NNN along the lines of \EEE an integral representation result in $GSBD$ \cite{SolFriCri20}. Yet, due to the presence of the void set, some nontrivial adaptations are necessary.   \EEE
We start by stating a fundamental estimate for $\mathcal{E}_0$, \MMM following \EEE  standard arguments (see for example \cite[Proposition 3.1]{BraDefVit96} and \cite[Lemma 3.8]{SolFriCri20}) combined with an adaptation of \cite[Lemma 4.4]{AmbBra90}.
\begin{proposition}[Fundamental estimate for $\mathcal{E}_0$] 
\label{fundestE0}
Let $\eta>0$ and let $A$, $A^\prime$, $B \in \mathcal{A}_0$ with $A^\prime \subset \subset A$. For every functional $\mathcal{E}_0$ satisfying $(\mathrm{H1})$, $(\mathrm{H3})$, and $(\mathrm{H4})$, and for every \MMM $(u,E) \in \mathcal{G}^p(A)$ and  $(v,F) \in \mathcal{G}^p(B)$, \EEE there exists \MMM $(w,D) \MMM \in   \mathcal{G}^{p}(A^\prime \cup B)$ \EEE such that
\begin{align*}
{\rm (i)} & \ \  \mathcal{E}_0(w,D,A^\prime \cup B)\leq (1+\eta) (\mathcal{E}_0(u,E,A)+\mathcal{E}_0(v,F,B))+\MMM M^p \EEE \Vert u -v\Vert^p_{L^p((A \setminus A^\prime)\cap B)}\\& \quad \quad \quad \quad \quad \quad \quad \quad \quad \quad \quad \quad \quad \quad \quad \quad \quad \quad \quad \quad \quad \quad \quad \quad \quad + \ \MMM M \EEE  \Vert \chi_E-\chi_F\Vert_{L^1((A \setminus A^\prime)\cap B)}+\eta,\\
{\rm (ii)}  & \ \   \MMM (w,D)=(u,E) \EEE  \:\: \text{on}\:\: A^\prime \:\: \text{and}\:\: \MMM  (w,D)=(v,F) \EEE \:\: \text{on}\:\: B \setminus  {A}, \EEE   \quad \quad \quad \quad \quad \quad \quad \quad \quad \quad \quad \quad \quad \quad \quad \quad \quad \quad \quad
\end{align*}
where $M=M(A,A^\prime, B, \MMM p, \EEE  \eta)>0$ depends only of $A$, $A^\prime$, $B$, $p$, and $\eta$, but is independent of \NNN $(u,E)$ and  $(v,F)$. \EEE Moreover, if \MMM for \EEE $\varepsilon>0$ and $x \in \mathbb{R}^d$  \MMM we have \EEE $A_{\varepsilon,x}$, $A_{\varepsilon,x}^\prime$ and $B_{\varepsilon,x}\subset \Omega$, then 
\begin{equation*}
M(A_{\varepsilon,x},A_{\varepsilon,x}^\prime, B_{\varepsilon,x}, p ,\eta)= \MMM \varepsilon^{-1} \EEE (A,A^\prime,B,p,\eta).
\end{equation*}
\end{proposition}
\MMM The proof is postponed to the next subsection. The fundamental estimate is crucial in the  \EEE   strategy of \cite{SolFriCri20} (see also \cite{ Bouchitt2002AGM,Bouchitt1998}), as it allows to adjust boundary data in \MMM specific  infimum problems defined \EEE on asymptotically small balls. \MMM We now formulate the main ingredients of the global method. To this end, for each $(u,E) \in  \mathcal{G}^{p}(\Omega)$,  \EEE we set $\mu \defas  \mathcal{L}^d + \mathcal{H}^{d-1}\llcorner \partial^* E + 2\mathcal{H}^{d-1}\llcorner (J_u  \NNN \cap E^0 \EEE )$.

\begin{lemma}
\label{LEMMA 4.1 CRIFRIESOL}
Let $\Omega \in \mathcal{A}_0$. Suppose that $\mathcal{E}_0$ satisfies $(\mathrm{H1})$--$(\mathrm{H4})$ \MMM and \EEE let $(u,E) \in \mathcal{G}^p(\Omega)$. Then for $\mu$-a.e.\ $x \in \Omega $ we have
\begin{equation}
    \lim\limits_{\varepsilon \to 0}\frac{\mathcal{E}_0(u,E,B_{\varepsilon}(x))}{\mu(B_{\varepsilon}(x))}=\lim\limits_{\varepsilon \to 0}\frac{m_{\mathcal{E}_0}(u,E,B_{\varepsilon}(x))}{\mu(B_{\varepsilon}(x))}.
\end{equation}
\end{lemma}
\begin{lemma}
\label{LEMMA 4.2 CRIFRIESOL}
Let $\Omega \in \mathcal{A}_0$. Suppose that $\mathcal{E}_0$ satisfies $(\mathrm{H1})$ and $(\mathrm{H3})$--$(\mathrm{H4})$. Let $(u,E) \in \mathcal{G}^p(\Omega)$. Then for $\mathcal{L}^d$-a.e.\ $x \in \NNN E^0 \EEE$ we have
\begin{equation}
\label{volumeequation}
\lim\limits_{\varepsilon \to 0} \frac{m_{\mathcal{E}_0}(u,E,B_{\varepsilon}(x))}{\gamma_d \varepsilon^d}=\limsup\limits_{\varepsilon \to 0}\frac{m_{\mathcal{E}_0}(\overline{u}_{x}^{\mathrm{bulk}},\emptyset, B_{\varepsilon}(x))}{\gamma_d \varepsilon^d},
\end{equation}
where $\overline{u}^{\mathrm{bulk}}_{x} \defas \QQQ \bar{l}_{x,u(x),\nabla u(x)}\EEE$. 
\end{lemma}

\begin{lemma}
\label{LEMMA 4.3 CRIFRIESOL}
Let $\Omega \in \mathcal{A}_0$. Suppose that $\mathcal{E}_0$ satisfies $(\mathrm{H1})$ and $(\mathrm{H3})$--$(\mathrm{H4})$. Let $(u,E) \in  \mathcal{G}^p(\Omega)$. Then, for $\mathcal{H}^{d-1}$-a.e.\ $x\in J_u \cup (\partial^* E\cap \Omega)$, we have
\begin{equation}
\label{equation infimum jump}
 \lim\limits_{\varepsilon \to 0} \frac{m_{\mathcal{E}_0}(u,E,B_{\varepsilon}(x))}{\gamma_{d-1} \varepsilon^{d-1}}=\limsup\limits_{\varepsilon \to 0} \frac{m_{\mathcal{E}_0}(\MMM \overline{u}_{x}^{\mathrm{surf}}, \EEE \emptyset,B_{\varepsilon}(x))}{\gamma_{d-1} \varepsilon^{d-1}}\:\: \text{if}\:\: x \in J_u \cap E^0,
\end{equation}
and
\begin{equation}
\label{equation infimum void}
\lim\limits_{\varepsilon \to 0} \frac{m_{\mathcal{E}_0}(u,E,B_{\varepsilon}(x))}{\gamma_{d-1} \varepsilon^{d-1}}=\limsup\limits_{\varepsilon \to 0} \frac{m_{\mathcal{E}_0}(\MMM \overline{u}_{x}^{\mathrm{surf}}, \EEE \Pi^{\nu_E(x),-}_x,B_{\varepsilon}(x))}{\gamma_{d-1} \varepsilon^{d-1}}\:\: \text{if}\:\: x \in \partial^* E\cap \Omega, 
\end{equation}
where  \MMM
$\overline{u}_{x}^{\mathrm{surf}}\defas\overline{u}^{\nu_u(x)}_{x,u^-(x),u^+(x)}$ if $x \in J_u \cap E^0$ and $ \overline{u}_{x}^{\mathrm{surf}}\defas\overline{u}_{x,0,u^+(x)}^{\nu_{E}(x)}$ if $x \in \partial^* E\cap \Omega$. \EEE
\end{lemma}

\MMM The proof of  Lemma \ref{LEMMA 4.1 CRIFRIESOL}  follows almost verbatim as in  \cite[Lemma~4.1]{SolFriCri20}, and is thus omitted here. Lemma \ref{LEMMA 4.2 CRIFRIESOL} is similar to \cite[Lemma~4.2]{SolFriCri20}, essentially by  regarding $u$ in the pair $(u,E)$ as a $GSBD^p$-function. For convenience of the reader, we include the proof in  Appendix \ref{otherproof-sec}. Concerning Lemma \ref{LEMMA 4.3 CRIFRIESOL}, compared to \cite{SolFriCri20} there are some nontrivial adaptations due to the presence of the void sets, and we therefore give the proof in all details below in Subsection \ref{sec3.2}. 

\begin{proof}[Proof of Theorem \ref{Theorem2.1 CriFrieSol}]
\EEE An application of \MMM the \EEE Besicovitch Derivation Theorem together with Lemmas \ref{LEMMA 4.1 CRIFRIESOL}--\ref{LEMMA 4.3 CRIFRIESOL}  gives the  result, see the proof of \cite[Theorem 2.1]{SolFriCri20} for  further details.  \MMM The fact that the bulk density is independent of $u$ and $(\nabla u)^T - \nabla u$, and the fact that the surface density $h_0$ only depends on   $[u]$ has been addressed in \cite[Remark 2.2]{SolFriCri20}. 

We include some details for the density $g_0$ which has not been present in \cite{SolFriCri20}. We need to check that  for $\mathcal{H}^{d-1}$-a.e.\ $x \in \partial^*E\cap \Omega$ one has 
\begin{align}\label{eq: to show2}
\frac{\mathrm{d}\mathcal{E}_0(u,E,\cdot)}{\mathrm{d}\mathcal{H}^{d-1}\lfloor_{\partial^* E}}(x) =  g_0\big(x,\nu_E(x)\big).
\end{align}
 By Lemma \ref{LEMMA 4.1 CRIFRIESOL} and the fact that  $\lim_{\rho \to 0} (\gamma_{d-1}\rho^{d-1})^{-1}\mu(B_\rho(x))=1$ for $\mathcal{H}^{d-1}$-a.e.\ $x \in \partial^* E\cap \Omega$  we deduce 
 $$\frac{\mathrm{d}\mathcal{E}_0(u,E,\cdot)}{\mathrm{d}\mathcal{H}^{d-1}\lfloor_{\partial^*E}}(x)  = \lim_{\rho \to 0}\frac{\mathcal{E}_0(u,E,B_\rho(x))}{\mu(B_\rho(x))} =  \lim_{\rho \to 0}\frac{\QQQ m_{\mathcal{E}_0} \MMM (u,E,B_\rho(x))}{\mu(B_\rho(x))}  = \lim_{\rho \to 0}\frac{ \QQQ m_{\mathcal{E}_0} \MMM (u,E,B_\rho(x))}{\gamma_{d-1}\rho^{d-1}} < \infty$$
 for $\mathcal{H}^{d-1}$-a.e.\ $x \in \partial^* E\cap \Omega$. Now, by $(\mathrm{H5})$ and \eqref{equation infimum void} we get 
 $$\frac{\mathrm{d}\mathcal{E}_0(u,E,\cdot)}{\mathrm{d}\mathcal{H}^{d-1}\lfloor_{\partial^*E}}(x)  = \lim_{\rho \to 0}\frac{\QQQ m_{\mathcal{E}_0} \MMM \big(u - \NNN \chi_{\Omega \setminus E} \EEE u^+(x) , E,B_\rho(x)\big)}{\gamma_{d-1}\rho^{d-1}} =  \limsup\limits_{\NNN \rho \to 0} \frac{m_{\mathcal{E}_0}\big(0,\Pi^{\nu_E(x),-}_{x},B_{\NNN \rho}(x)\big)}{\gamma_{d-1} \NNN \rho^{d-1}},  $$
 where we used that \QQQ $\overline{u}_{x}^{\mathrm{surf}}$ \MMM for the shifted function $u - \NNN \chi_{\Omega \setminus E} \EEE u^+(x)$ satisfies $\QQQ \overline{u}_{x}^{\mathrm{surf}} \MMM = 0 $.  Eventually,   \eqref{eq: to show2} follows by recalling the definition in   \eqref{def g_0}. 
\end{proof}

\subsection{\MMM Proof of Proposition \ref{fundestE0} and  Lemma \ref{LEMMA 4.3 CRIFRIESOL}\EEE}\label{sec3.2}

This subsection is devoted to the proof of the fundamental estimate and the identities \eqref{equation infimum jump}--\eqref{equation infimum void}. \EEE

\begin{proof}[Proof of Proposition \ref{fundestE0}]
Let $\eta>0$ and $k \in \mathbb{N}$. Let $A_1$,...,$A_{k+1}$ be open subsets of $\mathbb{R}^d$ with $A^\prime \subset \subset A_1 \subset \subset ... \subset \subset A_{k+1}\subset \subset A$. For $i=1,...,k$ let $\varphi_i \in C_c^\infty(\mathbb{R}^d)$ be cut-off functions such that $\varphi_i=1$ on $A_i$ \MMM and \EEE $\varphi_i=0$ on $\mathbb{R}^d \setminus A_{i+1}$.  Let $(u,E) \in  \mathcal{G}^p(A)$ and  $(v,F) \in \mathcal{G}^p(B)$. We assume \MMM that \EEE $\Vert u-v\Vert_{L^p((A\setminus A^\prime)\cap B)}<\infty$ \MMM as \EEE otherwise the statement is trivial. For the construction of the set $D$, we follow the strategy of \cite[Lemma 4.4]{AmbBra90}. Due to the Fleming Rishel formula (\cite[(4.5.9)]{alma991013725792303131}), for every $i=1,...,k$ we find $t^1_i$, $t^2_i \in (0,1)$ such that \MMM $t^2_i > \frac{3}{4}$, $t^1_i<\frac{1}{4}$, \EEE $\{ x \in A\colon \varphi_i(x)<t^l_i\}$ is a set of finite perimeter \MMM for $l \in \{1,2\}$, \EEE 
\begin{equation}
\label{fleming rishel 0000}
   \mathcal{H}^{d-1}\big((\partial^* E\cup J_u) \cap \{ x \in A\colon \varphi_i(x)=t^l_i\}\big)=\mathcal{H}^{d-1}\big((\partial^* F\cup J_v) \cap \{ x \in A \MMM\cap B \EEE \colon \varphi_i(x)=t^l_i\}\big)=0,
\end{equation}
and 
\begin{align}
\label{fleming rishel 2}
    \int_{(A_{i+1}\setminus A_i)\cap B\cap \partial^*\{\varphi_i<t^l_i\}}\vert \chi_{E}-\chi_F\vert\, \mathrm{d}\mathcal{H}^{d-1}&\leq \MMM 4 \EEE \int_{(A_{i+1}\setminus A_i)\cap B}\vert \chi_E -\chi_F\vert \vert \nabla \varphi_i \vert\, \mathrm{d}x \nonumber
    \\& \vphantom{\int_{A}} \leq \MMM 4 \EEE \Vert \nabla \varphi_i \Vert_{L^\infty(A)}\Vert \chi_E -\chi_F\Vert_{L^1((A \setminus A^\prime)\cap B)}
\end{align}
\MMM for   $l \in \{1,2\}$. \EEE \QQQ Furthermore, we may assume that $\mathcal{L}^{d} (\{ \varphi_{i} = t^{l}_{i} \} ) = 0$ for $l \in \{1, 2\}$. \EEE
We define $D_i$ as  
\begin{equation}\label{deffi1}
    {D_i}\QQQ \defas \EEE \begin{cases}
        {E} \:\: \:\:\:\: \quad \:\text{if}\:\: \varphi_i > t^2_i, \\
        E\cup F \:\:\:\text{if}\:\: \varphi_i\in [t^1_i,t^2_i],\\
        F \:\: \:\:\:\: \quad \:\text{if}\:\: \varphi_i < t^1_i.
    \end{cases}
\end{equation}
\NNN We define \EEE $w_i \MMM  \colon A' \cup B \to \R^d \EEE$ as \NNN  $w_i(x)  \defas  0$ if $x \in D_i$, and for $x \in (A^\prime \cup B) \setminus {D_i}$ we set \EEE
\begin{equation}\label{deffi2}
  w_i(x)\QQQ \defas \EEE \begin{cases}
   \MMM {u} \EEE (x) & \text{if}\:\: \varphi_i(x)>t^2_i, \\
       \psi_i(x)\, \MMM {u} \EEE (x)+(1-\psi_i(x))\, \MMM {v} \EEE (x) &  \text{if}\:\: \varphi_i(x)\in [t^1_i,t_i^2], \\
      \MMM {v} \EEE (x) & \text{if}\:\: \varphi_i(x)< t^1_i, \\
    \end{cases}    
\end{equation}
where $\psi_i(x)\defas \frac{\varphi_i(x)-t^1_i}{t^2_i-t^1_i}$. We observe that $\mathcal{H}^{d-1}(J_{w_i}\cap \{\varphi_i=t^l_i\})=0$ \MMM for $l \in \lbrace 1, 2 \rbrace$ by \eqref{fleming rishel 0000}. Moreover, \EEE  $(w_i, D_i)=(u, E)$ on $A_i$ (and hence on $A^\prime$) and $(w_i,D_i)=(v,F)$ on $B \setminus A_{i+1}$ (and hence on $B \setminus \NNN {A} \EEE $). To simplify the notation, we set $S_i\defas (A_{i+1}\setminus A_i)\cap B$.
Due to $(\mathrm{H}1)$, $(\mathrm{H}4)$, and \eqref{fleming rishel 0000}, we have
\begin{align}
\label{fund est at 0 level eq1}
    \mathcal{E}_0(w_i,D_i,A^\prime \cup B)\leq & \  \mathcal{E}_0(w_i,D_i,\{\varphi_i>t^2_i\})+\mathcal{E}_0(w_i,D_i,\{\varphi_i<t^1_i\})+\mathcal{E}_0(w_i,D_i,\{\varphi_i\in (t^1_i,t^2_i)\}) \nonumber
    \\& +\beta\mathcal{H}^{d-1}(\partial^* D_i \cap \{ \varphi_i=t^1_i\})+\beta\mathcal{H}^{d-1}(\partial^* D_i \cap \{ \varphi_i=t^2_i\}),
\end{align}
where we also used $\mathcal{L}^d(\{\varphi_i=t^l_i\})+\Vert e(w_i)\Vert_{L^p(\{\varphi_i=t^l_i\})}=0$ and $\mathcal{H}^{d-1}(J_{w_i}\cap \{ \varphi_i=t^l_i\})=0$ for \MMM $l \in \{1,2\}$. \EEE Now we observe that
\begin{equation}
    \mathcal{H}^{d-1}(\partial^* D_i \cap \{ \varphi_i =t^l_i\})=\int_{S_i \cap \partial^*\{\varphi_i<t^l_i\}}\vert \chi_{E}-\chi_F\vert \, \mathrm{d}\mathcal{H}^{d-1}
\end{equation}
for    \MMM $l \in \{1,2\}$. \EEE
In addition, because of $(\mathrm{H}1)$, $(\mathrm{H}3)$ \MMM and \eqref{deffi1}--\eqref{deffi2}, \EEE it holds
\begin{align}
\mathcal{E}_0(w_i,D_i,\{\varphi_i>t^2_i\}) & =\mathcal{E}_0(u,E,\{\varphi_i>t^2_i\})\leq \mathcal{E}_0(u,E,A), \\
 \mathcal{E}_0(w_i,D_i,\{\varphi_i<t^1_i\})&=\mathcal{E}_0(v,F,\{\varphi_i<t^1_i\})\leq \mathcal{E}_0(v,F,B).  
\end{align}
By virtue of $(\mathrm{H}4)$, one also has
\begin{align}
\label{equation to estimate at 0 level}
    \mathcal{E}_0(w_i,D_i,\{\varphi_i \in (t^1_i,t_i^2)\})&\leq \beta (\mathcal{L}^d(S_i)+\Vert e(w_i)\Vert^p_{L^p(\{\varphi_i \in (t^1_i,t_i^2)\})})+2\beta \mathcal{H}^{d-1}(J_{w_i}\cap \{\varphi_i \in (t^1_i,t_i^2)\}) \nonumber
    \\& \ \ \  +\beta \mathcal{H}^{d-1}(\partial^* D_i\cap \{\varphi_i \in (t^1_i,t_i^2)\}).
\end{align}
Now, we estimate the terms \MMM on \EEE the right-hand side of \eqref{equation to estimate at 0 level}. \QQQ Since \EEE $t^2_i-t^1_i >\frac{1}{2}$, \MMM  we have \EEE
\begin{equation}
\Vert e(w_i)\Vert^p_{L^p(\{\varphi_i \in (t^1_i,t_i^2)\})}\leq c_p \big( \MMM \Vert \nabla \varphi_i\Vert^p_{L^{\infty}(S_i)} \EEE \Vert u-v\Vert_{L^p(S_i)}^p+\Vert e(u)\Vert^p_{L^p(S_i)}+\Vert e(v)\Vert^p_{L^p(S_i)}\big),    
\end{equation}
where $c_p>0$ denotes a constant depending only on $p$. \QQQ Moreover, \EEE
\begin{align}
 \mathcal{H}^{d-1}(J_{w_i}\cap \{\varphi_i \in (t^1_i,t_i^2)\}) &  \leq \mathcal{H}^{d-1}(J_u \cap S_i)+\mathcal{H}^{d-1}(J_v\cap S_i),\\
    \mathcal{H}^{d-1}(\partial^* D_i \cap \{\varphi_i \in (t^1_i,t_i^2)\}) & \leq \mathcal{H}^{d-1}(\partial^* E \cap S_i)+\mathcal{H}^{d-1}(\partial^* F \cap S_i).
\end{align}
For every Borel set $O \subset A^\prime \cup B$, let us define 
\begin{align}
\label{error function 2}
\mathrm{err}(u,v,E,F,O)\defas \mathcal{L}^d(O)+c_p(\Vert e(u)\Vert^p_{L^p(O)}+\Vert e(v)&\Vert^p_{L^p(O)})+2\big(\mathcal{H}^{d-1}(J_u\cap O)+\mathcal{H}^{d-1}(J_v\cap O)\big) \nonumber
    \\& +\mathcal{H}^{d-1}(\partial^* E \cap O)+\mathcal{H}^{d-1}(\partial^* F \cap O).
\end{align}
Combining \eqref{fleming rishel 2} with \eqref{fund est at 0 level eq1}--\eqref{error function 2} we get
\begin{align}
\label{right candidate}
\mathcal{E}_0(w_i,D_i,A^\prime \cup B) \leq & \ \mathcal{E}_0(u,E,A)+\mathcal{E}_0(v,F,B)+\beta c_p\MMM \Vert \nabla \varphi_i\Vert^p_{L^\infty(S_i)} \EEE \Vert u-v\Vert^p_{L^p(S_i)} \nonumber
\\& + \MMM 8 \EEE \beta \Vert \nabla  \MMM \varphi_i \EEE \Vert_{L^\infty(\NNN A)}\Vert \chi_{E}-\chi_F\Vert_{L^1(\NNN (A\setminus A') \cap B)} +\beta\, \mathrm{err}(u,v,E,F,S_i).
\end{align}
\MMM  Then, \EEE we can find an $i_0 \in \{1,...,k\}$ such that
\begin{align}
\label{final eq fund est at level 0}
    \mathcal{E}_0(w_{i_0},D_{i_0},A^\prime \cup B) \leq & \ \mathcal{E}_0(u,E,A)+\mathcal{E}_0(v,F,B)+ \beta c_p \NNN  \max_{i=1,\ldots,k}\ \EEE \Vert \nabla \varphi_i\Vert^p_{L^\infty } \EEE\Vert u-v\Vert^p_{L^p((A\setminus A^\prime)\cap B)} \nonumber
    \\& + \MMM 8    \beta  \NNN \max_{i=1,\ldots,k} \EEE \Vert \nabla \varphi_i\Vert_{L^\infty } \EEE \Vert \chi_E-\chi_F\Vert_{L^1((A\setminus A^\prime)\cap B)}+\frac{\NNN \beta}{k}\sum^k_{i=1}\mathrm{err}(u,v,E,F,S_i).
\end{align}
\NNN Using that \EEE  $(S_i)^k_{i=1}$ are disjoint and $\bigcup^k_{i=1} S_i \subset (A \setminus A^\prime)\cap B$ \NNN and \EEE  $(\mathrm{H}4)$, we have
\begin{align}
 \label{controll error at level 0}
 \frac{1}{k}\sum^k_{i=1}\mathrm{err}(u,v,E,F,S_i)&\leq  \frac{1}{k}\mathrm{err}(u,v,E,F,(A \setminus A^\prime)\cap B) \nonumber
 \\& \leq \frac{1}{k}\mathcal{L}^d((A\setminus A^\prime)\cap B)+\frac{(c_p+1)}{k \alpha}(\mathcal{E}_0(u,E,A)+\mathcal{E}_0(v,F,B)).    
\end{align}
Hence, we conclude the proof by choosing $k$ large enough so that $\frac{1}{k}\mathcal{L}^d((A\setminus A^\prime)\cap B)+\frac{(c_p+1)}{k \alpha} \leq \NNN \eta/ \beta \EEE$ and by setting $(w,D)=(w_{i_0},D_{i_0})$, $M=  \MMM (8\beta + (\beta c_p)^{1/p})  \max_{i=1,\ldots,k}\Vert \nabla \varphi_i\Vert_{L^\infty} \EEE $ and combining \eqref{final eq fund est at level 0}--\eqref{controll error at level 0}. \MMM Note that by \NNN the construction in \eqref{deffi1}--\eqref{deffi2} \EEE we clearly have $(D,w) \in   \mathcal{G}^{p}(A^\prime \cup B)$. \EEE
\end{proof}

\MMM We proceed with the  proof of Lemma \ref{LEMMA 4.3 CRIFRIESOL}. As a preparation, we state \EEE an adaptation of \cite[Lemma~6.1]{SolFriCri20}. \NNN Recall the definition of $\overline{u}_{x}^\mathrm{surf}$ below \eqref{equation infimum void}. \EEE
\begin{lemma}[Blow up at jump points]
\label{LEMMA 6.1 CRI FRIE SOL} 
Let $(u,E) \in \mathcal{G}^p(\Omega)$ and $\theta \in (0,1)$. For  $\mathcal{H}^{d-1}$-a.e.\ $\MMM x\in  (J_u \cup \partial^* E)\cap \Omega \EEE $ there exist families \MMM  $(u_\varepsilon,E_{\varepsilon}) \in \mathcal{G}^p(B_\varepsilon(x))$ \EEE  such that
\begin{subequations}
\label{blowupatjumppoints}
\begin{align}
    & \ u_\varepsilon =u \:\:  \text{near}\:\: \partial B_\varepsilon(x), \quad \lim\limits_{\varepsilon \to 0} \varepsilon^{-d}\mathcal{L}^d(\{ u_\varepsilon \neq u\})=0, \label{blowupatjumppoints-1} \\
    & \lim\limits_{\varepsilon \to 0}\varepsilon^{-(d-1+p)}\int_{B_{(1-\theta)\varepsilon}(x)}\vert u_\varepsilon(y)-\overline{u}_{x}^\mathrm{surf}(y)\vert^p \, \mathrm{d}y=0, \label{blowupatjumppoints-2}
    \\ & \lim\limits_{\varepsilon \to 0}\varepsilon^{-(d-1)}\mathcal{H}^{d-1}(J_{u_\varepsilon}\cap F_{\varepsilon,x})=\mathcal{H}^{d-1}\big( \MMM J_{\overline{u}_{x}^{\mathrm{surf}}} \EEE \cap F\big) \:\: \text{for all Borel sets}\:\: F \subset B_1(x), \label{blowupatjumppoints-3}
    \\ & \lim\limits_{\varepsilon \to 0} \varepsilon^{-(d-1)}\int_{B_\varepsilon(x)}\vert e(u_\varepsilon)(y)\vert^p \, \mathrm{d}y=0, \label{blowupatjumppoints-4}
    \\ & E_\varepsilon= E \:\: \text{near}\:\: \partial B_\varepsilon(x),  \quad \MMM \lim_{\varepsilon \to 0} \varepsilon^{-d}\mathcal{L}^d(\{ \chi_{E_\varepsilon} \neq \chi_E\})=0, \EEE    \label{blowupatjumppoints-5}
    \\ & E_{\varepsilon}= \MMM \overline{E}^{\rm surf}_x \EEE \:\: \text{on}\:\: B_{(1-\theta)\varepsilon}(x), \label{blowupatjumppoints-6}
    \\ & \lim\limits_{\varepsilon \to 0}\varepsilon^{-(d-1)}\mathcal{H}^{d-1}(\partial^* E_\varepsilon\cap F_{\varepsilon,x})=\mathcal{H}^{d-1}\big(\MMM  \partial \overline{E}^{\rm surf}_x \EEE \cap F\big) \:\: \text{for all Borel sets}\:\: F \subset B_1(x), \label{blowupatjumppoints-7}
\end{align}
\end{subequations}
\MMM
where  $ \overline{E}^{\rm surf}_x = \Pi^{\nu_E(x),-}_x$ if $x \in \partial^* E\cap \Omega$ and $ \overline{E}^{\rm surf}_x =\emptyset$ if $x \in J_u \NNN \cap E^0 \EEE $.  \EEE
\end{lemma}
\begin{proof}
We provide  a sketch of the  proof \MMM highlighting \EEE the adaptations needed. \MMM The construction of the sequence of functions \EEE is given in   \cite[Lemma 6.1]{SolFriCri20}: we  find a sequence $\MMM \hat{u}_{\varepsilon}  \EEE \in GSBD^p(B_{\varepsilon}(x))$  satisfying \eqref{blowupatjumppoints-1}--\eqref{blowupatjumppoints-4} which is of the  form
\begin{equation}\label{hatu}
   \hat{u}_{\varepsilon}(y)=\begin{cases}
        \hat{v}^+_{\varepsilon}(\QQQ y \EEE ) &  \text{if}\:\: y \in B^{\nu_0(x),+}_{s_{\varepsilon}}(x),\\
        \hat{v}^-_{\varepsilon}(y) &\text{if}\:\: y \in B^{\nu_0(x),-}_{s_{\varepsilon}}(x), \\
        u(y) &  \text{otherwise in}\:\: B_{\varepsilon}(x).
    \end{cases}
\end{equation}
Here, $\nu_0(x)=\nu_E(x)$ if $x \in \partial^* E\cap \Omega$ and  $\nu_0(x)=\nu_u(x)$ if $x \in J_u \NNN \cap E^0 \EEE$,  $\hat{v}^\pm_{\varepsilon}$ are suitable Sobolev replacements of $u$, \MMM and \EEE $s_{\varepsilon}\in (1-\theta,1-\frac{\theta}{2})\varepsilon$ \MMM is chosen by Fubini's theorem such that 
\begin{align}\label{fub1}
\lim_{\eps \to 0}   \eps^{-(d-1)}\mathcal{L}^d \big(  \big(B_\eps^{\Gamma,\pm}(x) \triangle B_\eps^{\nu_0(x),\pm}(x) \big) \cap \partial B_{s_\eps}(x) \big) = 0, 
\end{align}
where $B_\eps^{\Gamma,\pm}$  denotes a suitable partition of $B_\eps(x) $ satisfying $\lim_{\eps \to 0}  \eps^{-d}\mathcal{L}^d (  (B_\eps^{\Gamma,\pm}(x) \triangle B_\eps^{\nu_0(x),\pm}(x)  ) = 0$.

In the present setting, we additionally \EEE   need to construct \NNN the \EEE sets $E_\eps$. \MMM To this end, it is not restrictive to assume that    for $\mathcal{H}^{d-1}$-a.e.\ $x\in  \partial^* E \cap \Omega$  it holds that 
\begin{align}\label{88x}
\lim_{\eps \to 0}  \eps^{-d}\mathcal{L}^d \big(  \NNN (E \cap B_\eps(x))  \triangle B_\eps^{\nu_0(x),-}(x) \EEE  \big) = 0,
\end{align} and  for  $\mathcal{H}^{d-1}$-a.e.\ $x\in  J_u   \NNN \cap E^0 \EEE $  it holds that $\lim_{\eps \to 0}  \eps^{-d}\mathcal{L}^d (  E \NNN \cap \EEE B_\eps(x)  ) = 0$. In these cases,  we can select $s_{\varepsilon}\in (1-\theta,1-\frac{\theta}{2})\varepsilon$ above such that not only \eqref{fub1} holds, but also 
\begin{align}\label{fub2}
\lim_{\eps \to 0}   \eps^{-(d-1)} \NNN \mathcal{H}^{d-1} \EEE \big(  \big (E \triangle B_\eps^{\nu_0(x),-}(x)  \big) \cap \partial B_{s_\eps}(x) \big) = 0 \quad \text{ for }  x \in   \partial^*E\cap \Omega
\end{align}
or    
\begin{align}\label{fub3}
\lim_{\eps \to 0}   \eps^{-(d-1)} \NNN \mathcal{H}^{d-1} \EEE \big(  \NNN  E \EEE  \cap \partial B_{s_\eps}(x) \big) = 0  \quad \text{ for }  x \in J_u  \NNN \cap E^0, \EEE
\end{align}
respectively, are satisfied. 

\MMM We now treat the cases $x \in J_u \NNN \cap E^0\EEE$ and  $x \in \partial^* E\cap \Omega$ separately.  \EEE If $x\in J_u \NNN \cap E^0\EEE$, we \MMM set $u_\eps = \hat{u}_\eps$ and $E_{\varepsilon}= E \cap (B_{\varepsilon}(x)\setminus B_{s_{\varepsilon}}(x))$. We can check that  $(u_{\varepsilon},E_\eps) \in \mathcal{G}^p(B_{\varepsilon}(x))$ and that  \eqref{blowupatjumppoints-1}--\eqref{blowupatjumppoints-7} are satisfied, where we use that 
$\hat{u}_\eps$ satisfies \eqref{blowupatjumppoints-1}--\eqref{blowupatjumppoints-4}, for \eqref{blowupatjumppoints-5} we use $\lim_{\eps\to0}\eps^{-d}\mathcal{L}^d(E \cap B_\varepsilon(x) )=0$, and \eqref{blowupatjumppoints-7} follows from \eqref{fub3} and the fact that without restriction  $x\in J_u \NNN \cap E^0\EEE$ can be chosen such that $\lim_{\eps\to0}\eps^{-(d-1)}\mathcal{H}^{d-1}(\partial^* E \cap B_\varepsilon(x) )=0$.  

If $x \in \partial^* E\cap \Omega$, we define $u_\eps$ as in \eqref{hatu} with $0$ in place of  $\hat{v}^-_{\varepsilon}$ \EEE
and $E_{\varepsilon}= [E \cap (B_{\varepsilon}(x) \setminus B_{s_{\varepsilon}}(x))]\cup [\Pi^{\nu_E(x),-}_x\cap B_{s_{\varepsilon}}(x)]$. \MMM As before we can check that $(u_{\varepsilon},E_\eps) \in \mathcal{G}^p(B_{\varepsilon}(x))$ and that  \eqref{blowupatjumppoints-1}--\eqref{blowupatjumppoints-7} hold. \NNN Indeed, the proof of the properties \eqref{blowupatjumppoints-2}, \eqref{blowupatjumppoints-4}, \eqref{blowupatjumppoints-5}, and \eqref{blowupatjumppoints-6}   is analogous to the case $x \in J_u \cap E^0$, where in place of $\lim_{\eps \to 0}  \eps^{-d}\mathcal{L}^d (  E  \cap   B_\eps(x)  ) = 0$  we use \eqref{88x}. Due to the different definition of $u_\eps$ on $\Pi^{\nu_E(x),-}_x\cap B_{s_{\varepsilon}}(x)$ compared to the case $x \in J_u \cap E^0$, for \eqref{blowupatjumppoints-1} and \eqref{blowupatjumppoints-3} we additionally need to show that  
\begin{equation*}
    \lim_{\varepsilon \to 0}\frac{1}{\varepsilon^d}\mathcal{L}^d\big( (E \triangle \Pi^{\nu_E(x),-}_x) \cap B_{s_{\varepsilon}}(x) \big)=0, \quad     \lim_{\varepsilon \to 0}\frac{1}{\varepsilon^{d-1}}\mathcal{H}^{d-1}\big(  \big (E \triangle B_\eps^{\nu_0(x),-}(x)  \big) \cap \partial B_{s_\eps}(x) \big)=0,
\end{equation*}
where we recall that $u = 0$ on $E$. This is indeed true by  \eqref{88x} and \eqref{fub2}. \EEE   Finally, arguing like in \cite[(6.1)(iii)]{SolFriCri20}, \MMM in particular taking \eqref{fub2} into account,  one can show \eqref{blowupatjumppoints-7}. 
\end{proof}
We   now come to the proof of Lemma \ref{LEMMA 4.3 CRIFRIESOL}. The proof follows the same strategy of \cite[Lemma 4.3]{SolFriCri20}, up to taking into account \MMM the   void set. In particular, we use \EEE our fundamental estimate Proposition \ref{fundestE0}. 
\begin{proof}[Proof of Lemma \ref{LEMMA 4.3 CRIFRIESOL}]
We fix $x\in (J_u  \MMM \cup \partial^* E)\cap \Omega \EEE$ such that the statement of Lemma \ref{LEMMA 6.1 CRI FRIE SOL} holds at $x$ and \MMM we have \EEE $\lim_{\varepsilon \to 0}\varepsilon^{-1-d}\mu(B_{\varepsilon}(x))=\gamma_{d-1}$ if $x \in \partial^* E\cap \Omega$ and $\lim_{\varepsilon \to 0}\varepsilon^{-1-d}\mu(B_{\varepsilon}(x))=2\gamma_{d-1}$ if $x \in J_u \cap E^0 $. This is possible for $\mathcal{H}^{d-1}$-a.e.\ $x \in (J_u \cup \partial^* E)\cap \Omega$. Then also $\lim_{\varepsilon\to 0}\varepsilon^{-d-1}m_{\mathcal{E}_0}(u,E,B_{\varepsilon}(x))\in \mathbb{R}$ \MMM exists, \EEE see Lemma~\ref{LEMMA 4.1 CRIFRIESOL}.   

\emph{Step 1 (Inequality ``$\leq$" in \eqref{equation infimum jump}--\eqref{equation infimum void}):}  We fix $\eta>0$, $\theta>0$, and consider $(z_{\varepsilon },F_{\varepsilon})\in  \mathcal{G}^p(B_{(1-3\theta)\varepsilon}(x))$ with \MMM $(z_{\varepsilon}, F_{\varepsilon})=(\overline{u}_{x}^{\mathrm{surf}}, \overline{E}^{\rm surf}_x)$ \EEE near $\partial B_{(1-3\theta)\varepsilon}(x)$ and
\begin{equation}
\label{infimum problem blow up jump points}
    \mathcal{E}_0(z_{\varepsilon}, F_{\varepsilon},B_{(1-3\theta)\varepsilon}(x))\leq m_{\mathcal{E}_0}\big(\overline{u}_{x}^\mathrm{surf},\overline{E}^{\rm surf}_x,B_{(1-3\theta)\varepsilon}(x)\big)+\varepsilon^d.
\end{equation}
We extend $(z_{\varepsilon},F_{\varepsilon})$ \MMM onto \EEE $B_{\varepsilon}(x)$ by setting $(z_{\varepsilon},F_{\varepsilon})=(\overline{u}_{x}^\mathrm{surf},\overline{E}^{\rm surf}_x)$ on $B_{\varepsilon}(x)\setminus B_{(1-3\theta)\varepsilon}(x)$. Let $(u_{\varepsilon},E_{\varepsilon})_{\varepsilon}$ be the family \MMM given by \EEE Lemma \ref{LEMMA 6.1 CRI FRIE SOL}. \MMM We \EEE apply Proposition \ref{fundestE0} on $(z_{\varepsilon},F_{\varepsilon})$ (in place of $(u,E)$) and  on $(u_{\varepsilon},E_{\varepsilon})$ (in place of $(v,F)$), for $\eta$ fixed above and the sets 
\begin{equation}\label{sets fund est}
    A^\prime=B_{1-2\theta}(x),\:\: A=B_{1-\theta}(x), \:\: B= \QQQ B_{1} \EEE (x) \setminus \overline{B_{1-4\theta}(x)} 
\end{equation}
\MMM or, more precisely, on their rescalings   $A_{\varepsilon,x}$, $A_{\varepsilon,x}^\prime$, and $B_{\varepsilon,x}$. \EEE  In this way, we find   $(w_{\varepsilon}, D_\eps)\in \mathcal{G}^p(B_{\varepsilon}(x))$ such that $(w_{\varepsilon},{D_{\varepsilon}})=(u_{\varepsilon},{E_{\varepsilon}})$ on $B_{\varepsilon}(x)\setminus B_{(1-\theta)\varepsilon}(x)$ and
\begin{align}
\label{fund est blow up jump points}
    \mathcal{E}_0(w_{\varepsilon},D_{\varepsilon},B_{\varepsilon}(x))\leq ( 1 & +\eta)\big(\mathcal{E}_0(z_{\varepsilon},F_{\varepsilon},A_{\varepsilon,x})+\mathcal{E}_0(u_{\varepsilon},E_{\varepsilon}, B_{\varepsilon,x})\big) 
    \\
    &
    \notag +\frac{M^p}{\varepsilon^p}\Vert z_{\varepsilon}-u_{\varepsilon}\Vert^p_{L^p(A_{\varepsilon,x}\setminus A_{\varepsilon,x}')}+\MMM \eps^d \EEE \eta,
\end{align}
where we used $F_{\varepsilon}= E_{\varepsilon}= \MMM \overline{E}^{\rm surf}_x \EEE $ on $A_{\varepsilon,x}\setminus A_{\varepsilon,x}'$ because of \eqref{blowupatjumppoints-6}. \NNN Note \EEE that $M>0$ is independent of $\varepsilon$ and $\theta$, \NNN and \EEE that we have $w_{\varepsilon}=u_{\varepsilon}=u$ in a neighborhood of $\partial B_{\varepsilon}(x)$ by \eqref{blowupatjumppoints-1}. By \eqref{blowupatjumppoints-2} and the fact that $z_{\varepsilon}=\overline{u}^\mathrm{surf}_{x}$ outside $B_{(1-3\theta)\varepsilon}(x)$ we find
\begin{equation*}
    \lim\limits_{\varepsilon \to 0}\varepsilon^{-(d-1+p)}\Vert z_{\varepsilon}-u_{\varepsilon}\Vert^p_{L^p(A_{\varepsilon,x}\setminus A_{\varepsilon,x}')}=\lim\limits_{\varepsilon \to 0}\varepsilon^{-(d-1+p)}\Vert u_{\varepsilon}-\overline{u}_{x}^\mathrm{surf}\Vert_{L^p(A_{\varepsilon,x}\setminus A_{\varepsilon,x}')}^p=0.
\end{equation*}
Inserting this in \eqref{fund est blow up jump points}, we find that there exists a sequence $(\rho_{\varepsilon})_{\varepsilon}\subset (0,\infty)$ with $\rho_{\varepsilon}\to 0$  such that
\begin{align}
\label{simplified inequality in fund est}
\mathcal{E}_{\varepsilon}(w_{\varepsilon},D_{\varepsilon},B_{\varepsilon}(x))\leq &(1+\eta)\big(\mathcal{E}_0(z_{\varepsilon},F_{\varepsilon},A_{\varepsilon,x})+\mathcal{E}_0(u_{\varepsilon},E_{\varepsilon}, B_{\varepsilon,x})\big) + \varepsilon^{d-1}\rho_{\varepsilon}+\MMM \varepsilon^{d} \EEE \eta.
\end{align}
We now evaluate  the terms \MMM on \EEE  the right-hand side of \eqref{simplified inequality in fund est}. Since $(z_\varepsilon,F_{\varepsilon})=(\overline{u}^\mathrm{surf}_{x}, \overline{E}^{\rm surf}_x)$  on $B_\varepsilon(x)\setminus B_{(1-3\theta)\varepsilon}(x)$, by $\mathrm{(H1)}$, \MMM  $\mathrm{(H4)}$, \EEE  and \eqref{infimum problem blow up jump points} we have
\begin{align}
\label{first equation collect}
    \limsup\limits_{\varepsilon \to 0}\frac{\mathcal{E}_0(z_\varepsilon,F_\varepsilon, \QQQ A_{\varepsilon, x} \EEE )}{\gamma_{d-1}\varepsilon^{d-1}} & \leq \limsup\limits_{\varepsilon \to 0}\frac{\mathcal{E}_0(z_\varepsilon,F_\varepsilon, B_{(1-3\theta)\varepsilon}(x))}{\gamma_{d-1}\varepsilon^{d-1}}+\limsup\limits_{\varepsilon \to 0}\frac{\mathcal{E}_0(\overline{u}^\mathrm{surf}_{x},\overline{E}^{\rm surf}_x, B_{\varepsilon,x})}{\gamma_{d-1}\varepsilon^{d-1}} 
    \\& \leq \limsup\limits_{\varepsilon \to 0}\frac{m_{\mathcal{E}_0}\big(\overline{u}_{x}^\mathrm{surf}, \overline{E}^{\rm surf}_x, B_{(1-3\theta)\varepsilon}(x)\big)}{\gamma_{d-1}\varepsilon^{d-1}}+\MMM \frac{ \NNN 3 \EEE  \beta}{\gamma_{d-1}} \EEE \mathcal{H}^{d-1}\big(\partial \overline{E}^{\rm surf}_x\cap \MMM B \EEE \big) \nonumber
    \\& \leq (1-3\theta)^{d-1}\limsup\limits_{\varepsilon' \to 0}\frac{m_{\mathcal{E}_0}\big(\overline{u}_{x}^\mathrm{surf}, \overline{E}^{\rm surf}_x, B_{\NNN \varepsilon'}(x)\big)}{\gamma_{d-1}\MMM (\varepsilon')^{d-1} \EEE }+ \NNN 3 \EEE\beta (\MMM 1 \EEE -(1-\MMM 4 \EEE \theta)^{d-1}),\nonumber
\end{align}
\MMM where in the last step we substituted $(1-3\theta)\varepsilon$ with $\varepsilon^\prime$. \EEE
Regarding the other term \MMM in \eqref{simplified inequality in fund est}, \EEE by $\mathrm{(H4)}$ and \eqref{blowupatjumppoints-3}, \eqref{blowupatjumppoints-4}, \eqref{blowupatjumppoints-7} we get
\begin{align}
\label{second equation collect}
\limsup\limits_{\varepsilon \to  0}\frac{\mathcal{E}_0(u_\varepsilon,E_\varepsilon, B_{\varepsilon,x})}{\gamma_{d-1}\varepsilon^{d-1}}  &\leq \limsup\limits_{\varepsilon \to 0}\frac{\QQQ \beta \EEE }{\gamma_{d-1}\varepsilon^{d-1}}\int_{B_\varepsilon(x)} \MMM \big( 1 + \EEE \vert e(u_\varepsilon)\vert^p\big)\, \mathrm{d}x \nonumber
\\&  \ \ \ +\QQQ \beta \EEE  \ \limsup\limits_{\varepsilon \to 0}\frac{\mathcal{H}^{d-1}(\partial^* E_\varepsilon \cap B_{\varepsilon,x})}{\gamma_{d-1}\varepsilon^{d-1}}+2\QQQ \beta \EEE  \ \limsup\limits_{\varepsilon \to 0}\frac{\mathcal{H}^{d-1}(J_{u_\varepsilon} \cap B_{\varepsilon,x})}{\gamma_{d-1}\varepsilon^{d-1}} \nonumber
\\& \leq 3\QQQ \beta \EEE \ (1-(1-4\theta)^{d-1}).
\end{align}
Collecting \eqref{simplified inequality in fund est}, \eqref{first equation collect}, \eqref{second equation collect}, recalling $\rho_\varepsilon \to 0$ as well as the fact that $(w_\varepsilon,D_\varepsilon)=(u,E)$ in a neighborhood of $\partial B_\varepsilon \MMM (x) \EEE$, \NNN see \eqref{blowupatjumppoints-5}, \EEE we obtain
\begin{align*}
\lim\limits_{\varepsilon \to 0}\frac{m_{\mathcal{E}_0}(u,E,B_\varepsilon(x))}{\gamma_{d-1}\varepsilon^{d-1}}&\leq \MMM \limsup_{\varepsilon \to 0}\EEE \frac{\mathcal{E}_0(w_\varepsilon,D_\varepsilon,B_\varepsilon(x))}{\gamma_{d-1}\varepsilon^{d-1}} 
\\& \leq (1+\eta)(1-3\theta)^{d-1}\limsup\limits_{\varepsilon \to 0}\frac{m_{\mathcal{E}_0}(\overline{u}_{x}^\mathrm{surf}, \overline{E}^{\rm surf}_x, B_{\varepsilon}(x))}{\gamma_{d-1}\varepsilon^{d-1}}+\NNN 6 \EEE (1\MMM + \EEE \eta)\beta(1-(1-4\theta)^{d-1}).
\end{align*}
Passing to the limits $\eta\to 0$, $\theta \to 0$ we obtain  ``$\leq$" in \eqref{equation infimum jump}--\eqref{equation infimum void}.  

{\emph{Step 2 (Inequality ``$\geq$" in \eqref{equation infimum jump}--\eqref{equation infimum void}):}} We fix $\eta$, $\theta>0$ and let, as in {Step 1}, $(u_\varepsilon,E_\varepsilon)_\varepsilon$ be the family \MMM given in \EEE Lemma \ref{LEMMA 6.1 CRI FRIE SOL}. By \eqref{blowupatjumppoints-1}, \eqref{blowupatjumppoints-5}, and Fubini's Theorem, for each $\varepsilon>0$, we can find $s_\varepsilon \in (1-4\theta,1-3\theta)\varepsilon$ such that
\begin{align}
\label{conditions on seps}
\mathrm{(i)}&\:\: \lim\limits_{\varepsilon \to 0}\varepsilon^{-(d-1)}\mathcal{H}^{d-1}\big((\{ u_\varepsilon \neq u\} \MMM \cup \{\chi_E \neq \chi_{E_\varepsilon}\} ) \EEE \cap \partial B_{s_\varepsilon}(x)\big)=0, \nonumber
\\ \mathrm{(ii)}&\:\: \mathcal{H}^{d-1}\big((J_u \cup J_{u_\varepsilon}\cup \partial^* E \cup \partial^* E_\varepsilon)\cap \partial B_{s_\varepsilon}(x)\big)=0\:\: \text{for all}\:\: \varepsilon>0.
\end{align}
Let  $(z_\varepsilon,F_\varepsilon) \in \mathcal{G}^p(B_{s_\varepsilon}(x))$ \MMM be \EEE  such that $(z_\varepsilon,{F_\varepsilon})=(u,E)$ near $\partial B_{s_\varepsilon}\MMM (x) \EEE $ and
\begin{equation}
\label{zeps competitor}
    \mathcal{E}_0(z_\varepsilon, F_\varepsilon, B_{s_\varepsilon}(x))\leq m_{\mathcal{E}_0}(u,E,B_{s_\varepsilon}(x))+\varepsilon^d.
\end{equation}
We extend $(z_\varepsilon,{F_\varepsilon})$ to the whole $B_\varepsilon(x)$ by setting 
\begin{equation}
\label{extension zeps}
 (z_\varepsilon,{F_\varepsilon})=(u_\varepsilon,E_\varepsilon)\:\: \text{on}\:\: B_\varepsilon(x)\setminus B_{s_\varepsilon}(x).   
\end{equation}
We apply \MMM Proposition \EEE \ref{fundestE0} for $(z_\varepsilon,F_\varepsilon)$ (in place of $(u,E)$), $(\overline{u}_{x}^\mathrm{surf}, \overline{E}^{\rm surf}_x)$ (in place of $(v,F)$), and for the sets $A$, $A^\prime$ and $B$ of {Step 1}. Thus, we find $(w_\varepsilon,D_\varepsilon) \in \mathcal{G}^p(B_\varepsilon(x))$ such that $(w_\varepsilon,{D_\varepsilon})=(\overline{u}_{x}^\mathrm{surf}, \overline{E}^{\rm surf}_x)$ on $B_\varepsilon(x) \setminus B_{(1-\theta)\varepsilon}(x)$ and
\begin{align*}
    \mathcal{E}_0(w_{\varepsilon},D_{\varepsilon},B_{\varepsilon}(x))\leq (1 & +\eta) \big(\mathcal{E}_0(z_{\varepsilon},F_{\varepsilon},A_{\varepsilon,x})+\mathcal{E}_0\big(\overline{u}_{x}^\mathrm{surf},\overline{E}^{\rm surf}_x,B_{\varepsilon,x}\big) \big)   
    \\
    &
    +\frac{M^p}{\varepsilon^p}\Vert z_{\varepsilon}-\overline{u}_{x}^\mathrm{surf}\Vert^p_{L^p(A_{\varepsilon,x}\setminus A_{\varepsilon,x}')}+\MMM \eps^d \EEE \eta,
\end{align*}
where we used ${F_\varepsilon}={E_\varepsilon}={\overline{E}^{\rm surf}_x}$ on $A_{\varepsilon,x}\setminus A_{\varepsilon,x}'$, \MMM see \eqref{blowupatjumppoints-6}. \EEE In addition, we observe that by construction $z_\varepsilon=u_\varepsilon$ outside \MMM of \EEE $B_{(1-3\theta)\varepsilon}(x)$, \MMM see \eqref{extension zeps}. \EEE Then, as in {Step 1} we apply \eqref{blowupatjumppoints-2} \MMM to get \EEE a sequence $(\rho_\varepsilon)_\varepsilon$ with $\rho_\varepsilon \to 0$ such that
\begin{equation}
\label{simplified fund est jump points 2}
    \mathcal{E}_0(w_\varepsilon,D_\varepsilon,B_\varepsilon(x))\leq (1+\eta) \big(\mathcal{E}_0(z_{\varepsilon},F_{\varepsilon},A_{\varepsilon,x})+\mathcal{E}_0\big(\overline{u}_{x}^\mathrm{surf},\overline{E}^{\rm surf}_x,B_{\varepsilon,x}\big) \big)+\varepsilon^{d-1}\rho_\varepsilon+\NNN \varepsilon^d\eta. \EEE
    \end{equation}
We estimate the first term on the right-hand side of \eqref{simplified fund est jump points 2}. \QQQ By \EEE $\mathrm{(H1)}$, $\mathrm{(H4)}$, \eqref{zeps competitor}, \eqref{extension zeps}, and the choice of $s_\varepsilon$ \QQQ we get \EEE that
\begin{align*}
 \mathcal{E}_0(z_\varepsilon,F_\varepsilon, A_{\varepsilon,x} )&\leq m_{\mathcal{E}_0}(u,E,B_{s_\varepsilon}(x))+\varepsilon^d+\beta \mathcal{H}^{d-1}\big((\{\chi_E \neq \chi_{E_\varepsilon}\}\cup \partial^* E_\varepsilon \cup \partial^* E)\cap \partial B_{s_\varepsilon}(x)\big) \nonumber
 \\&  \ \ \ + 2\beta \mathcal{H}^{d-1}\big((\{u \neq u_\varepsilon\}\cup J_{u_\varepsilon} \cup J_u)\cap \partial B_{s_\varepsilon}(x)\big)+\mathcal{E}_0(u_\varepsilon,E_\varepsilon,B_{\varepsilon,x}).
\end{align*}
By \eqref{second equation collect} \NNN and \EEE \eqref{conditions on seps} we deduce that  
\begin{align}
\label{equation to collect in second part 1}
\limsup\limits_{\varepsilon \to 0}\frac{\mathcal{E}_0(z_\varepsilon,F_\varepsilon,A_{\varepsilon,x})}{\gamma_{d-1}\varepsilon^{d-1}}&\leq \limsup\limits_{\varepsilon \to 0}\frac{m_{\mathcal{E}_0}(u,E,B_{s_\varepsilon}(x))}{\gamma_{d-1}\varepsilon^{d-1}}  + 3\beta (1-(1-4\theta)^{d-1}) \nonumber
\\& \leq (1-3\theta)^{d-1}\limsup\limits_{\varepsilon \to 0}\frac{m_{\mathcal{E}_0}(u,E,B_{\varepsilon}(x))}{\gamma_{d-1}\varepsilon^{d-1}}  + 3\beta (1-(1-4\theta)^{d-1}),
\end{align}
\MMM where in the last step we used   the fact that $s_\varepsilon \leq (1-3\theta)\varepsilon$.  Similarly to \eqref{first equation collect}, \EEE it also holds that
\begin{equation}
\label{equation to collect in second part 2}
\limsup\limits_{\varepsilon \to 0}\frac{\mathcal{E}_0(\overline{u}_{x}^\mathrm{surf},\overline{E}^{\rm surf}_x,B_{\varepsilon,x})}{\gamma_{d-1}\varepsilon^{d-1}}\leq \NNN 3 \EEE \beta (1-(1-4\theta)^{d-1}).
\end{equation}
Collecting \eqref{simplified fund est jump points 2}, \eqref{equation to collect in second part 1}, \eqref{equation to collect in second part 2}, and using $\rho_\varepsilon \to 0$   we get \MMM 
\begin{align*}
 \limsup_{\varepsilon \to 0}\frac{\mathcal{E}_0(w_\varepsilon,D_\varepsilon,B_\varepsilon(x))}{\gamma_{d-1}\varepsilon^{d-1}} 
\leq (1+\eta)(1-3\theta)^{d-1}\limsup\limits_{\varepsilon \to 0}\frac{m_{\mathcal{E}_0}(u,E, B_{\varepsilon}(x))}{\gamma_{d-1}\varepsilon^{d-1}}+6(1   +   \eta)\beta(1-(1-4\theta)^{d-1}).
\end{align*}
Passing to the limits $\eta\to 0$, $\theta \to 0$  and recalling that  $(w_\varepsilon,D_\varepsilon)= (\overline{u}_{x}^\mathrm{surf}, \overline{E}^{\rm surf}_x)$ in a neighborhood of $\partial B_\varepsilon (x)$, we conclude 
$$\limsup_{\varepsilon \to 0}\frac{m_{\mathcal{E}_0}(\overline{u}_{x}^\mathrm{surf},\overline{E}^{\rm surf}_x),B_\varepsilon(x))}{\gamma_{d-1}\varepsilon^{d-1}}\leq
\limsup\limits_{\varepsilon \to 0}\frac{m_{\mathcal{E}_0}(u,E, B_{\varepsilon}(x))}{\gamma_{d-1}\varepsilon^{d-1}} = \lim\limits_{\varepsilon \to 0}\frac{m_{\mathcal{E}_0}(u,E, B_{\varepsilon}(x))}{\gamma_{d-1}\varepsilon^{d-1}}.
 $$
\EEE This  shows  ``$\geq$" in \eqref{equation infimum jump}--\eqref{equation infimum void}.  
\end{proof}

\subsection{Fundamental estimate and compactness for the sequence $\mathcal{E}_{\varepsilon}$}\label{sec3.3}
 We now  proceed \MMM by \EEE  showing that $(\mathcal{E}_\varepsilon)_\eps$ admits a $\Gamma$-converging subsequence and that its limit   satisfies the assumptions of Theorem~\ref{Theorem2.1 CriFrieSol}. \MMM To this end, \EEE we first  prove  a fundamental estimate. The proof \EEE is based on the strategy of \cite[Proposition~4.1]{FriPerSol20a}, up to taking into account the \MMM presence of the void set, in particular the \EEE constraint that the jump set of competitors needs to be contained in the essential boundary of the voids. \MMM Recall the definition in  \eqref{SBVdefW}.  \NNN By convention, given $(u,E) \in \mathcal{W}^{1,p}(A)$, we may regard $u$ as a measurable function on $\R^d$, extended by $u = 0$ on $\R^d \setminus A$. \EEE  We further define \EEE $\psi \colon [0,\infty)\to [0,1)$ by $\psi(t)=\frac{t}{1+t}$ for $t\geq 0$. \MMM We observe that $u_n \to u$ in measure on $U$ as $n \to \infty$ \MMM if and only if
$\NNN \int_{U}\psi(\vert u - v\vert)\, \mathrm{d}x    \to 0$  \EEE   as $n \to \infty$. 
\begin{proposition}[Fundamental estimate for \MMM energies \eqref{newenergies}\EEE]
\label{fundamentalestimate}
 Let $\eta \in (0,1)$ and let $A^\prime$, $A$, $B \in \mathcal{A}$.  Assume that \MMM $A' \subset \subset A$ and  \EEE $A\setminus  \MMM \overline{A'} \EEE  \subset B$. Then, there exists a function $\Lambda \colon \MMM L^0(A;\R^d) \times L^0(B;\R^d) \EEE \to [0,\infty]$ which is lower semicontinuous with respect to the convergence in measure and satisfies
\begin{equation}
\label{eq4.25} 
\Lambda(z_1,z_2)\to 0\:\: \text{whenever}\:\: z_1-z_2 \to 0\:\: \text{in measure on}\:\: A\setminus \overline{A^\prime} 
\end{equation}
such that the following holds: for   \MMM every functional $\mathcal{E}$ of the form \eqref{newenergies} (omitting index $\eps$), \EEE   every $(u,E) \in \mathcal{W}^{1,p}(A)$, and  $(v,F) \in \mathcal{W}^{1,p}(B)$ \EEE there exists \MMM $(w,D)  \in \mathcal{W}^{1,p}(A^\prime\cup B)$ \EEE such that
\begin{subequations}
    \label{eqfundamentalestimate}
\begin{align}
 & \MMM \mathcal{E} \EEE(w,D,A'\cup B)\leq (1+\eta)(\MMM \mathcal{E} \EEE(u,E,A)+\MMM \mathcal{E} \EEE(v,F,B))+\Lambda(u,v)+ C\Vert \chi_{E}-\chi_{F}\Vert_{L^1(A \setminus A^\prime)}+ \eta, \EEE   \label{eqfundamentalestimate-1}
 \\ &   \NNN \Vert\min \{ \psi(|w-u|), \psi(|w-v|) \} \Vert_{L^1(A' \cup B)} \EEE \leq (1+ \mathcal{E}(u,E,A)+\mathcal{E}(v,F,B))\eta,   \label{eqfundamentalestimate-2}
 \\ &\NNN \Vert\min \{ |\chi_D-\chi_E|, |\chi_D-\chi_F| \} \Vert_{L^1(A' \cup B)} \EEE\leq (1+ \mathcal{E}(u,E,A)+\mathcal{E}(v,F,B))\eta,   \label{eqfundamentalestimate-3}
 \\  & (w,D)=(u,E) \:\: \text{on}\:\: A^\prime, \:\: (w,D)=(v,F) \:\: \text{on}\:\: B \setminus\overline{A},     \label{eqfundamentalestimate-4}
\end{align}
\end{subequations}\
 where \MMM $C>0$ denotes a constant which only depends on  $A'$,  $A$, $p$, $\alpha$, $\beta$,  $\eta$, \NNN and $d$. \EEE 
  
\end{proposition}

\begin{proof}
\MMM We proceed in several steps. \EEE

\noindent \emph{Step 1 (Preliminaries)}. We let $k \in \mathbb{N}$. \MMM Below in \eqref{first equation of this page-we} (see also \eqref{error}), $k$ will be chosen large enough depending on $p$,$d$,$\alpha$,$\beta$, $\eta$, $A$, and $A'$. \EEE Let $A_1,...,A_{k+1}$ be in $\mathcal{A}$ with 
\begin{equation}
    A^\prime  \subset \subset A_1 \subset \subset ... \subset \subset A_{k+1}\subset \subset A.
\end{equation}
We also define further open sets $    A_i \subset \subset A_i^+ \subset \subset A^-_{i+1} \subset \subset A_{i+1}$, $i=1,\ldots,k$, 
and
\begin{align}
\label{TiSi}
 S_i \defas A_{i+1}\setminus \overline{A_i},  \quad \quad \quad T_i \defas A^-_{i+1}\setminus \overline{A^+_i}.
\end{align}
 Let $\varphi_i \in C_c^\infty(\mathbb{R}^d)$ be a cut off function with $\varphi_i =1$ on $A^+_i$ and $0$ in $\mathbb{R}^d \setminus A^-_{i+1}$ i.e.,
\begin{equation*}
\{ 0<\varphi_i<1\}\subset T_i.   
\end{equation*}
We apply Lemma \ref{lemmamatrix} for $\delta=1/2$, $R=\sqrt{d}$, and let $\tau_{\psi}(0)=0$ be the continuous, strictly increasing function with $\tau_{\psi}(0)=0$. As $\tau_{\psi}$ is uniformly continuous on $[0,\frac{1}{2}]$ we can choose $\lambda \in (0,\infty)$ such that
\begin{subequations}
    \label{lambdatrick}
\begin{align}
   & \ \ 2\lambda<\frac{1}{2},  \label{lambdatrick-1}\\
    & \ \  \NNN 32^p\beta \EEE \big(1+\max_{i=1,\ldots,k}\Vert \nabla \varphi_i \Vert^p_{L^{\infty}(\mathbb{R}^d)}\big) \, \mathcal{L}^d(A \setminus A^\prime)\max_{t \in [0,\frac{1}{2}]}\vert \tau^p_{\psi}(t+2\lambda)-\tau_{\psi}^p(t)\vert \leq \frac{\eta}{2}.  \label{lambdatrick-2}
\end{align}
\end{subequations}
The parameter $\lambda$ will be important for the definition of the function $\Lambda$ of the statement. We highlight that, since $A \setminus \overline{A^\prime} \subset B$, we have $T_i \subset \subset S_i \subset A \setminus \MMM \overline{A^\prime}\EEE$. We pick $\rho \in (0,\eta)$ sufficiently small such that
\begin{equation}
\label{defrho}
{\rm (i)}\:\:   2^p   \lambda^{-p}\overline{c}\rho^{p-1}\leq \lambda ,\quad\quad {\rm (ii)}\:\: \rho \MMM \le   \EEE  \min_{i=1,\ldots,k}\{ \mathrm{dist}(T_i,\mathbb{R}^d \setminus S_i), \Vert \nabla \varphi_i \Vert^{-1}_{L^{\infty}(\mathbb{R}^d)}\},
\end{equation}
\QQQ where $\overline{c}$ is the constant appearing in Theorem~\ref{Korn inequality for functions with small jump set} and Remark~\ref{scalinginvarianceoncubes}. \EEE We now fix \MMM $(u,E) \in \mathcal{W}^{1,p}(A)$ \EEE and $(v,F)  \in \mathcal{W}^{1,p}(B)$. Without loss of generality \MMM we \EEE assume that
\begin{equation}
\label{uclosetov}
2\rho^{-d}\int_{\MMM A \setminus A^\prime \EEE }\psi(\vert u-v\vert)\, \mathrm{d}x \leq \frac{1}{2}.    
\end{equation}
Indeed, if \eqref{uclosetov} does not hold, we simply set $\Lambda(u,v)=\infty$,  $D=E\cup \MMM (F\setminus A') \EEE $,   and
\begin{equation}
    w(x)\defas \begin{cases}
        \varphi_1(x) u(x) +(1-\varphi_1(x))v(x) \:\: \text{if}\:\: x \in (A^\prime \cup B) \setminus D, \\
        0 \quad \quad \quad \quad \quad \quad \quad \quad \quad \quad \quad \quad\: \, \text{if}\:\: x \in D.
    \end{cases}
\end{equation}
\MMM Based on this, for each $z_1 \in \MMM L^0(A;\R^d)   \EEE$ and $z_2 \in \MMM  L^0(B;\R^d) \EEE$ we define
\begin{equation}
\label{Lambda*}
\Lambda^*(z_1,z_2)\defas  2^{p} \EEE\tau^p_\psi\Big(2\rho^{-d}\int_{A\setminus A^\prime}\psi(\vert z_1-z_2\vert)\, \mathrm{d}x+2\lambda\Big),
\end{equation}
whenever $2\rho^{-d}\int_{(A\setminus A^\prime)\cap B}\psi(\vert z_1-z_2\vert)\, \mathrm{d}x\leq \frac{1}{2}$ and $\Lambda^*(z_1,z_2)=\infty$ else, \MMM where \EEE $\tau_{\psi}\colon [0,1)\to [0,\infty)$ is the function obtained by Lemma \ref{lemmamatrix} for $\delta=\frac{1}{2}$ and $R=\sqrt{d}$. \MMM Note that this is well defined by \eqref{lambdatrick-1}.

\emph{Step 2 (Covering with cubes \MMM and Korn's inequality)\EEE}. For each $i=1,\ldots,k$, we cover $T_i$ up to a set of $\mathcal{L}^d$-negligible measure with a finite number of pairwise disjoint open cubes $\mathcal{Q}^i=(Q)_{Q \in Q^i}$ with centers $(x_Q)_Q \subset \rho \mathbb{Z}^d\cap T_i$ and sidelength $\rho$. \EEE Then $A \setminus \overline{A^\prime} \subset B$, \EEE \eqref{TiSi}, and \eqref{defrho} imply
\begin{equation*}
 T_i \subset \bigcup_{Q \in \MMM \mathcal{Q}^i\EEE}\overline{Q} \subset \subset S_i \subset A \cap B.  
\end{equation*}
Let $\overline{c}$ be the constant of Remark \ref{scalinginvarianceoncubes}. We say that $Q \MMM \in \mathcal{Q}^i \EEE $ is a \emph{bad cube} if \EEE 
\begin{equation}
\label{korninequalityapplied}
 (2\overline{c})^{\frac{d-1}{d}}\mathcal{H}^{d-1}\big((\partial^* E \cup \partial^* F)\cap Q\big)+ \Vert e(u)\Vert_{L^p(Q)}^p  + \Vert e(v)\Vert_{L^p(Q)}^p \geq \rho^{d-1},
\end{equation}
otherwise we say it is a \emph{good cube}. In particular, we denote the family of all good cubes in $\mathcal{Q}^i$ \MMM by \EEE $\mathcal{Q}_{\mathrm{good}}^i$ and the family of all bad cubes $\mathcal{Q}^i$ \MMM by \EEE $\mathcal{Q}_{\mathrm{bad}}^i$.  

 \MMM Regarding $u-v$ as a function in $GSBD^p(A\cap B)$ with $J_{u-v} \subset \partial^* E\cup \partial^* F$ and applying \EEE  Theorem~\ref{Korn inequality for functions with small jump set} and Remark \ref{scalinginvarianceoncubes} \MMM on $u-v$, \EEE for every cube $Q \in \MMM \mathcal{Q}^i_{\rm good}\EEE$ there exists a measurable set ${\omega}_Q \subset Q$ and a rigid motion ${a}_Q$ such that  
\begin{align}
\label{eq1.15}
\Vert u-v-{a}_{Q}\Vert^p_{L^p(Q \setminus {\omega}_Q)}&\leq \overline{c} \NNN 2^{p-1} \EEE\rho^p(\Vert e(u)\Vert_{L^p(Q)}^p  +\Vert e(v)\Vert_{L^p(Q)}^p),\\
\label{measuresingularset}
  \mathcal{L}^d({\omega}_Q)&\leq \frac{1}{2}\mathcal{L}^d(Q),\\
\label{claimofstep2}    
\Vert {a}_{Q}\Vert^p_{L^p( Q)}&\leq \rho^d\Lambda^*(u,v).
\end{align}
\MMM Indeed, \eqref{eq1.15} follows from the statement and the subsequent remark, and for \eqref{measuresingularset} we apply \EEE \eqref{korninequalityapplied} and again Remark \ref{scalinginvarianceoncubes}. Using \eqref{eq1.15}, \eqref{measuresingularset}, and arguing like in \cite[Proof of Proposition 4.1, {Step~4}]{FriPerSol20a}, one can show \eqref{claimofstep2}. \MMM (At this point, we also use property \eqref{defrho}(i).)

\noindent\emph{Step 3 (Construction of \MMM interpolations and \EEE void sets).}  In this step we construct pairs  \MMM $(w_i,D_i)$ for $i=1,\ldots,k$. In Step  5 below, we then pick a specific pair $(w_{i_0},D_{i_0})$. \EEE For every $Q \in \mathcal{Q}_{\mathrm{good}}^i$ we let $\omega_Q$  be \MMM the \EEE set satisfying \eqref{eq1.15} and \eqref{measuresingularset}, while for every $Q \in \mathcal{Q}_{\mathrm{bad}}^i$ we let $\omega_Q=Q$. \MMM We \EEE define $U^i \defas \bigcup_{Q \in \mathcal{Q}^i}\omega_Q$ and  set
\begin{equation}
\label{ui and vi}
    u_i \defas u(1-\chi_{U^i})\quad \text{and}\:\: v_i \defas v(1-\chi_{U^i}).
\end{equation}
\MMM We \EEE notice that $(u_i,E_i) \in \mathcal{W}^{1,p}(A)$ and $(v_i,F_i) \in \mathcal{W}^{1,p}(B)$, where $E_i\defas E\cup U^i$ and $F_i\defas F \cup U^i$.

To construct \MMM $D_i$, \EEE we follow an argument similar to the one of \cite[Lemma 4.4]{AmbBra90}. 
 By \MMM the \EEE Fleming-Rishel formula (\cite[(4.5.9)]{alma991013725792303131}, see also \cite[(2.8)]{AmbBra90}), we know that there exist $t_i^1$ $t^2_i \in (0,1)$, with $t_i^2-t^1_i>\frac{1}{2}$ such that  $\{ \varphi_i> \MMM t^l_i \EEE \}$ has finite perimeter in $A$, $\mathcal{H}^{d-1}(\partial^* E_i\cap \{ \varphi_i=t^l_i\})=\mathcal{H}^{d-1}(\MMM \partial^* F_i \EEE \cap \{ \MMM \varphi_i=t^l_i \EEE \})=0$, \NNN  $\mathcal{L}^{d} (\{ \varphi_{i} = t^{l}_{i} \} ) = 0$ for $l \in \{1, 2\}$, \EEE and
\begin{align}
\label{fleming rishel}
 \int_{S_i \cap \partial^* \{\varphi_i<t^l_i\}} \MMM |\chi_{E_i}-\chi_{F_i}| \EEE \,\mathrm{d}\mathcal{H}^{d-1}&\leq \MMM 4 \EEE \int_{S_i}  \MMM |\chi_{E_i}-\chi_{F_i}| \EEE \vert \nabla \varphi_i \vert \mathrm{d}x   \leq  \MMM 4 \EEE \Vert \nabla\varphi_i\Vert_{L^\infty} \Vert \chi_{E}-\chi_{F}\Vert_{L^1(S_i)}
\end{align}
\MMM for  $l \in \{1,2\}$, cf.\ \eqref{fleming rishel 2} for a similar argument. \EEE We define  
\begin{equation}
\label{def Di}
    {D_i} \QQQ \defas \EEE \begin{cases}
        {E_i} \:\: \:\:\:\: \quad \:\text{if}\:\: \varphi_i > t^2_i, \\
        E_i\cup F_i \:\:\text{if}\:\: \varphi_i\in [t^1_i,t^2_i],\\
        F_i \:\: \:\:\:\: \quad \:\text{if}\:\: \varphi_i < t^1_i.
    \end{cases}
\end{equation}
Notice that ${D_i}={E}$ on $A^\prime$ since $U^i\subset S_i \subset A \setminus A^\prime$. \NNN In a similar fashion,  ${D_i}={F}$ on $B \setminus \overline{A}$. \EEE 

\MMM Next, we  \EEE  define $w_i \in SBV^{p}(A^\prime \cup B)$ as $w_i(x) \QQQ \defas \EEE 0$ if $x \in D_i$, and for $x \in (A^\prime \cup B) \setminus {D_i}$ we set
\begin{equation}
\label{def wi}
    w_i(x) \QQQ \defas \EEE \begin{cases}
   {u}_i(x)\quad \quad \quad \quad \quad \quad \quad \quad \quad \quad \:\:\:\:\:\: \text{if}\:\: \varphi_i(x)>t^2_i, \\
       \psi_i(x)\,{u}_i(x)+(1-\psi_i(x))\,{v}_i(x)\:\: \text{if}\:\: \varphi_i(x)\in [t^1_i,t_i^2], \\
      {v}_i(x)\quad \quad \quad \quad \quad \quad \quad \quad \quad \quad \:\:\:\:\:\: \text{if}\:\: \varphi_i(x)< t^1_i, \\
    \end{cases}
\end{equation}
where $\psi_i(x)\defas \frac{\varphi_i(x)-t^1_i}{t^2_i-t^1_i}$. \MMM In view of \eqref{ui and vi}, \eqref{def Di}, and the fact that  $(u_i,E_i) \in \mathcal{W}^{1,p}(A)$,   $(v_i,F_i) \in \mathcal{W}^{1,p}(B)$, we note that $J_{w_i} \cap (A' \cup B) \subset \partial^* D_i$ and thus $(w_i,D_i) \in  \mathcal{W}^{1,p}(A' \cup B)$.  Moreover, by \eqref{ui and vi} and $U^i \subset A \setminus A^\prime$ we have $w_i = u$ on $A^\prime$ and $w_i = v$ on $B \setminus \overline{A}$.  \EEE

\noindent \emph{Step 4 \MMM (Energy estimate for $(w_i,D_i)$)\EEE}. \MMM In this step we estimate $\mathcal{E}(w_i,D_i, A^\prime \cup B)$. \MMM For notational convenience, we write the functional $\mathcal{E}$ as $\mathcal{E}(u,E,A)  = \mathcal{F}(u,E,A) + \QQQ \mathcal{J} \MMM (E,A) $, where \EEE
\begin{equation}
\label{energywithonlythesurface}
  \mathcal{F}(u,E,A)\defas  
        \int_{A \setminus \MMM E \EEE }f(x,e(u))\, \mathrm{d}x , \quad \quad \quad   \QQQ \mathcal{J} \EEE (E,A)\defas  
        \int_{\partial^* E\cap A }g(x,\nu_E)\, \MMM {\rm d}\mathcal{H}^{d-1}, \EEE 
\end{equation}
with $f$ and $g$ being the densities of the functional $\mathcal{E}$. \MMM We use similar notation for $\mathcal{E}(v,F,B)$. We estimate bulk and surface contributions separately. We start with the surface term. \EEE   For this purpose, \MMM recalling the definition of $U^i$ preceding \eqref{ui and vi}, \EEE we observe that
\begin{equation}\label{twotwoterms}
    \mathcal{H}^{d-1}(\partial^*U^i)\leq \sum_{Q\in \mathcal{Q}_{\mathrm{good}}^i}\mathcal{H}^{d-1}(\partial^* \omega_Q)+\sum_{Q^\prime\in \mathcal{Q}_{\mathrm{bad}}^i}\mathcal{H}^{d-1}(\partial^* \omega_{Q^\prime}).
\end{equation}
Regarding the first term on the right-hand side, because of $ J_{u-v} \cap Q  \subset \partial^* E \cup \partial^* F \cap Q$ for every $Q\in \mathcal{Q}^i$, \EEE Theorem \ref{Korn inequality for functions with small jump set} and Remark~\ref{scalinginvarianceoncubes} \MMM yield \EEE
\begin{equation}
\label{perimeter good cubes}
\sum_{Q\in \mathcal{Q}_{\mathrm{good}}^i}\mathcal{H}^{d-1}(\partial^* \omega_Q) \leq \overline{c}\sum_{Q\in \mathcal{Q}_{\mathrm{good}}^i}\mathcal{H}^{d-1}\big((\partial^* E \cup \partial^* F)\cap Q\big)\leq \frac{\overline{c}}{\alpha}\big( \mathcal{E}(u,E,S_i)+\mathcal{E}(v,F,S_i)\big),
\end{equation}
where we used that the cubes $Q \in \mathcal{Q}^i$ are all disjoint and contained in $S_i$, and then we applied $(g_2)$. Concerning the second term,  by \eqref{korninequalityapplied} and $(g_2)$ we get \EEE 
\begin{align}
\label{perimeter of bad cubes}
\NNN \frac{1}{2d} \EEE \sum_{Q^\prime\in \mathcal{Q}_{\mathrm{bad}}^i}\mathcal{H}^{d-1}(\partial^* \omega_{Q^\prime})&=    \# Q_{\mathrm{bad}}^i \rho^{d-1}\leq (2\overline{c})^{\frac{d-1}{d}}\mathcal{H}^{d-1}\big((\partial^* E \cup \partial^* F)\cap S_i\big)+ \Vert e(u)\Vert_{L^p(S_i)}^p  + \Vert e(v)\Vert_{L^p(S_i)}^p \nonumber
\\& \leq \frac{1}{\alpha} (2\overline{c})^{\frac{d-1}{d}}\big( \mathcal{E}(u,E,S_i)+\mathcal{E}(v,F,S_i)\big),
\end{align}
where we used $\overline{c}>1$ and again that cubes \MMM in \EEE $\mathcal{Q}^i$ are disjoint and contained in $S_i$. Thus, \MMM \eqref{twotwoterms}, \EEE \eqref{perimeter good cubes}, \eqref{perimeter of bad cubes}, and $(g_3)$ give \NNN
\begin{align}
\label{estimate on the energies}
 \mathcal{J}   (E_i, A )+  \mathcal{J}   (F_i, B  )\leq      \mathcal{J}  (E,A) +    \mathcal{J}   (F,B)     +  \EEE  C_1\big(\mathcal{E}(u,E,\MMM S_i \EEE )+\mathcal{E}(v,F,\MMM S_i \EEE )\big) \EEE
\end{align}
\NNN for some $C_1 \ge 1$ depending on $\alpha$, $\beta$, and $d$. \EEE   Let us now estimate the surface energy of $D_i$. \EEE By $(g_3)$ and the choice of $t^1_i$ and $t^2_i$, we have 
\begin{align*}
\QQQ \mathcal{J} \MMM (D_i,A' \cup B) & \leq   \mathcal{J} \MMM (\NNN E_i, \EEE A) +  \QQQ \mathcal{J} \MMM ( \NNN F_i, \EEE B) \EEE   +\int_{\partial^* D_i\cap \{\varphi_i \in (t^1_i,t^2_i)\}}\hspace{-0.3cm}g(x,\nu_{D_i})\, \mathrm{d}\mathcal{H}^{d-1} +\beta \sum_{l=1,2} \mathcal{H}^{d-1}(\partial^* D_i \cap \{\varphi_i=t^l_i \}).
\end{align*}
Combining \MMM this with \NNN \eqref{twotwoterms}--\eqref{estimate on the energies} \EEE and  \eqref{fleming rishel}  we obtain 
\begin{align}
\label{correction fund est 3}
  \QQQ \mathcal{J} \MMM (D_i,A' \cup B) \EEE &\leq  \MMM \mathcal{J}(E,A) +  \mathcal{J}(F,B) \EEE  + 2 \EEE {C_1}\big(\mathcal{E}(u,E,S_i)+\mathcal{E}(v,F,S_i)\big)+8 \EEE \beta\Vert \nabla\varphi_i\Vert_{L^\infty} \Vert \chi_{E}-\chi_{F}\Vert_{L^1(S_i)}.
\end{align}
We now come to the bulk part.  \EEE  Let us denote by \NNN $V^i_{\mathrm{good}}$ \EEE the union of all good cubes $Q \in \mathcal{Q}_\mathrm{good}^i$. We notice that  
\begin{align}
\label{fullenergywi-corrected}
    \mathcal{F}(w_i,D_i,A^\prime \cup B)&\leq \mathcal{F}(w_i,D_i,A^-_{i+1}\cup B) = \mathcal{F}(w_i,D_i,A_{i+1}^-)+\mathcal{F}(w_i,D_i, B \setminus A^-_{i+1}) \nonumber
    \\& \leq \mathcal{F}({u}_i,E_i,A)+\mathcal{F}({v}_i,F_i,B)+\mathcal{F}(w_i,D_i, \{\varphi_i \in (t^1_i,t_i^2)\}) \nonumber
    \\& \leq \mathcal{F}(u,E,A)+\mathcal{F}(v,F,B)+\mathcal{F}\big(w_i,D_i, V_{\mathrm{good}}^i\cap \{\varphi_i \in (t^1_i,t_i^2)\}\big),
\end{align}
where we used that $\mathcal{F}(u_i,E_i,A)\leq \mathcal{F}(u,E,A)$ and $\mathcal{F}(v_i,F_i,B)\leq \mathcal{F}(v,F,B)$ \MMM due to   $E\subset E_i$,  $F\subset F_i$, \eqref{ui and vi}, and $(f_1)$. \EEE \MMM Denote the last term on the right-hand side of \eqref{fullenergywi-corrected} by $\mathcal{F}_i^*$. We estimate  \EEE  
\begin{align*}
\MMM \mathcal{F}_i^* \EEE &\leq \beta \int_{V_{\mathrm{good}}^i}(1+\vert e(w_i)\vert^p)\, \mathrm{d}x\leq \beta \mathcal{L}^d(S_i)+\beta \int_{V_{\mathrm{good}}^i}\vert \psi_i e({u}_i-{v}_i)+\nabla \psi_i \odot ({u}_i - {v}_i)\vert^p\, \mathrm{d}x   
\\& \leq \beta \mathcal{L}^d(S_i)+\NNN 3^{p-1} \EEE \beta \int_{V_{\mathrm{good}}^i} \big(\vert e({u}) \vert^p + \vert e({v})\vert^p +  2^p \EEE\max_{i=1,\ldots,k}\Vert \nabla \varphi_i \Vert_{L^\infty(S_i)}^p\vert {u_i}-{v_i}\vert^p \big)\, \mathrm{d}x,
\end{align*}
where we used \MMM $(f_3)$ and the fact \EEE that $\frac{1}{t^2_i-t^1_i}\leq 2$ for every $i=1,\ldots,k$.  \NNN We further compute \EEE 
\begin{align}
\label{stimanormalp}
\int_{V_{\mathrm{good}}^i}\vert {u}_i-{v}_i\vert^p \, \mathrm{d}x &\leq
\sum_{Q \in \Q_{\mathrm{good}}^i}\int_{Q}\vert {u}_i-{v}_i\vert^p \, \mathrm{d}x
\leq 2^{p-1}\sum_{Q \in \Q_{\mathrm{good}}^i} \Big(\int_{Q}\vert {u}_i-{v}_i-{a}_Q\vert^p \, \mathrm{d}x
+\int_{Q}\vert a_Q\vert^p \, \mathrm{d}x \Big) \nonumber
\\& =  2^{p-1}\sum_{Q \in \Q_{\mathrm{good}}^i} \Big(\int_{Q\setminus \omega_Q}\vert {u}-{v}-{a}_Q\vert^p \, \mathrm{d}x
+ \MMM 2 \EEE \int_{Q}\vert a_Q\vert^p \, \mathrm{d}x \Big)\nonumber
\\&  \leq \NNN  4^{p-1} \EEE \overline{c}\rho^p (\Vert e(u)\Vert^p_{L^p(S_i)}+\Vert e(v)\Vert^p_{L^p(S_i)})+\NNN 2^{p} \EEE \mathcal{L}^d(S_i)\Lambda^*(u,v), 
\end{align}
where we used \eqref{eq1.15}, \eqref{claimofstep2}, \MMM \eqref{ui and vi},  and again the fact that all the cubes $Q \in \mathcal{Q}_{\mathrm{good}}^i$ are disjoint and contained in $S_i$. Combining the previous equations and using $(f_2)$, \eqref{defrho}(ii),  we have
\begin{align}
\label{controlloenergiaui}
\MMM \mathcal{F}_i^* \EEE&\leq C_2 \mathcal{L}^d(S_i)+C_2(\mathcal{E}(u,E,S_i)+\mathcal{E}(v,F,S_i)) +  16^p \beta \EEE \max_{i=1,\ldots,k}\Vert \nabla \varphi_i \MMM \Vert^p_{L^{\infty}(S_i)} \EEE \mathcal{L}^d(S_i)\Lambda^*(u,v),
\end{align}
where  \NNN $C_2 \ge 1 $    depends only on $p$, $\alpha$, $\beta$, and $d$.  Let $L=L(p,d,\alpha,\beta)  \defas \MMM 2C_1+  C_2$, \MMM where $C_1$ is \NNN given in \EEE \eqref{estimate on the energies}. \EEE  
Combining \eqref{correction fund est 3}, \eqref{fullenergywi-corrected},  and \eqref{controlloenergiaui}, we get
\begin{align}
\label{fullenergyestimate}
\mathcal{E}(w_i,D_i, A^\prime \cup B)&\leq \mathcal{E}(u,E,A)+\mathcal{E}(v,F,B)  +\mathrm{err}(u,v,E,F,S_i) +M\Lambda^*(u,v)+ \MMM C \EEE \Vert \chi_E - \chi_F\Vert_{L^1(\MMM A \setminus A'\EEE)},
\end{align}
where 
\begin{align}
\label{error}
    \mathrm{err}(u,v,E,F,S_i)   \defas L \mathcal{L}^d(S_i)+L(\mathcal{E}(u,E,S_i)+\mathcal{E}(v,F,S_i)),
\end{align}
and
\begin{align*}
    M \defas \MMM   16^p \EEE\beta \EEE\max_{i=1,\ldots,k}\Vert \nabla \varphi_i \Vert_{L^{\infty}(S_i)}^p \mathcal{L}^d(A \setminus A^\prime), \MMM  \quad \quad C \defas 8 \beta \max_{i=1,\ldots,k}\Vert \nabla \varphi_i \Vert_{L^\infty(S_i)}. \EEE
\end{align*}
\MMM Note that $C$  depends only on $\beta$, $A'$,  $A$, and $k$, and thus only on  $A'$,  $A$, $p$, $\alpha$, $\beta$,  $\eta$, \NNN and $d$, \EEE cf.\ \eqref{first equation of this page-we} below.  \EEE

\MMM \noindent \emph{Step 5 (Conclusion)}.  Let \EEE
\begin{equation}
\label{defLambda}
\Lambda(z_1,z_2)\defas  M \EEE 2^p \tau^p_\psi\Big(2\rho^{-d} \MMM \int_{A\setminus A^\prime } \EEE \psi(\vert z_1-z_2\vert)\, \mathrm{d}x\Big),
\end{equation}
whenever $\MMM \int_{A\setminus A^\prime } \EEE \psi(\vert z_1-z_2\vert)\, \mathrm{d}x\leq \frac{1}{2}$ and $\Lambda(z_1,z_2)=\infty$ else. \MMM Note that $\Lambda$ is lower semicontinuous by Fatou's lemma and the fact 
that $\tau^p_\psi$  is continuous and increasing. Then, by  \EEE \eqref{lambdatrick-2} and \eqref{Lambda*} we have
\begin{equation}\label{forlambda}
    \vert \Lambda(u,v)-M\Lambda^*(u,v)\vert \leq \frac{\eta}{2}.
\end{equation}
Notice that by definition $ \mathrm{err}(u,v,E,F,\cdot)$ is   additive   in the last variable and $\mathrm{err}(u,v,E,F,A\setminus A^\prime)<\infty$.
Recalling that the sets $S_i$ are contained in $A\setminus A^\prime$ and that they are pairwise disjoint, we can choose an $i_0 \in \{1,\ldots,k\}$ such that  \MMM
\begin{align}
\label{first equation of this page}
\mathrm{err}(u,v,E,F,S_{i_0}) \le \frac{1}{k} \mathrm{err}(u,v,E,F, A \setminus A^\prime \EEE ).
\end{align}
By taking $k\ge  \frac{2L}{\eta} \max\lbrace 1, \mathcal{L}^{d}(A \setminus A^\prime)\rbrace\EEE$ we have \EEE
\begin{align}
\label{first equation of this page-we}
 \frac{1}{k}\mathrm{err}(u,v,E,F,A\setminus A^\prime) &=   {L}\frac{(\mathcal{L}^d(A\setminus A^\prime)+\mathcal{E}(u,E,A)+\mathcal{E}(v,F,B))}{k} 
  \leq  \MMM \frac{\eta}{2} \EEE \big(1+ \mathcal{E}(u,E,A)+\mathcal{E}(v,F,B)\big) .  
\end{align} 
Thus, by applying the previous inequality together with \MMM \eqref{fullenergyestimate}, \EEE \eqref{forlambda}, and  \eqref{first equation of this page}  we get 
\begin{align}
\label{finalestimateenegy}
\mathcal{E}(w_{i_0},D_{i_0},A^\prime \cup B) \leq  (1+\eta)(\mathcal{E}(u,E,A)+\mathcal{E}(v,F,B)) +\Lambda(u,v)+C \Vert \chi_E-\chi_F\Vert_{L^1(A \setminus A^\prime)}+\eta.  
\end{align}
 We set $w=w_{i_0}$, $D=D_{i_0}$ and verify \eqref{eqfundamentalestimate}.   \EEE First.   \eqref{eqfundamentalestimate-1} follows \MMM from \eqref{finalestimateenegy}. \EEE \MMM The properties in  \eqref{eqfundamentalestimate-4} have been observed below \eqref{def Di} and below \eqref{def wi}, respectively. To \EEE confirm \eqref{eqfundamentalestimate-2}--\eqref{eqfundamentalestimate-3}, we notice that   \NNN
 \begin{equation*}
  \Vert\min \{ \psi(|w-u|), \psi(|w-v|) \} \Vert_{L^1(A' \cup B)}\leq \mathcal{L}^d(S_{i_0}), \quad \quad 
  \Vert\min \{ |\chi_D-\chi_E|, |\chi_D-\chi_F| \} \Vert_{L^1(A' \cup B)}\leq \mathcal{L}^d(S_{i_0}).
 \end{equation*} \EEE
Due to  \eqref{error}, \MMM  \eqref{first equation of this page},  \eqref{first equation of this page-we}, and $L \ge 1$, \EEE  we have
\begin{equation*}
\mathcal{L}^d(S_{i_0})\leq \mathrm{err}(u,v,E,F,S_{i_0}) \leq (1+ \mathcal{E}(u,E,A)+\mathcal{E}(v,F,B))\eta.
\end{equation*}
Thus, \eqref{eqfundamentalestimate-2}--\eqref{eqfundamentalestimate-3} follow.  
\end{proof}

\MMM We proceed with two consequences. 

\begin{remark}[Fundamental estimate for sets]\label{remark.on.fund.est}
Inspection of the proof shows that $w = 0$ if $u=v=0$. In this case, the fundamental estimate essentially reduces to a version of sets, see   \cite[Lemma 4.4]{AmbBra90} for a similar statement. 

In the special case that $A'$, $A$, and $A' \cup B$ are concentric balls, and $\partial E \cap A$ and $\partial F \cap B$ are polyhedral boundaries, we observe the following:  one can choose $D$ such that also $\partial D \cap (A'\cup B)$ is a  polyhedral boundary. In fact, this can be achieved by choosing piecewise affine Lipschitz functions $\varphi_i$ in \eqref{def Di} whose sublevel sets have polyhedral boundary.  
\end{remark}

\begin{corollary}\label{set.corollary}
In the setting of  Proposition \ref{fundamentalestimate}, suppose that $E$ and $F$ disconnect $A$ and $B$, respectively, in the sense that there are sets $S^\pm_A$ and $S^\pm_B$ such that $S^+_A,S^-_A,E$ is a partition of $A$ and $S^+_B,S^-_B,F$ is a partition of $B$ (up to $\mathcal{L}^d$-negligible sets) satisfying $\mathcal{H}^{d-1}(\partial^* S^+_A \cap \partial^* S^-_A) = 0$ and $\mathcal{H}^{d-1}(\partial^* S^+_B \cap \partial^* S^-_B) = 0$. We  further assume that 
\begin{align}\label{nonintersectionassumption}
\mathcal{L}^d(S^+_A \cap S^-_B) = \mathcal{L}^d(S^-_A \cap S^+_B) = 0.
\end{align}
 Then, $D$ disconnects $A' \cup B$ in the sense that there exists a partition $S^+_D,S^-_D,D$ of $A' \cup B$ with  $\mathcal{H}^{d-1}(\partial^* S^+_D \cap \partial^* S^-_D) = 0$.
 \end{corollary}
 
 \begin{proof}
We recall the definition of $D=D_{i_0}$ in \eqref{def Di}, and choose $t \in (t_{i_0}^1,t_{i_0}^2)$ such that $K \defas \lbrace \varphi_{i_0} > t \rbrace$ satisfies $\mathcal{H}^{d-1}((\partial^* S^\pm_A \cup  \partial^* S^\pm_B) \cap \partial^*K ) = 0$. We define the sets 
$$S_D^\pm = \big((S_A^\pm \setminus D) \cap K\big) \cup \big((S_B^\pm \setminus D) \setminus K\big).$$
Using the fact that $E_{i_0} \subset D_{i_0}$ on $K^1$ and $F_{i_0} \subset D_{i_0}$ on $B \cap K^0$, see \eqref{def Di}, it is elementary to check that $S^+_D$, $S^-_D$, and $D$ forms a partition of $A' \cup B$, up to \NNN a \EEE set of $\mathcal{L}^d$-negligible measure, and $\mathcal{H}^{d-1}(\partial^* S^+_D\cap \partial^* S^-_D\cap K^0)=\mathcal{H}^{d-1}(\partial^* S^+_D\cap \partial^* S^-_D\cap K^1)=0$.  Moreover, using \cite[Theorem~3.61]{ambrosio2000fbv} we find that
\begin{align}\label{xXXXX}
 \mathcal{H}^{d-1}(\partial^* S^+_D \cap \partial^* S^-_D)& =\mathcal{H}^{d-1}(\partial^* S^+_D \cap \partial^* S^-_D\cap \partial^*K) \nonumber
\\&  =\mathcal{H}^{d-1}  \big(\big\{ x\in \partial^*K \colon x \in (S^+_A)^1 \cap (S^-_B)^1 \text{ or } x\in  (S^+_B)^1 \cap (S^-_A)^1 \big\}\big),
\end{align} 
where for the last term we used that $\mathcal{H}^{d-1}((\partial^* S^\pm_A \cup  \partial^* S^\pm_B) \cap \partial^*K ) = 0$. Finally, we notice that the last term  has  negligible $\mathcal{H}^{d-1}$-measure by  \eqref{nonintersectionassumption}.
 \end{proof}
 In Figure \ref{figure10} we provide an example in which \eqref{nonintersectionassumption} is not satisfied and the void $D$, provided by the fundamental estimate, indeed does not separate $A^\prime \cup B$.  
\begin{center}
\begin{figure}[htp]

\includegraphics[width=150mm]{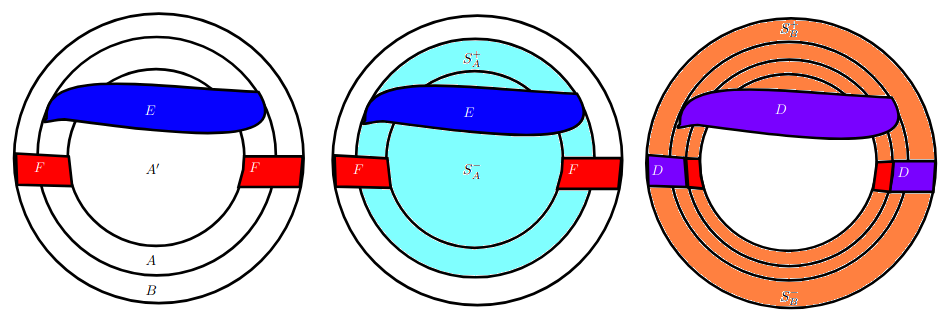}

\captionof{figure}{ Consider $A^\prime=B_{1-2\delta}$, $A=B_{1-\delta}$ and $B=B_1 \setminus B_{1-2\delta}$. In blue the void $E$, in red the void $F$, and in purple the void $D$. In light blue $S_A^+ \setminus F$ and $S^-_A \setminus F$, in orange $S^+_B\setminus E$ and $S^-_B\setminus E$. Notice in particular that $\mathcal{L}^2(S^-_A\cap S_B^+)>0$. }
\label{figure10}
\end{figure}
\end{center}
Following the lines of the proof of \cite[Lemma 4.3]{FriPerSol20a}, up to using Proposition \ref{fundamentalestimate} in place of \cite[Proposition 4.1]{FriPerSol20a}, \MMM we \EEE  can show the following \MMM lemma. \EEE 
\begin{lemma}[Properties of $\Gamma$-$\liminf\limits$ and $\Gamma$-$\limsup\limits$]
\label{propertiesofgammaliminflimsup}
Let $(\mathcal{E}_\varepsilon)_\eps$ be a sequence of energies of the \MMM form \EEE \eqref{newenergies}. Define, for every $A \in \mathcal{A} \MMM(\Omega)\EEE$, $E \in \mathcal{M}(\Omega)$, and $u  \MMM \in L^0(\Omega;\R^d) \EEE$, the \MMM $\Gamma$-$\liminf_{\varepsilon \to 0}\mathcal{E}_\varepsilon$ and $\Gamma$-$\limsup_{\varepsilon \to 0}\mathcal{E}_\varepsilon$ by \EEE
\begin{align*}
   \mathcal{E}^\prime(u,E,A)&\defas  \inf\Big\{ \liminf\limits_{\varepsilon \to 0}\mathcal{E}_\varepsilon(u_\varepsilon,E_\varepsilon,A)\colon u_\varepsilon\to u \:\: \text{in measure on}\:\: A, \:\: \chi_{E_\varepsilon}\to \chi_E \:\: \text{in}\:\: L^1(A)\Big\},    \\
   \mathcal{E}^{\prime \prime}(u,E,A)&\defas \inf\Big\{ \limsup\limits_{\varepsilon \to 0}\mathcal{E}_\varepsilon(u_\varepsilon,E_\varepsilon,A)\colon u_\varepsilon \to u \:\: \text{in measure on}\:\: A, \:\: \chi_{E_\varepsilon}\to \chi_E \:\: \text{in}\:\: L^1(A)\Big\},   
\end{align*}
\MMM respectively. \EEE  Then we have
\begin{subequations}
\label{e:lemma3.11}
\begin{align}
&  \mathcal{E}^\prime(u,E,A)\leq \mathcal{E}^\prime(u,E,B), \quad \mathcal{E}^{\prime \prime}(u,E,A)\leq \mathcal{E}^{\prime \prime}(u,E,B)\quad \text{whenever}\:\: A \subset B, \label{e:lemma3.11-1}
\\ & \alpha \mathcal{I}(u,E,A)\leq \mathcal{E}^\prime(u,E,A)\leq \mathcal{E}^{\prime \prime}(u,E,A)\leq \beta \mathcal{L}^d(A)+\beta \mathcal{I}(u,E,A), \label{e:lemma3.11-2}
\\ & \mathcal{E}^\prime(u,E,A)=\sup\nolimits_{B \subset \subset A}\mathcal{E}^\prime(u,\NNN E, \EEE B), \quad \mathcal{E}^{\prime \prime}(u,E,A)=\sup\nolimits_{B \subset \subset A}\mathcal{E}^{\prime \prime}(u,\NNN E, \EEE B)\quad \text{\NNN for}\:\: A \in \mathcal{A}(\Omega) \label{e:lemma3.11-3}
\\ &  \mathcal{E}^\prime(u,E,A\cup B)\leq \mathcal{E}^\prime(u,E,A)+\mathcal{E}^{\prime \prime}(u,E,B) \quad \text{whenever}\:\: A, B\in \mathcal{A}(\Omega)\label{e:lemma3.11-4}
\\& \mathcal{E}^{\prime \prime}(u,E,A\cup B)\leq \mathcal{E}^{\prime \prime}(u,E,A)+\mathcal{E}^{\prime \prime}(u,E,B)\quad \text{whenever}\:\: A, B\in \mathcal{A}(\Omega), \label{e:lemma3.11-5}
\end{align}
\end{subequations}
 where $\alpha$ and $\beta$ are the  constants of the growth conditions $(f_2)$--$(f_3)$, $(g_2)$--$(g_3)$, \MMM and where \EEE for brevity we write $\mathcal{I}(u,E,A)\defas \Vert e(u)\Vert_{L^p(A)}^p+\mathcal{H}^{d-1}(\partial^* E \cap A)+2\mathcal{H}^{d-1}(J_u \cap A)$.
\end{lemma}

\begin{proof} \MMM
The proof is identical to the one of \cite[Lemma 4.3]{FriPerSol20a}, up to using the fundamental estimate for functionals defined  on pairs of function-set. Moreover, in the lower bound in \eqref{e:lemma3.11-2}, in place of \cite[Theorem~3.7]{FriPerSol20a}, we use the relaxation result  in \cite[Proposition 2.1]{Crismale_2020}. \EEE
\end{proof}
 
\MMM We close this section with the proof of Theorem \ref{first gamma convergence result}. \EEE
\begin{proof}[Proof of Theorem \ref{first gamma convergence result}]
 Compactness and integral representation are carried \MMM out \EEE by using the localization method of $\Gamma$-convergence (see \cite[Chapters 14--16, 18--20]{DalMaso:93}). In particular, the proof follows the same steps of \cite[Theorem 2.1]{FriPerSol20a} up to using Lemma \ref{propertiesofgammaliminflimsup} in place of \cite[Lemma 4.3]{FriPerSol20a} and Theorem \ref{Theorem2.1 CriFrieSol} in place of \cite[Theorem 2.1]{SolFriCri20}.   
\end{proof}

\section{\MMM Bulk and surface densities: \EEE Proof of Proposition \ref{limsupliminf}} \label{section proof of proposition limsupliminf}

\MMM In this section we prove Proposition \ref{limsupliminf}. \EEE The  identities   \eqref{f1 = f2} and \eqref{g1 = g2} \MMM for a suitable subsequence are well known in the literature and we simply \EEE refer to \cite[Proposition 2.3 and Section 3.2]{FriPerSol20a}.   (Note that the minimization problems \eqref{simplified_cellformula_voids} and \cite[(3.16)]{FriPerSol20a} can be naturally identified.)  

 From now on, we focus on \EEE    \eqref{h1 = h2} and \eqref{2g=h}. \MMM As observed already in Remark \ref{h-remark}, it suffices to prove
 \begin{align}\label{h-remark2}
h' \ge 2\hat{g}, \quad \quad  2\hat{g} \ge h'',
 \end{align}
 where $h'$  and $h''$ are defined in \eqref{h-remark3}, and \EEE we have chosen a   subsequence such that \eqref{g1 = g2} holds \MMM with the corresponding density denoted by  \EEE $\hat{g}$.

\MMM Recalling the notation in \eqref{thin layer},  \EEE for every $x \in \mathbb{R}^d$, $\zeta \in \mathbb{R}^{d}\setminus \{0\}$, and $\nu \in \mathbb{S}^{d-1}$,  we consider the function
\begin{equation}
\label{function thin layer}
    \eta_{x,\zeta,\varepsilon}^{\nu} \QQQ (y) \EEE \defas \begin{cases}
        \zeta \:\: \text{if}\:\: y \in \MMM B^{\nu,+}_1 \EEE (x) \setminus E^{x,\nu}_{\varepsilon}, \\
        0 \:\: \text{otherwise}.
    \end{cases}
\end{equation}
We \MMM note \EEE that $(\eta_{x,\zeta,\varepsilon}^{\nu},E^{x,\nu}_{\varepsilon}) \to (u^\nu_{x,\zeta},\emptyset)$ in measure on $\MMM B_{1} \EEE (x)$ as $\varepsilon \to 0$. \EEE   \MMM We proceed in two steps showing separately  the two inequalities in \eqref{h-remark2}. \EEE  We highlight that the  assumption that  \MMM  $g_{\varepsilon}$  is continuous for every $\varepsilon$ \EEE is only needed for the \MMM inequality $2\hat{g} \ge h''$. Fix $x \in \mathbb{R}^d$ and $\nu \in \mathbb{S}^{d-1}$.  \EEE

\emph{Step 1: $h^\prime (x,\nu) \geq 2\hat{g}(x,\nu)$}.
\MMM Fix $\rho>0$. \EEE Up to extracting a further subsequence (not relabeled), we can also assume that 
\begin{equation}
\label{h0>2g0 eq 1}
  \liminf\limits_{\varepsilon \to 0}m^\mathrm{jump}_{\mathcal{E}_\varepsilon}(u_{x,e_1}^\nu,B_\rho(x))=\lim\limits_{\varepsilon \to 0}m^\mathrm{jump}_{\mathcal{E}_\varepsilon}(u_{x,e_1}^\nu,B_\rho(x)).
\end{equation}
Let $(u_\varepsilon,E_\varepsilon)_\varepsilon$ be such that $(u_\varepsilon,E_\varepsilon)\in \mathcal{D}^{\nu}_{\varepsilon}(B_{\rho}(x))$ (recall \MMM \eqref{set of competitors} \NNN and see Figure~\ref{figure1} for an illustration) \EEE and
\begin{equation}
\label{h0>2g0 eq 2}
\lim\limits_{\varepsilon \to 0}\mathcal{E}_\varepsilon(u_\varepsilon,E_\varepsilon,B_\rho(x))=\lim\limits_{\varepsilon \to 0}m^\mathrm{jump}_{\mathcal{E}_\varepsilon}(u_{x,e_1}^\nu,B_\rho(x)).
\end{equation}
Let $S^+_\varepsilon$ and $S^-_\varepsilon$ denote the two regions disconnected by \MMM $E_\varepsilon$, \EEE and \MMM note that by definition \EEE $u_\varepsilon=e_1 \chi_{S^+_\varepsilon}$. 
Let $\delta \in (0,\MMM \rho^2)\EEE$. \MMM Recalling \eqref{function thin layer}, \EEE we extend $(u_\varepsilon,E_\varepsilon)$ \MMM to \EEE $B_{\rho+\delta}(x)$ by setting $(u_\varepsilon,{E_\varepsilon})=(\eta_{x,e_1,\varepsilon}^\nu,{E_\varepsilon^{x,\nu}})$ on $B_{\rho+\delta}(x) \MMM \setminus B_{\rho}(x) \EEE $.  \MMM Accordingly, we extend $S^\pm_\varepsilon$ by $ S^\pm_\varepsilon = B^{\nu,\pm}_{\rho+\delta} \NNN (x) \EEE \setminus  E_\varepsilon^{x,\nu} $ on $B_{\rho+\delta}(x)  \setminus B_{\rho}(x)  $. \EEE If 
\begin{align}\label{choici}
\int_{\partial^* S^+_\varepsilon \cap B_{\rho+\delta}(x)}g_\varepsilon( \QQQ y \EEE ,\nu_{S^+_\varepsilon})\, \mathrm{d}\mathcal{H}^{d-1}  \NNN (y) \EEE \leq \int_{\partial^* S^-_\varepsilon \cap B_{\rho+\delta}(x)}g_\varepsilon(\QQQ y \EEE ,\nu_{S^-_\varepsilon})\, \mathrm{d}\mathcal{H}^{d-1}  \NNN (y), \EEE
\end{align}
we define $\hat{E}_\varepsilon =  {E}_\varepsilon \cup S_\eps^-$ on $B_{\rho+\delta}(x)$. Otherwise, we let $\hat{E}_\varepsilon =   \MMM  S_\eps^- $. In both cases, notice that by and $(f_1)$,   $(g_1)$, $(g_3)$,  and \eqref{h0>2g0 eq 2}  it holds
\begin{equation}
\label{h0>2g0 eq 3}
\lim\limits_{\varepsilon \to 0}m^\mathrm{jump}_{\mathcal{E}_\varepsilon}(u_{x,e_1}^\nu,B_\rho(x))+\MMM c_d\delta \rho^{d-2} \EEE \geq \MMM \liminf_{\varepsilon \to 0} \EEE    \mathcal{E}_\varepsilon(u_\varepsilon,E_\varepsilon,B_{\rho+\delta}(x))\geq \MMM \liminf_{\varepsilon \to 0} \EEE  2\mathcal{E}_\varepsilon(\MMM 0, \EEE \hat{E}_\varepsilon,B_{\rho+\delta}(x)),
\end{equation}
where $c_d>0$ denotes a positive constant depending only on the dimension $d$ \MMM and on $\beta$. Here, in the last step we exploited the alternative choice in \eqref{choici} and the corresponding definition of $\hat{E}_{\varepsilon}$. By a compactness result for sets of finite perimeter,  $\chi_{\hat{E}_\varepsilon}$ converges  in $L^1(B_{\rho+\delta}(x))$. In particular, by construction, ${\hat{E}_\varepsilon}$ converges to ${\Pi^{\nu,-}_x}$ on $B_{\rho+\delta}(x)\setminus B_\rho(x)$. We apply Proposition \ref{fundamentalestimate} and Remark \ref{remark.on.fund.est} with $A=B_{\rho+\frac{\delta}{2}}(x)$, $A^\prime = B_\rho(x)$, $B=B_{\rho+\delta}(x)\setminus \overline{B_\rho(x)}$, $(\MMM 0, \EEE \hat{E}_\varepsilon)$ (in place of \MMM $(u,E)$), \EEE  $(\MMM 0, \EEE \Pi^{\nu,-}_x)$ (in place of \MMM $(v,F)$), \MMM and $\eta>0$. \EEE Thus, we find $D_\varepsilon \MMM \in \mathcal{P}(B_{\rho+\delta}(x)) \EEE $ such that $ D_\varepsilon= {\Pi^{\nu,-}_x}$  \MMM near  $\partial B_{\rho+\delta}(x)$ \EEE and
\begin{equation*}
\MMM \liminf_{\eps \to 0} \EEE \mathcal{E}_\varepsilon(\MMM 0, \EEE D_\varepsilon,B_{\rho+\delta}(x))\leq \MMM (1+\eta) \liminf_{\eps \to 0} \EEE  \mathcal{E}_\varepsilon(\MMM 0, \EEE \hat{E}_\varepsilon,B_{\rho+\delta}(x))+ \MMM c_d \delta \rho^{d-2}  + \eta \EEE.    
\end{equation*}
Recalling \eqref{simplified_cellformula_voids},  using the \MMM previous inequality along with    \eqref{h0>2g0 eq 3} and sending    $\eta \to 0$ \EEE we get
\begin{equation*}
   \MMM \lim\limits_{\varepsilon \to 0} \EEE m^\mathrm{jump}_{\mathcal{E}_\varepsilon}(u_{x,e_1}^\nu,B_\rho(x))+ \NNN 3 \EEE  c_d\delta \rho^{d-2} \EEE \geq \MMM \liminf_{\eps \to 0} \EEE  2\mathcal{E}_\varepsilon( \MMM 0, \EEE D_\varepsilon,B_{\rho+\delta}(x))\geq \NNN 2 \EEE  \liminf_{\eps \to 0} \EEE m^\mathrm{voids}_{\mathcal{E}_\varepsilon}(\Pi^{\nu,-}_x,B_{\rho+\delta}(x)).
\end{equation*}
By dividing both sides by $\gamma_{d-1}(\rho+\delta)^{d-1}$, recalling $\MMM\delta<\rho^2\EEE$, and by passing to the limit $\rho+\delta \to 0$,  we  \MMM get   $h^{\prime  }(x,\nu)\geq 2\hat{g}(x,\nu)$. \EEE

\noindent
\emph{Step 2: $2\hat{g}(x,\nu)\geq h^{\prime \prime}(x,\nu)$.} 
Fix $\rho>0$.  \NNN In view of \EEE \eqref{simplified_cellformula_voids}, let $E_{\varepsilon} \in \mathcal{P}(\Omega)$ be \EEE such that ${E_{\varepsilon}}={\Pi^{\nu,-}_{x}}$ on a neighborhood of $\partial B_{\rho}(x)$ \EEE and  
\begin{equation}
\label{equation 7.45 old}
 \MMM \limsup_{\eps \to 0} \EEE\mathcal{E}_\varepsilon( \MMM 0, \EEE E_{\varepsilon},B_{\rho}(x))\leq  \MMM \limsup_{\eps \to 0} \EEE m^{\mathrm{void}}_{\mathcal{E}_{\varepsilon}}(\Pi^{\nu,-}_{x},B_{\rho}(x)).
\end{equation}
\NNN Let $\delta \in (0,\MMM \rho^2)\EEE$. \EEE We extend $E_{\varepsilon}$ on $B_{\NNN \rho + \delta}(x)$ by setting $E_{\varepsilon}=\Pi^{\nu,-}_{x}$ on $B_{\NNN \rho + \delta}(x)\setminus B_{\rho}(x)$. Thus, by virtue of \eqref{equation 7.45 old} and $(g_3)$, we have
\begin{equation}
\label{equation 7.45}
\limsup\limits_{\varepsilon \to 0}\mathcal{E}_{\varepsilon}(0,E_{\varepsilon},B_{\NNN \rho + \delta}(x))\leq \limsup_{\eps \to 0} \EEE m^{\mathrm{void}}_{\mathcal{E}_{\varepsilon}}(\Pi^{\nu,-}_{x},B_{\rho}(x))+ c_d \delta \NNN \rho^{d-2} \EEE 
\end{equation}
for a constant $c_d>0$ depending only on the dimension $d$ \NNN and on $\beta$. \EEE
By  \cite[Theorem 2.1 and Corollary~2.4]{BraConGar16}, \MMM  the continuity of $g_\eps$,   and a suitable diagonal argument, we can approximate $E_\varepsilon$ with open sets $R_\varepsilon$ such that
\begin{subequations}
    \label{polyhedral approximation}
\begin{align}
   &  \vphantom{\int} R_\varepsilon\:\: \text{has polyhedral boundary},  \label{polyhedral approximation-1}
    \\ &  \MMM \int_{\partial^* R_\varepsilon \cap B_{\NNN \rho + \delta}(x)}g_\varepsilon(z,\nu_{R_\varepsilon })\, \mathrm{d}\mathcal{H}^{d-1}(z)\leq \int_{\partial^* E_\varepsilon \cap B_{\NNN \rho + \delta}(x)}g_\varepsilon(z,\nu_{E_\varepsilon })\, \mathrm{d}\mathcal{H}^{d-1}(z) +  \varepsilon,   \label{polyhedral approximation-2} \EEE 
    \\ & \vphantom{\int} 
   \Vert \chi_{R_\varepsilon}-\chi_{\Pi^{\nu,-}_{x}}\Vert_{L^1(B_{\NNN \rho + \delta}(x)\setminus B_{\rho}(x)}) \le \varepsilon.  \label{polyhedral approximation-3}
\end{align}
\end{subequations}
\NNN 
Below we apply the fundamental estimate on \EEE the sets  as $A= B_{\NNN \rho + \frac{\delta}{2}}(x)$, $A^\prime = B_{\rho}(x)$, and  $B=B_{\NNN \rho + \delta}(x)\setminus \overline{A'}$. \EEE  By applying $(g_3)$ and  \MMM Remark \ref{remark.on.fund.est}   \EEE to the sets $R_\varepsilon$ and $\Pi^{\nu,-}_x$,  we obtain a set of finite perimeter $P_\varepsilon$ such that $P_\varepsilon$ has still polyhedral boundary, ${P_\varepsilon}= {\Pi^{\nu,-}_x}$ \NNN on  $  B_{ \rho + \delta}(x) \setminus B_{ \rho + \frac{\delta}{2}}(x)$, \EEE and using  \eqref{polyhedral approximation-2}--\eqref{polyhedral approximation-3}
\begin{align}
\label{amb bra fund est}
&\limsup_{\eps \to 0}  \hspace{-0.1cm}  \int_{\partial^* P_\varepsilon \cap B_{\NNN \rho + \delta}(x)}g_\varepsilon(\cdot,\nu_{P_\varepsilon })\, \mathrm{d}\mathcal{H}^{d-1}  \leq \MMM \limsup_{\eps \to 0}   (1+ \eta) \Big(\EEE \int_{\partial^* E_\varepsilon \cap B_{\NNN \rho + \delta}(x)}g_\varepsilon(\cdot,\nu_{E_\varepsilon })\, \mathrm{d}\mathcal{H}^{d-1} \MMM   +   c_d  \NNN \rho^{d-2}   \delta \Big) \EEE   \MMM + \eta \EEE
\end{align}
\MMM for $\eta >0$. We now claim that, starting from $P_\varepsilon$, we can construct a void $D_\varepsilon$ disconnecting two separated regions $S_\varepsilon^+$ and $S_\varepsilon^-$ such that $(e_1\chi_{S_\varepsilon^+},D_\varepsilon)$ belongs to $\mathcal{D}^{\nu}_{\varepsilon}(B_{\NNN \rho + \delta}(x))$ and
\begin{equation}
\label{claim 2g=h}
   \limsup\limits_{\varepsilon \to 0}\int_{\partial^* D_\varepsilon \cap B_{\NNN \rho + \delta}(x)}g_\varepsilon(\NNN \cdot, \EEE \nu_{D_\varepsilon})\, \mathrm{d}\mathcal{H}^{d-1}\leq \limsup\limits_{\varepsilon \to 0}\int_{\partial^*  P_\varepsilon \cap B_{\NNN \rho + \delta}(x)} 2g_\varepsilon(\NNN \cdot, \EEE \nu_{  P_\varepsilon})\, \mathrm{d}\mathcal{H}^{d-1} +\NNN c_d  \delta\rho^{d-2}\EEE .
\end{equation}
Since $\int_{\partial^* D_\varepsilon \cap B_{\NNN \rho + \delta}(x)}g_\varepsilon(x,\nu_{D_\varepsilon})\, \mathrm{d}\mathcal{H}^{d-1}\geq m^{\mathrm{jump}}_{\mathcal{E}_\varepsilon}(u^\nu_{x,e_1},  B_{\NNN \rho + \delta} \EEE (x))$, \MMM this  concludes the proof  by \eqref{equation 7.45} \NNN and \eqref{amb bra fund est}, \EEE where as  \NNN in Step 1 we first pass to the limit $\eta \to 0$ and then after division by $ \gamma_{d-1} (\rho+\delta)^{d-1}$ we pass to the limsup in $\rho \to 0$ using $\delta < \rho^2$. \EEE

We now come to the proof of  \eqref{claim 2g=h}. \EEE
To simplify the notation, let us \MMM set \EEE $\Gamma_\varepsilon \defas \partial^* P_\varepsilon \cap B_{\NNN \rho + \delta}(x)$. Up to an \MMM arbitrarily \EEE small error \MMM in the energy, \EEE it is not restrictive to assume \MMM that \EEE $\vert \nu_{\Gamma_\varepsilon} \cdot \nu \vert \neq 0$ \MMM a.e., \EEE where we denote \MMM by \EEE $\nu_{\Gamma_{\varepsilon}}$ the normal to $\Gamma_{\varepsilon}$.   We denote by \MMM $\sigma_\varepsilon$ \EEE the modulus of continuity of $g_\varepsilon$ and by $M_\varepsilon $ the number of \MMM  $(d-2)$-dimensional edges \EEE of $\Gamma_\varepsilon$. Let $(\theta_\varepsilon)_\varepsilon$ be a sequence such that \MMM $\theta_\eps \le \eps$ and 
\begin{equation}
\label{good continuity}
{\rm (i)} \ \ 
    \limsup\limits_{\varepsilon \to 0}\sigma_\varepsilon(\theta_\varepsilon)=0\quad\text{and}\quad \quad {\rm (ii)} \ \  \limsup\limits_{\varepsilon \to 0}M_\varepsilon \theta_\varepsilon=0.
\end{equation}
We define 
\begin{equation}\label{eq: fdef}
    F_\varepsilon\defas \big\{ z \in B_{\NNN \rho + \delta}(x)\colon z=y+t\nu \:\: \text{for}\:\: y \in \Gamma_\varepsilon\:\: \text{and}\:\: t \in \NNN (-\theta_\eps,\theta_\varepsilon) \EEE \big\}.
\end{equation}
\begin{figure}
    \centering
    \begin{tikzpicture} 
\tikzset{x=1.3ex,y=1.3ex} 

\draw[line width=0.35pt, color=color0, ] (56.12, 46.60) -- (31.09, 46.60);
\draw[line width=0.35pt, color=color0, ] (56.12, 46.60) -- (31.09, 23.35);
\draw[line width=0.35pt, color=color0, ] (56.12, 23.35) -- (31.09, 23.35);
\draw[line width=0.35pt, color=color0, ] (56.12, 48.15) -- (31.09, 48.15);
\draw[line width=0.35pt, color=color0, ] (56.12, 48.15) -- (31.09, 24.90);
\draw[line width=0.35pt, color=color0, ] (56.12, 24.90) -- (31.09, 24.90);
\draw[line width=0.35pt, color=color2, ] (31.09, 46.60) -- (31.03, 48.18);
\draw[line width=0.35pt, color=color2, ] (56.12, 46.60) -- (56.14, 48.16);
\draw[line width=0.35pt, color=color2, ] (31.03, 25.02) -- (31.09, 23.35);
\draw[line width=0.35pt, color=color2, ] (56.12, 23.35) -- (56.15, 24.95);
\node[color=color2, anchor=north west, rotate=0.00, xscale=0.9, yscale=0.9, xslant=0.000, yslant=0.000] at (28.04, 25.08) {$2\theta_\eps$};

\node[color=color2, anchor=north west, rotate=0.00, xscale=0.9, yscale=0.9, xslant=0.000, yslant=0.000] at (28.03, 48.57) {$2\theta_\eps$};

\node[color=color2, anchor=north west, rotate=0.00, xscale=0.9, yscale=0.9, xslant=0.000, yslant=0.000] at (56.23, 48.59) {$2\theta_\eps$};

\node[color=color2, anchor=north west, rotate=0.00, xscale=0.9, yscale=0.9, xslant=0.000, yslant=0.000] at (56.23, 25.11) {$2\theta_\eps$};

\end{tikzpicture} 
 \captionof{figure}{An example in which shifting a curve by $\theta_\eps$ creates an additional perimeter of $\NNN 8 \EEE \theta_\eps$.}
    \label{figure2}
   \end{figure}
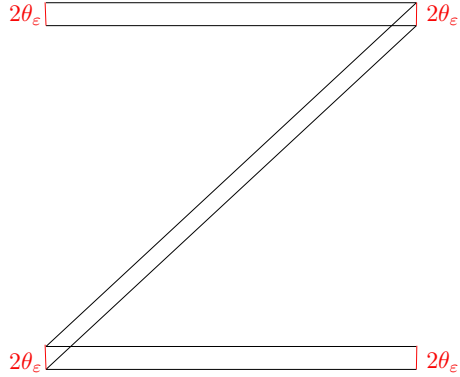
 Using the continuity of $g_\eps$ \NNN and $(g_1)$, \EEE we can check that 
\begin{align}
\label{last equation 2g=h000}
 \int_{\partial^* F_\varepsilon \cap B_{\NNN \rho + \delta}(x)}g_\varepsilon(z,\nu_{F_\varepsilon })\, \mathrm{d}\mathcal{H}^{d-1}&\leq 2\int_{\Gamma_{\varepsilon}}g_\varepsilon(z,\nu_{\Gamma_{\varepsilon} })\, \mathrm{d}\mathcal{H}^{d-1}   + \NNN 2 \EEE \sigma_\varepsilon(\theta_\varepsilon)\mathcal{H}^{d-1}(\Gamma_{\varepsilon}) + CM_\varepsilon \theta_\varepsilon,
\end{align} 
where \MMM $C$  \EEE is a positive constant depending only on the dimension \MMM and on $\rho$, cf.\ Figure \ref{figure2}. 

Let $T_\varepsilon^-=P_\varepsilon$ and $T_\varepsilon^+=B_{\NNN \rho + \delta}(x)\setminus (P_\varepsilon \cup F_\varepsilon)$. One can check that $(e_1 \chi_{T_\varepsilon^+},F_\varepsilon)$ belongs to $\NNN \mathcal{D}^{\nu}_{\theta_\varepsilon} \EEE (B_{\NNN \rho + \delta}(x))$. \MMM Here, we essentially exploit the fact that $\vert \nu_{\Gamma_\varepsilon} \cdot \nu \vert \neq 0$ almost everywhere. \EEE  In particular, using that $\theta_\eps \le \eps$, we obtain \EEE
$
    \Vert \chi_{F_\varepsilon}- \MMM   \chi_{ E^{x,\nu}_{\varepsilon}} \EEE \Vert_{L^1(B_{\NNN \rho + \delta}(x)\setminus B_{\rho}(x))}\leq c\varepsilon
$
\MMM for a universal constant $c>0$. \NNN Now we use the fundamental estimate on the sets  $\hat{A}= B_{\rho + \frac{3\delta}{4}}(x)$, $\hat{A}^\prime = B_{\rho+ \frac{\delta}{2}}(x)$, and  $\hat{B}=B_{\rho + \delta}(x)\setminus {\hat{A}'}$: \EEE  By applying Remark~\ref{remark.on.fund.est} and  Corollary~\ref{set.corollary} on   $F_\eps$ and $E^{x,\nu}_{\varepsilon}$     we modify $F_\varepsilon$ to obtain a void set $D_\varepsilon$ such that $D_\varepsilon$ \MMM disconnects $B_{\NNN \rho + \delta}(x)$,  $D_\eps = E^{x,\nu}_{\varepsilon}$ near $\partial B_{\NNN \rho + \delta}(x)$, \EEE and
\begin{align}
\label{last equation 2g=h}
\limsup_{\eps \to 0}  &\int_{\partial^* D_\varepsilon \cap B_{\NNN \rho + \delta}(x)}g_\varepsilon(\NNN \cdot , \EEE \nu_{D_\varepsilon })\, \mathrm{d}\mathcal{H}^{d-1} \nonumber
\\ \leq 
&(1+\eta) \limsup_{\eps \to 0}  \Big( \int_{\partial^* F_\varepsilon \cap B_{\NNN \rho + \delta}(x)}g_\varepsilon(\NNN \cdot , \EEE \nu_{F_\varepsilon })\, \mathrm{d}\mathcal{H}^{d-1} +\NNN c_d  \delta\rho^{d-2}\EEE  \Big)   + \eta,
\end{align} 
where \MMM as before we  use $(g_3)$ and recall that $c_d$ also depends on $\beta$.  Here, we particularly observe  that condition \eqref{nonintersectionassumption} is satisfied since   $F_\eps = \Pi^\nu_x + \nu ( \NNN -\theta_\eps, \EEE \theta_\eps)$ and $E^{x,\nu}_{\varepsilon} = \Pi^\nu_x + \nu (-\eps,\eps)$  \NNN on $B_{\rho+\delta}(x) \setminus B_{\rho+\frac{\delta}{2}}(x)$.   We  \MMM choose $S_\varepsilon^\pm$ such that $S^+_\eps$, $S^-_\eps$, and  $D_\eps$ form a partition of $B_{\NNN \rho + \delta}(x)$ up to a set of negligible \NNN $\mathcal{L}^{d}$-measure.  \EEE We observe that $(e_1 \MMM \chi_{S_\varepsilon^+}  , D_\varepsilon)\in \mathcal{D}^{\nu}_{\varepsilon}(B_{\NNN \rho + \delta}(x))$.  Eventually,  \MMM combining  \NNN   \eqref{good continuity} \EEE and \eqref{last equation 2g=h000},  sending $\eps \to 0$,  \NNN and using a suitable diagonal argument for $\eta \to 0 $, \EEE  we get \eqref{claim 2g=h}.    \eop
 
  \EEE

%
%

\MMM 
 \begin{remark}\label{periodic-homo-remark}
 In case that   $g$ is periodic of period $1$ with respect to the coordinates, and $g_\eps$ is defined as $g_\varepsilon(x,\nu)\defas g(\frac{x}{\varepsilon},\nu)$, Step 2 can be performed without the continuity assumption in the $x$-variable. Let us sketch the argument. Due to the continuity of $\hat{g}$ in $\nu$, \NNN see \cite[Theorem 5.11]{ambrosio2000fbv}, \EEE it suffices to show the inequality $2\hat{g}(x,\nu)\geq h^{\prime \prime}(x,\nu)$ for $x \in \Omega$ and $\nu \in \mathbb{S}^{d-1} \cap \mathbb{Q}^d$. \NNN Consider again a sequence of competitors $(E_\eps)_\eps$ as in \eqref{equation 7.45 old}. Due to the periodicity, in view of \cite[Proposition 6.1]{BraDefVit96}, we can assume that 
$E_\eps \subset  \lbrace x\in B_{\rho+\delta}(x) \colon  | x \cdot \nu | \le C\eps\rbrace$ for some $C>0$ large enough.   
  The definition of $F_\eps$ in \eqref{eq: fdef} is replaced by the set in $B_{\rho+\delta}(x)$ which is enclosed by the two boundaries
  $$
 \Gamma_\eps^+ \defas  K\eps \nu +  ( \partial^* E_\varepsilon \cap B_{\rho+\delta}(x)), \quad \quad  \Gamma_\eps^- \defas - K\eps \nu +  ( \partial^* E_\varepsilon \cap B_{\rho+\delta}(x)),
$$
where $K \in \N$ is chosen  such that $K \nu \in \mathbb{Z}^d$ \NNN and $K \gg C$.  Here, we note that $K \gg C$ ensures that $ \Gamma_\eps^+ \cap  \Gamma_\eps^- = \emptyset$ and thus $F_\eps$ is well defined. \EEE In particular, note that we start from the original boundary $E_\eps$ and do not perform a polyhedral approximation which would require continuity of $g$ in $x$. The periodicity of $g$ along with $K\nu \in \Z^d$ allows to show an inequality of the form \eqref{last equation 2g=h000} without assuming continuity for $g$ in the $x$-variable. \NNN Here, we also observe that,  due to different construction of $F_\eps$, the lateral contribution (term $CM_\eps \theta_\eps$ in \eqref{last equation 2g=h000}) does not appear. \EEE 
  \end{remark}
  \EEE

\section{\MMM Identification of the $\Gamma$-limit: Proof of Theorem \ref{Identification of Gamma-limit} }
\label{section proof 2}
In this section we present the proof of Theorem \ref{Identification of Gamma-limit}. We highlight that it is sufficient to show \MMM  \eqref{identification gamma limit volume}\EEE--\eqref{h0 pc}.  Indeed,   \eqref{identification gamma limit jump parts}   follows directly from \eqref{h0 pc} and \eqref{2g=h}. \MMM The identification of the bulk density in \eqref{identification gamma limit volume} is very similar to the corresponding estimates in  \cite[Section~6.1]{SolFriCri20}, again  regarding $u$ in the pair $(u,E)$ as an $GSBD^p$-function. For convenience of the reader, we \NNN give \EEE  the proof in Appendix \ref{otherproof-sec}  below, including all adaptations to the present setting. In this section, we address the identification of \NNN the \EEE surface densities \eqref{identification gamma limit voids}--\eqref{h0 pc} where nontrivial adaptations to \cite{SolFriCri20} are in order. We first prove an auxiliary approximation result and then come to the proof of Theorem \ref{Identification of Gamma-limit}. \NNN Eventually, we give the proof of the relaxation result stated in Theorem \ref{relaxation}. \EEE

\subsection{An approximation result for void sets}\label{approxi-void-sec}

We state   the following variant \QQQ of \MMM \cite[Lemma 5.2]{FriPerSol20a}. \MMM We highlight again that this is the only \QQQ reason \MMM  why in Theorem \ref{Identification of Gamma-limit}(iii),(iv) we need to restrict ourselves \NNN to $d = 2$ and $p \ge 2$.  \EEE 
\begin{lemma}
\label{Lemma 5.2}
Let $d = 2$ and $p\geq 2$. Let $\zeta \in \mathbb{R}^2 \setminus \{0\}$ and $\nu \in \mathbb{S}^1$. Let   $(u_n,E_n)_n \subset \mathcal{W}^{1,p}(\MMM B_1)\EEE$ \NNN satisfy \EEE 
\begin{subequations}
\label{hp-Lemma 5.1}
\begin{align}
    & \ \ \lim\limits_{n \to \infty}\Vert e(u_n)\Vert_{L^p(B_{1})}=0, \label{hp-Lemma 5.1-1}
    \\ & \ \  \sup_{n \in \mathbb{N}}\mathcal{H}^{1}( \partial^* E_n\cap B_1)<\infty, \label{hp-Lemma 5.1-2}
    \\ & \ \  u_n \to u^\nu_{0,\zeta}, \:\: \chi_{E_n}\to 0\:\: \text{in measure on}\:\: B_1 \:\: \text{as}\:\: n \to \infty. \label{hp-Lemma 5.1-3}
\end{align}
\end{subequations}
  Let $\theta \in (0,\frac{1}{2}]$ and let $N_{\theta}\defas B_{1}\setminus B_{1-\theta}$. \EEE 
Then, there exists   a  \NNN universal \EEE constant $C>0$,  \EEE a sequence of voids sets $ (D_{n,\theta})_n$, and pairwise disjoint sets $ S^+_{n,\theta}$ and $  S^{-}_{n,\theta}$ with $D_{n,\theta} \cap S^\pm_{n,\theta}=\emptyset$, $ \mathcal{L}^2((B_1 \setminus D_{n,\theta}) \setminus (S^+_{n,\theta}\cup S^-_{n,\theta}))=0$ for all $n \in \mathbb{N}$ such that 
\begin{subequations}
\label{statement lemma 5.2}
\begin{align}
&   \lim\limits_{n \to \infty}\mathcal{L}^2(S^\pm_{n,\theta}\triangle B^{\nu,\pm}_1)  \NNN \le C\theta, \EEE \label{statement lemma 5.2-1}
\\ &   (B^{\nu,\pm}_1 \setminus
 {D}_{n,\theta}) \cap  N_\theta  \NNN  =  S_{n,\theta}^\pm \cap N_\theta\EEE \:\: \text{for all}\:\: n \in \mathbb{N}, \label{statement lemma 5.2-2} 
\\  & \vphantom{\lim_{n}}  \sup_{n \in \mathbb{N}}\mathcal{H}^1(\partial^* D_{n,\theta} \setminus \partial^* E_n)\leq C\theta, \quad \MMM \lim\limits_{n \to \infty}\mathcal{L}^2(D_{n,\theta} ) = 0 , \EEE \label{statement lemma 5.2-3}
\\ &    \vphantom{\lim_{n}} D_{n,\theta} \:\: \text{disconnects}\:\: S^+_{n,\theta} \:\: \text{and}\:\: S^-_{n,\theta}, \:\: \text{i.e.,}\:\: \MMM \mathcal{H}^{1}(\partial^* S^+_{n,\theta} \cap \partial^* S^-_{n,\theta}) \EEE =0. \label{statement lemma 5.2-4}
\end{align}
\end{subequations}
\end{lemma}
\begin{proof}
The proof relies on an adaptation of \cite[Lemma 5.2]{FriPerSol20a}. For this reason, we do not provide all the details and we mostly highlight the \MMM necessary \EEE changes.

\noindent \emph{Step 1 (Construction with no voids)}. 
The goal of this step is to show that \QQQ $\partial^* E_n\cap B_1$ \EEE can be approximated with a 1-rectifiable set separating $B_1$ \MMM into \EEE two disjoint regions. \MMM By regarding $(u_n)_n$ as $GSBD^p$-functions and by \EEE applying \cite[Proposition 3.4]{FriPerSol20a} to the sequence $(u_n)_n$, with $\theta^2$ instead of $\theta$, we find  Caccioppoli partitions $\MMM \bigcup^{J_n}_{j=1} \EEE P^n_j\cup R_n=B_1$ and rigid motions \NNN $(a^n_j)^{J_n}_{j=1}$ \EEE such that
\begin{subequations}
\label{approximation with pr functions lemma}
\begin{align}
   &  \mathcal{H}^1\big((\partial^* R_n\cap B_1)\setminus \partial^* E_n\big)+\sum\nolimits^{J_n}_{j=1}\mathcal{H}^1\big((\partial^* P^n_j\cap B_1)\setminus \partial^* E_n\big)\leq C_0 \theta^2, \label{approximation with pr functions lemma-1}
   \\ & \mathcal{L}^2(R_n)\leq C^2_0 \theta^2, \:\: \mathcal{L}^2(P_j^n)\geq \theta^6 \:\: \text{for all}\:\: j=1,\ldots,J_n, \label{approximation with pr functions lemma-2}
   \\ &  \max_{1\leq j\leq J_n}\Vert u_n -a_j^n\Vert_{L^\infty(P_j^n)}\leq C_{\theta^2}\Vert e(u_n)\Vert_{L^2(B_1)}, \label{approximation with pr functions lemma-3}
\end{align}
\end{subequations}
 where $C_0\defas \sup_{n \in \mathbb{N}}\mathcal{H}^{1}(\partial^* E_n)+\mathcal{H}^1(\partial B_1)$ and $C_{\theta^2}>0$ depends only on $\theta$. 

Using \eqref{approximation with pr functions lemma-3} and arguing like in {Step 1} of the proof of \cite[Proposition 3.4]{FriPerSol20a} one can show that
\begin{equation}
\label{separation of components}
   \mathcal{L}^2( P_j^n \cap B^-_n)=0 \:\: \text{or}\:\: \mathcal{L}^2(P_j^n\cap B^+_n)=0,
\end{equation}
where $B^+_n\defas \{ x \in B^{\nu,+}_1 \colon \vert u_n(x)-\zeta \vert<\eta_n\}$ and $B_n^-\defas \{ x \in B_1^{\nu,-}\colon \vert u_n(x)\vert <\eta_n\}$, and $(\eta_n)_n \subset (0,\infty)$ is a suitable sequence such that $\eta_n \to 0$, and
\begin{equation}
\label{small volume halfplane}
    \mathcal{L}^2(B^{\nu,\pm}_1\setminus B^\pm_n)\leq \QQQ \frac{\theta^8}{4} \EEE
\end{equation}
for every $n \geq n_{\theta}$, with $n_{\theta}\in \mathbb{N}$ sufficiently large. Now we define
\begin{equation}
\label{theta stripe}
    V_{\theta}\defas \{ x \in B_1 \colon \vert x \cdot \nu\vert \leq \theta\}.
\end{equation}
Fix $P_j^n$. Because of \eqref{separation of components}, we can assume without restriction $\mathcal{L}^2(P_j^n\cap B_n^+)=0$. The case $\mathcal{L}^2(P_j^n\cap B_n^-)=0$ is analogous.  By Fubini's theorem, \eqref{small volume halfplane}, and $\theta\leq\frac{1}{2}$ we get
\begin{equation}
\label{fubini partition}
    \int^{\theta}_{0}\mathcal{H}^1(P^n_j \cap L(s))\, \mathrm{d}s\leq \int^{\frac{1}{2}}_0\mathcal{H}^1((B^{\nu,+}_1\setminus B^+_n)\cap L(s))\, \mathrm{d}s\leq \mathcal{L}^2(B^{\nu,+}_1 \setminus B_n^+)\leq \QQQ \frac{\theta^8}{4} \EEE ,
\end{equation}
where $L(s)=\{x \in B_1 \colon x \cdot \nu=s\}$. Thus, we can choose $s^n_j \in (0,\theta)$ such that for $P_j^{n,+}\defas P_j^n\cap \{ x \in B_1 \colon x \cdot \nu >s_j^n\}$ and $P_j^{n,-}\defas P_j^n\cap \{ x \in B_1 \colon x \cdot \nu <s_j^n\}$ we have 
\begin{equation}
\label{theta 7}
    \mathcal{H}^1(\partial^* P_j^{n,\pm}\setminus\partial^* P^n_j)=\mathcal{H}^1(P_j^n\cap L(s_j^n))\leq \theta^{-1}\frac{\theta^8}{4}\leq \theta^7.
\end{equation}
Notice that $P_j^{n,\pm}\subset B^{\nu,\pm}_1\cup V_{\theta}$. We repeat this construction for each $P_j^n$. Similarly, one can define $R^\pm_n$ in such a way that $R_n=R^+_n\cup R^-_n$, $R_n^\pm \subset B^{\nu,\pm}_1 \cup V_{\theta}$, and 
\begin{equation}
\label{difference perimeter rn}
    \mathcal{H}^1(\partial^* R^\pm_n \setminus \partial^* R_n)\leq \theta^{-1}\mathcal{L}^2(R_n)\leq C^2_0 \theta,
\end{equation}
where we used \eqref{approximation with pr functions lemma-2} in the last inequality. Because of \eqref{theta 7}, we also have
\begin{equation}
\label{difference perimeter pnj}
   \sum^{J_n}_{j=1} \mathcal{H}^1(\partial^* P_j^{n,\pm}\setminus \partial^* P_j^n)\leq J_n \theta^7\leq \MMM \gamma_2 \EEE \theta,
\end{equation}
where we used again \eqref{approximation with pr functions lemma-2}, in particular \MMM the fact \EEE that $\gamma_2\geq\mathcal{L}^2(\bigcup^{J_n}_{j=1}P_j^n)\geq J_n \theta^6$. Summarizing, we have shown
\begin{subequations}
    \label{hyperplanes}
\begin{align}
  &  R^\pm_n \cup \bigcup^{J_n}_{j=1}P_j^{n,\pm}  \subset B^{\nu,\pm}_1 \cup V_{\theta}, \label{hyperplanes-1}
  \\ & \sum^{J_n}_{j=1}\mathcal{H}^1(\partial^* P_j^{n,\pm}\setminus \partial^* P^n_j)+\mathcal{H}^1(\partial^* R^\pm_n\setminus \partial^* R_n)\leq \big(C^2_0+\MMM \gamma_2 \EEE \big)\theta, \label{hyperplanes-2}
\end{align}
\end{subequations}
where in the last inequality we applied \eqref{difference perimeter rn} and \eqref{difference perimeter pnj}. \NNN Next, we consider \EEE
\begin{equation*}
    \hat{T}^+_{n,\theta}\defas \bigcup^{J_n}_{j=1}P^{n,+}_j \cup R_n^+, \quad \hat{T}^-_{n,\theta}\defas \bigcup^{J_n}_{j=1}P_j^{n,-}\cup R^-_n.
\end{equation*}
By \eqref{approximation with pr functions lemma-1} and \eqref{hyperplanes-2} we have \NNN
\begin{align}
\label{void contained in T}
    \mathcal{H}^1\big((\partial^* \hat{T}_{n,\theta}^- \cap \partial^* \hat{T}_{n,\theta}^+)\setminus \partial^* E_n\big) & \leq \sum^{J_n}_{j=1}\big(\mathcal{H}^1(\partial^* P^{n,\pm}_{j}\setminus \partial^* P_j^n)  +\mathcal{H}^1( (\partial^* P^{n}_j \cap B_1 )  \setminus \partial^* E_n) \big)    
 \nonumber
    \\&
  \ \ \ +\mathcal{H}^1(\partial^* R_n^\pm \setminus \partial^* R_n)    
    + \mathcal{H}^1((\partial^* R_n\cap B_1)\setminus \partial^* E_n)
               \nonumber
    \\& \leq \big( \QQQ C^{2}_0+ \gamma_2 \EEE\big)\theta+ C_0 \theta^2\leq 2\big(\MMM C_0^2+ \gamma_2 \EEE\big)\theta.
\end{align}
\NNN Recalling  $N_{\theta}\defas B_{1}\setminus B_{1-\theta}$ \EEE we set
\begin{align*}
    T^+_{n,\theta} & \QQQ \defas \EEE \big(\hat{T}^+_{n,\theta}\cup (N_{\theta}\cap B^{\nu,+}_1\cap V_{\theta})\big)\setminus (N_{\theta}\cap B^{\nu,-}_1\cap V_{\theta}), \\
   T^-_{n,\theta} & \QQQ \defas \EEE \big(\hat{T}^-_{n,\theta}\cup (N_{\theta}\cap B^{\nu,-}_1\cap V_{\theta})\big)\setminus  (N_{\theta}\cap B^{\nu,+}_1\cap V_{\theta}).
\end{align*}
We notice  that, \MMM by \eqref{theta stripe} and \eqref{hyperplanes-1}, \EEE 
\begin{equation}
\label{result before vodification}
\mathcal{L}^2((T^+_{n,\theta}\cup T^-_{n,\theta})\setminus B_1)=0,\qquad  \mathcal{L}^2(B^{\nu,\pm}_{1}\triangle T^\pm_{n,\theta})\leq 2\theta, \qquad   N_{\theta}\cap B^{\nu,\pm}_1 \NNN = N_\theta \cap \EEE T^\pm_{n,\theta}.
\end{equation}
 In addition,
\begin{align}
\label{control perimeter on new sets}
\mathcal{H}^1((\partial^* T^+_{n,\theta}\cap \partial^* T^-_{n,\theta})\setminus \partial^* E_n)&\leq  \MMM c \theta \EEE  + \mathcal{H}^1\big((\partial^* \hat{T}^+_{n,\theta}\cap \partial^* \hat{T}^-_{n,\theta})\setminus \partial^* E_n\big)  \leq c\theta+ 2\big( \MMM C^2_0+ \gamma_2 \EEE \big)\theta
\end{align}
for some universal constant $c>0$, where we used \eqref{void contained in T} and 
\begin{equation*}
     \mathcal{H}^1\big((\partial^* T^+_{n,\theta}\cap \partial^* T^-_{n,\theta})\setminus (\partial^* \hat{T}^+_{n,\theta}\cap \partial^* \hat{T}^-_{n,\theta})\big) \leq \mathcal{H}^1\big(\partial(N_{\theta}\cap B^{\nu,-}_1\cap V_{\theta})\big)+\mathcal{H}^1\big(\partial(N_{\theta}\cap B^{\nu,+}_1\cap V_{\theta})\big)\leq c \theta.
\end{equation*}
\emph{Step 2 (Definition of $S^+_{n,\theta}$, $S_{n,\theta}^-$, $D_{n,\theta}$)}. \EEE  Now we replace the interface \MMM $\partial^* T^+_{n,\theta}\cap \partial^* T^-_{n,\theta}$ \EEE constructed in the previous step suitably by voids. \NNN We have that \EEE  $2\mathcal{H}^{1}((\partial^* {T}_{n,\theta}^- \cap \partial^* {T}_{n,\theta}^+)\setminus \partial^* E_n)\geq \mathcal{S}^{1}((\partial^* {T}_{n,\theta}^- \cap \partial^* {T}_{n,\theta}^+)\setminus \partial^* E_n)$, where $\mathcal{S}^{1}$ denotes the $1$-dimensional spherical Hausdorff measure (see \cite[Chapter 3]{alma991013725792303131}). For every $n\in \mathbb{N}$,  $\theta \in (0,\frac{1}{2}]$, \NNN and   $\delta>0$ sufficiently small \EEE   there exists a countable family of balls $(B_i)_{i\in \mathcal{I}}$, depending on $\delta$, $n$,  and $\theta$ \NNN such that \EEE
\begin{subequations}
    \label{consequence of Hausdorff def}
\begin{align}
& (\partial^* {T}_{n,\theta}^- \cap \partial^* {T}_{n,\theta}^+)\setminus \partial^* E_n \subset \bigcup_{i\in \mathcal{I}}B_i=: F_{n,\theta}^{\delta} \label{consequence of Hausdorff def-1}
\\ & \mathrm{diam}(B_i)\leq \delta ,  \label{consequence of Hausdorff def-2}
\\ &  \vphantom{\sum_{i=1}^{N}}   2 \mathcal{H}^{1}(\partial^* {T}_{n,\theta}^- \cap \partial^* {T}_{n,\theta}^+)\setminus \partial^* E_n) \NNN\geq \EEE \sum_{i \in \mathcal{I}}\mathrm{diam}(B_i)-\theta, \label{consequence of Hausdorff def-3}
\end{align}
\end{subequations}
where $\mathrm{diam}(B_i)$ denotes the diameter of $B_i$.  In particular, \eqref{control perimeter on new sets} and \eqref{consequence of Hausdorff def-2}--\eqref{consequence of Hausdorff def-3} imply that there exists a \NNN universal \EEE constant $C>0$  such that
\begin{align}
\label{delta control voidification: perimeter}
\mathcal{H}^{1}\big(\partial^* F^{\delta}_{n,\theta}\big) \le C\theta, \quad  \quad    \mathcal{L}^d(F^{\delta}_{n,\theta}) \le C\theta \delta.    
\end{align}
\EEE

We define
\begin{equation}
\label{def Dn}
   {D}^{\delta}_{n,\theta}\defas E_n \cup F^{\delta}_{n,\theta}.
\end{equation}
By virtue of \eqref{delta control voidification: perimeter}, we have 
\begin{equation}
\label{control voidification}
 \mathcal{H}^1\big(\partial^* {D}^{\delta}_{n,\theta}\setminus \partial^* E_{n}\big)\leq C \theta, \quad  \quad \mathcal{L}^2({D}^{\delta}_{n,\theta})\leq \mathcal{L}^2(E_n)+C\theta\delta.
\end{equation}
Now we define the sets $  S^\pm_{n,\theta}$. First, we let
\begin{equation*}
   {S}^{\delta,+}_{n,\theta}\defas {T}_{n,\theta}^+\setminus 
   D^\delta_{n,\theta}\EEE,\quad   {S}^{\delta,-}_{n,\theta}\defas {T}_{n,\theta}^-\setminus 
    D^\delta_{n,\theta}.
\end{equation*}
\NNN Due to \EEE  \eqref{result before vodification} and \eqref{control voidification}, the sets \QQQ $S^{\delta, +}_{n,\theta}$, $S_{n,\theta}^{\delta, -}$, $D_{n,\theta}^{\delta}$ \EEE satisfy
\begin{subequations}
\label{added equation voidification}
\begin{align}
&  \mathcal{L}^2\big((D^\delta_{n,\theta}\cup S^{\delta,+}_{n,\theta}\cup S^{\delta,-}_{n,\theta})\triangle B_1\big)=0,
\\ &   \mathcal{L}^2(S^{\delta,\pm}_{n,\theta}\triangle B^{\nu,\pm}_1)\leq   \NNN 2\theta + \EEE  C \theta\delta+\mathcal{L}^2(E_n),
\\ &  N_{\theta}\cap (B^{\nu,\pm}_{1}\setminus D^\delta_{n,\theta}) \NNN =   N_{\theta}\cap  S^{\delta, \pm}_{n,\theta} .    \EEE
\end{align}
\end{subequations}
In addition, exploiting the construction of $ F^\delta_{n,\theta}$,   \eqref{consequence of Hausdorff def-1}, and \eqref{def Dn}, \EEE we can check that  
\begin{equation}
\label{added equation voidification 2}
  \mathcal{H}^{1}(\partial^* {S}^{\delta,+}_{n,\theta}  \cap \partial^* {S}^{\delta,-}_{n,\theta}  )=\mathcal{H}^{1}\big(\partial^*  {T}^+_{n,\theta}  \cap \partial^*  {T}^-_{n,\theta}\cap ({D}^\delta_{n,\theta})^0\big)=0.
\end{equation}
In particular, $  {D}^\delta_{n,\theta}$ disconnects $   {S}_{n,\theta}^{\delta,+} \EEE $ and $  {S}^{\delta,-}_{n,\theta} \EEE $. Now, the definition of   $S^{\pm}_{n,\theta}$ and $D_{n,\theta}$, and the \EEE proof follow by a diagonal argument with a suitable sequence $  \delta_n \to 0$. Indeed,    \eqref{statement lemma 5.2-1}--\eqref{statement lemma 5.2-2} and \eqref{statement lemma 5.2-4} follow from \eqref{added equation voidification} and  \eqref{added equation voidification 2} \MMM while \EEE \eqref{statement lemma 5.2-3} follows from \eqref{control voidification}, \NNN where we also use $\mathcal{L}^d(E_n)\to 0$.
\end{proof}
 
\begin{center}
\begin{figure}
\begin{tikzpicture} 
\tikzset{x=1.7ex,y=1.7ex} 
\definecolor{color5}{HTML}{0801ff}
\definecolor{color4}{HTML}{ea01d9}
\draw[line width=0.34pt, color=color0, fill=color1, fill opacity=0.00] (37.82, 33.44) circle [radius=12.76];
\draw[line width=0.34pt, color=color5, ] (34.52, 38.13) .. controls (31.97, 39.02) and (33.00, 38.68) .. (25.15, 34.81);
\draw[line width=0.34pt, color=color5, ] (34.47, 38.15) .. controls (30.29, 31.85) and (28.19, 30.71) .. (25.21, 31.76);
\draw[line width=0.34pt, color=color5, ] (49.99, 29.67) .. controls (38.19, 41.52) and (37.82, 40.44) .. (49.07, 39.50);
\draw[line width=0.34pt, color=color5, ] (40.88, 39.43) -- (34.48, 38.13);
\draw[line width=0.34pt, color=color4, fill=color1, fill opacity=0.00] (34.78, 38.13) circle [radius=0.29];
\draw[line width=0.34pt, color=color4, fill=color1, fill opacity=0.00] (35.16, 38.28) circle [radius=0.26];
\draw[line width=0.34pt, color=color4, fill=color1, fill opacity=0.00] (35.50, 38.37) circle [radius=0.19];
\draw[line width=0.34pt, color=color4, fill=color1, fill opacity=0.00] (35.88, 38.54) circle [radius=0.26];
\draw[line width=0.34pt, color=color4, fill=color1, fill opacity=0.00] (36.28, 38.53) circle [radius=0.14];
\draw[line width=0.34pt, color=color4, fill=color1, fill opacity=0.00] (36.67, 38.58) circle [radius=0.25];
\draw[line width=0.34pt, color=color4, fill=color1, fill opacity=0.00] (37.14, 38.67) circle [radius=0.17];
\draw[line width=0.34pt, color=color4, fill=color1, fill opacity=0.00] (37.47, 38.71) circle [radius=0.31];
\draw[line width=0.34pt, color=color4, fill=color1, fill opacity=0.00] (38.05, 38.80) circle [radius=0.51];
\draw[line width=0.34pt, color=color4, fill=color1, fill opacity=0.00] (38.49, 39.13) circle [radius=0.56];
\draw[line width=0.34pt, color=color4, fill=color1, fill opacity=0.00] (38.96, 39.07) circle [radius=0.14];
\draw[line width=0.34pt, color=color4, fill=color1, fill opacity=0.00] (39.23, 39.10) circle [radius=0.12];
\draw[line width=0.34pt, color=color4, fill=color1, fill opacity=0.00] (39.59, 39.18) circle [radius=0.15];
\draw[line width=0.34pt, color=color4, fill=color1, fill opacity=0.00] (39.90, 39.16) circle [radius=0.29];
\draw[line width=0.34pt, color=color4, fill=color1, fill opacity=0.00] (39.37, 39.13) circle [radius=0.09];
\draw[line width=0.34pt, color=color4, fill=color1, fill opacity=0.00] (40.49, 39.35) circle [radius=0.23];
\draw[line width=0.34pt, color=color4, fill=color1, fill opacity=0.00] (40.19, 39.26) circle [radius=0.30];
\draw[line width=0.34pt, color=color4, fill=color1, fill opacity=0.00] (40.62, 39.43) circle [radius=0.29];
\node[color=color5, anchor=north west, scale=1.00, rotate=0.00, text width=0.58cm] at (27.73, 35.55) { $E_n$};
\node[color=color5, anchor=north west, scale=1.00, rotate=0.00, text width=0.58cm] at (44.80, 37.54) {  $E_n$};
\node[color=color4, anchor=north west, scale=1.00, rotate=0.00, text width=0.68cm] at (35.68, 37.16) {$F_{n,\theta}^{\delta}$};
\node[color=color0, anchor=north west, scale=1.00, rotate=0.00, text width=0.68cm] at (35.68, 43.08) {$S^{\delta,+}_{n,\theta}$};
\node[color=color0, anchor=north west, scale=1.00, rotate=0.00, text width=0.68cm] at (35.68, 31.09) {$S^{\delta,-}_{n,\theta}$};
\end{tikzpicture} 
    \label{figure4}
    \captionof{figure}{Example of \NNN the \EEE construction of the void $D_{n,\theta}^\delta$. }
\end{figure}
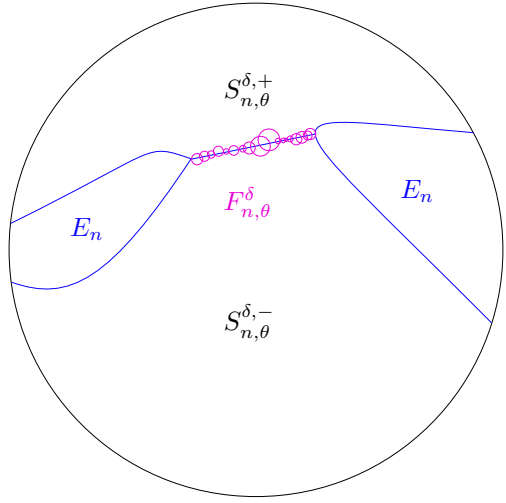    
\end{center}
\subsection{Surface part: Proof of \eqref{identification gamma limit voids} and \eqref{h0 pc}}
Before we come to the proof, we state the following general fact. 
\begin{remark}
\label{remark recovery sequences}
We will use this general property of recovery sequences: if $(u_\varepsilon,E_{\varepsilon})_{\varepsilon}$ is a recovery sequence for $(u,E)$ with respect to $\mathcal{E}_{\varepsilon}(\cdot,\cdot,\Omega)$, then $(u_{\varepsilon}, E_{\varepsilon})_{\varepsilon}$ is still a recovery sequence for $(u,E)$ with respect to $\mathcal{E}_{\varepsilon}(\cdot,\cdot,A)$, \EEE for every $A \in \mathcal{A}(\Omega)$ such that $\mathcal{E}_0(u,E,\partial A)=0$. This follows from the fact that $\mathcal{E}_0(u,E,\cdot)$ is a Radon measure for each \NNN $(u, E) \in \mathcal{G}^p(\Omega)$. \EEE
\end{remark}

\begin{proof}[Proof of Theorem \ref{Identification of Gamma-limit}(ii).]
\emph{Step 1: $g_0(x,\nu_{E}\MMM (x) \EEE )\leq \hat{g}(x,\nu_{E}\MMM (x) \EEE )$ \MMM for  $\mathcal{H}^{d-1}$-a.e.\ $\NNN x \in \EEE \partial^*E\cap \Omega$}.  \MMM  Recalling \eqref{infimumproblem1} and \eqref{simplified_cellformula_voids} (for $\mathcal{E}_0$ in place of $\mathcal{E}_\eps$),  we clearly have $m_{\mathcal{E}_0}(  0 ,  \Pi^{\nu_E(x),-}_{x},B_{\rho}(x)) \le m_{\mathcal{E}_0}^{\mathrm{voids}}(\Pi_x^{\nu_E(x),-}, B_{\rho}(x))$. Arguing like in \cite[Proposition 3.14]{FriPerSol20a} one can show that \MMM the density $\hat{g}$ given in \eqref{g1 = g2} satisfies \EEE
\begin{equation*}
    \hat{g}(x,\nu)=\limsup\limits_{\rho \to 0}\frac{\MMM m_{\mathcal{E}_0}^{\mathrm{voids}}(\Pi_x^{\nu,-}, B_{\rho}(x))}{\MMM \gamma_{d-1} \EEE \rho^{d-1}}\:\: \text{for all}\:\: \MMM x \in \Omega, \EEE \nu \in \mathbb{S}^{d-1}.
\end{equation*}
\MMM Now, the inequality follows from \EEE \eqref{def g_0}.

\noindent
\emph{Step 2: $g_0(x,\nu_{E}\MMM (x) \EEE )\ge \hat{g}(x,\nu_{E}\MMM (x) \EEE )$ \MMM for  $\mathcal{H}^{d-1}$-a.e.\ $ \NNN x \in \EEE \partial^*E \cap \Omega$}.  Now we prove the converse inequality.   \MMM By \eqref{the e-0 energy} applied on the pair $(0,E)$ and the Radon-Nikod\'ym Theorem we have  \EEE
 that
\begin{equation}\label{new equation g hat0}
 g_0(x,\nu_E(x))=\lim\limits_{\rho \to 0}\frac{\mathcal{E}_0(0,E,B_{\rho}(x))}{\gamma_{d-1}\rho^{d-1}}\:\: \text{for $\mathcal{H}^{d-1}$-a.e.}\:\: x \in \partial^*E\cap \Omega.     
\end{equation}
\MMM  We can assume that  $x \in \partial^*E\cap \Omega$ is chosen with this property and \EEE  as  an approximate jump point of $\chi_{E}$. \MMM Indeed,  recall \EEE that the set of jump points of $\chi_{E}$ has full $\mathcal{H}^{d-1}\llcorner (\partial^* E\cap \Omega)$-measure. \MMM For notational simplicity, we write $\nu$ in place of $\nu_E(x)$ in the sequel. \EEE

Let $(u_{\varepsilon},E_\varepsilon)_\varepsilon$ be a recovery sequence for $(0,E)$. Since $u_{\varepsilon}\to 0$, \MMM in view of $(f_1)$, \EEE also $(0,E_{\varepsilon})_{\varepsilon}$ defines a recovery sequence \NNN for \EEE $(0,E)$. 
In view of \eqref{g1 = g2}, we can find a sequence $(\rho_n)_n$ with $\rho_n \searrow 0$ such that  $\mathcal{E}_0(0,E,\partial B_{\rho_n}(x))=0$ and \EEE
\begin{equation}
\label{new equation g hat}
    \hat{g}(x,\nu)=\lim\limits_{n \to \infty}\limsup\limits_{\varepsilon \to 0}\frac{m^{\mathrm{voids}}_{\mathcal{E}_{\varepsilon}}(\Pi^{\nu,-}_x,B_{\rho_n}(x))}{\MMM \gamma_{d-1}\rho^{d-1}_n\EEE}.
\end{equation}
Then, since $(0,E_{\varepsilon})_{\varepsilon}$ is a recovery sequence for $(0,E)$ on $\Omega$ and $\mathcal{E}_0(0,E,\partial B_{\rho_n}(x))=0$, for every $n \in \mathbb{N}$ there exists $\varepsilon_n$ such that for $\varepsilon \leq \varepsilon_n$ it holds
\begin{equation}
\label{PC inequality 0}
\frac{1}{\gamma_{d-1}\rho^{d-1}_n}\mathcal{E}(0,E,B_{\rho_n}(x))\geq \frac{1}{\gamma_{d-1}\rho^{d-1}_n}\mathcal{E}_{\varepsilon}(0,E_{\varepsilon},B_{\rho_n}(x)) -\frac{1}{n},
\end{equation}
\MMM see Remark \ref{remark recovery sequences}. \EEE Because of \eqref{new equation g hat}, one can choose the sequence $(\varepsilon_n)_n$ to further satisfy
\begin{equation}\label{alsotouse}
    \hat{g}(x,\nu)=\lim\limits_{n \to \infty}\frac{m^{\mathrm{voids}}_{\mathcal{E}_{\varepsilon_n}}(\Pi^{\nu,-}_x,B_{\rho_n}(x))}{\gamma_{d-1}\rho^{d-1}_{n}}.
\end{equation}
We define $\tilde{E}^n_{\varepsilon}=\frac{E_{\varepsilon}-x}{\rho_n}$ and $\tilde{E}^n=\frac{E-x}{\rho_n}$. Notice that $\chi_{\tilde{E}^n_\varepsilon} \to \chi_{\tilde{E}^n}$ on $B_1$ \MMM as $\eps \to 0$ \EEE since $\chi_{E_\varepsilon}\to \chi_E$ in $L^1(\Omega)$. Since $x \in \partial^* E \cap \Omega$ is an approximate jump point of $\chi_E$, we find \MMM (up to performing a reflection) \EEE $\MMM \chi_{\tilde{E}^n} \EEE \to \chi_{\MMM \Pi^{\nu,-}_0 \EEE }$ in $L^1(B_1)$ for $n \to \infty$. Thus,  using a diagonal argument, up to passing to smaller $\varepsilon_n$, we can assume \MMM that \EEE
$\varepsilon_n \to 0$ as $n \to \infty$ and
\begin{equation}
\label{eq convergence rescaled set}
    \chi_{F_n} \to \chi_{\Pi_0^{\nu,-}}\:\: \text{in $L^1(B_1)$},
\end{equation}
where  $F_n\NNN \defas \EEE \tilde{E}^n_{\varepsilon_n}$. By a change of \MMM variables and \eqref{PC inequality 0} \EEE we get that 
\begin{align}
\label{converse inequality 1 voids}
    \frac{1}{\gamma_{d-1}}\int_{\partial^* F_n\cap B_1}g_{\varepsilon_n}(x+\rho_n y, \nu_{F_n})\, \mathrm{d}\mathcal{H}^{d-1}(y)&=\frac{1}{\gamma_{d-1}\rho^{d-1}_n}\int_{\partial^* \MMM E_{\varepsilon_n} \EEE \cap B_{\rho_n}(x)}g_{\varepsilon_n}(z,\nu_{E_{\varepsilon_n}})\, \mathrm{d}\mathcal{H}^{d-1}(z) \nonumber
    \\& \leq \frac{\mathcal{E}_0(0,E,B_{\rho_n}(x))}{\gamma_{d-1}\rho^{d-1}_n}+\frac{1}{n}.
\end{align}
Let $\eta>0$ and let $N_0 \subset B_1$ be a neighborhood of $\partial B_1$ such that $\beta \mathcal{H}^{d-1}(\Pi^{\nu}_0\cap N_0)\leq \eta$. We apply the fundamental estimate  (Proposition \ref{fundamentalestimate} \MMM and   Remark \ref{remark.on.fund.est}) \EEE  with $(0,F_n)$ \MMM (in place of $(u,E)$), \EEE $(\MMM 0, \EEE \Pi^{\nu,-}_{0})$ \MMM (in place of $(v,F)$), \EEE $A \subset \MMM B_1 \EEE$, $A^\prime \subset \subset A$ such that $B_1 \setminus \overline{A^\prime} \subset N_0$, and $B=B_{1} \setminus \overline{A^\prime}$ to the functional 
\begin{equation}\label{sametildeE}
    \tilde{\mathcal{E}}_{\varepsilon}  (u,E,  U ) \EEE =\int_{ U  }f_{\varepsilon}\big(x+\rho_{\varepsilon}y,e(u)(y)\big)\, \mathrm{d}y + \int_{\partial^* E \cap  U  }g_{\varepsilon}(x+\rho_{\varepsilon}y, \MMM \nu_E)  \, \mathrm{d}\mathcal{H}^{d-1} \NNN (y). 
\end{equation}
Indeed, notice that this functional satisfies the same growth conditions \MMM as \EEE $\mathcal{E}_{\varepsilon}$. In this way, due to \eqref{eq convergence rescaled set}, we can find voids sets $(D_n)_n$ such that ${D_n}={\Pi^{\nu,-}_0}$ near $\partial B_1$ and
\begin{equation*}
  \MMM     \limsup\limits_{n\to \infty} \EEE  \int_{\partial^* D_n \cap B_1}g_{\varepsilon_n}(x+\rho_n y,\nu_{D_n}  )\, \mathrm{d}\mathcal{H}^{d-1}\MMM (y) \EEE\leq    \MMM     \limsup\limits_{n\to \infty} \EEE \int_{\partial^* F_n\cap B_1}g_{\varepsilon_n}(x+\rho_n y, \nu_{F_n} )\, \mathrm{d}\mathcal{H}^{d-1}(y)+\NNN 2\eta. \EEE
\end{equation*}
By setting $\tilde{D}_n \defas x+\rho_n D_n$, again using a change of variables and \eqref{converse inequality 1 voids}, we get
\begin{align*}
\label{eq XYZ}
    \limsup\limits_{n\to \infty}\frac{1}{\gamma_{d-1}\rho_n^{d-1}}\int\limits_{\partial^* \tilde{D}_n \MMM \cap B_{\rho_n}(x)} \hspace{-0.4cm}g_{\varepsilon_n}(\NNN \cdot , \EEE \nu_{\tilde{D}_n})\, \mathrm{d}\mathcal{H}^{d-1}&\leq \limsup\limits_{n \to \infty}\frac{1}{\gamma_{d-1}\rho^{d-1}_n}\int\limits_{\partial^* E_{\varepsilon_n}\cap B_{\rho_n}(x)} \hspace{-0.4cm} g_{\varepsilon_n}(\NNN \cdot, \EEE \nu_{E_{\varepsilon_n}})\, \mathrm{d}\mathcal{H}^{d-1} +\NNN\frac{ 2  \eta}{\gamma_{d-1}} \EEE \\&
 \leq \limsup\limits_{n \to \infty}\frac{\mathcal{E}(0,E,B_{\rho_n}(x))}{\gamma_{d-1}\rho^{d-1}_n}+ \NNN\frac{ 2  \eta}{\gamma_{d-1}}. \EEE
\end{align*}
Since ${\tilde{D}_n}={\Pi^{\nu,-}_x}$ near $\partial B_{\rho_n}(x)$, by applying \MMM $(f_1)$,  \eqref{new equation g hat0}, \eqref{alsotouse},  \EEE  and the arbitrariness of $\eta>0$, \EEE it follows $\hat{g}(x,\nu)\leq g_0(x,\nu)$.
\end{proof}

\begin{proof}[Proof of Theorem \ref{Identification of Gamma-limit}(iii).]
\MMM From \EEE now on we assume $d=2$ and $p \geq 2$. To simplify the notation, we set $\zeta=[u](x)$ and $\nu=\nu_u(x)$. This proof is an adaptation of the proof in \cite[Theorem 2.4, (2.21)]{FriPerSol20a}.  

\noindent \emph{Step 1: $h_0(x,\zeta,\nu)\leq \hat{h}(x,\nu)$ \MMM for $\mathcal{H}^{1}$-a.e.\ $x \in J_u\cap \NNN E^0 $\EEE}. 

\MMM Recall the definitions in \eqref{set of competitors}--\eqref{simpilfied_cellformula_jump}. Let $(\hat{u}_{\varepsilon},E_{\varepsilon})\in \mathcal{D}^{\nu}_{\varepsilon}(B_{\rho}(x))$ \MMM be \EEE such that
\begin{equation*}
    \mathcal{E}_{\varepsilon}(\MMM \hat{u}_{\varepsilon}, \EEE E_{\varepsilon},B_{\rho}(x))\leq m^{\mathrm{jump}}_{\mathcal{E}_{\varepsilon}}(u^{\nu}_{x,\MMM e_1},B_{\rho}(x))+\varepsilon.
\end{equation*}
\MMM  Denote by $S^\pm_\eps$ the sets corresponding to the void $E_\eps$, as given in \eqref{set of competitors}. We define ${u}_{\varepsilon} = \hat{u}_{\varepsilon} + (\zeta-e_1)\chi_{S^+_\eps}$. We note that $ {u}_{\varepsilon} = \eta^\nu_{x,\zeta,\eps}$  near $\partial B_{\rho}(x)$ (recall  \eqref{function thin layer}) and that 
\begin{equation}\label{nehmermit0}
    \mathcal{E}_{\varepsilon}(u_{\varepsilon},E_{\varepsilon},B_{\rho}(x)) =    \mathcal{E}_{\varepsilon}(\hat{u}_{\varepsilon},E_{\varepsilon},B_{\rho}(x))\leq m^{\mathrm{jump}}_{\mathcal{E}_{\varepsilon}}(u^{\nu}_{x,\MMM e_1},B_{\rho}(x))+\varepsilon.
\end{equation} \EEE 
Let $\delta=\delta(\rho)>0$ with $0<\delta<\rho^2$.
Extend $(u_{\varepsilon},E_{\varepsilon})$ on $B_{\rho+\delta}(x)$ by setting it equal to $(\MMM \eta^\nu_{x,\zeta,\eps}, \EEE E^{x,\nu}_{\varepsilon})$ (recall \eqref{thin layer}) outside \MMM of \EEE $B_{\rho}(x)$. Notice that by construction and \MMM $(g_3)$ \EEE it holds that
\begin{equation}\label{nehmermit}
    \mathcal{E}_{\varepsilon}(u_{\varepsilon},E_{\varepsilon},B_{\rho+\delta}(x))\leq \mathcal{E}_{\varepsilon}(u_{\varepsilon},E_{\varepsilon},B_{\rho}(x))+4\beta \delta.
\end{equation}
\MMM By a compactness argument for sets of finite perimeter, \EEE up to extracting a subsequence (not relabeled), $(u_{\varepsilon},E_{\varepsilon})\to (v,E)$  \MMM for some $(v,E) \in \mathcal{G}^p(B_{\rho+\delta}(x))$ with  $v(B_{\rho+\delta}(x)) = \{0,\zeta\}$ \EEE and $v=u^\nu_{x,\zeta}$, ${E}=\emptyset$ on $B_{\rho+\delta}(x)\setminus B_{\rho}(x)$. Hence, \MMM using the $\Gamma$-convergence of $(\mathcal{E}_\eps)_\eps$ to $\mathcal{E}_0$ and \eqref{nehmermit0}--\eqref{nehmermit}, \EEE we have
\begin{align*}
    m_{\mathcal{E}_0}(u^\nu_{x,\zeta},\emptyset,B_{\rho+\delta}(x))\leq \mathcal{E}_0({v},E,B_{\rho+\delta}(x))&\leq \liminf\limits_{\varepsilon \to 0}\mathcal{E}_{\varepsilon}(u_{\varepsilon},E_{\varepsilon},B_{\rho}(x))+4\beta\delta
    \\& \leq\liminf\limits_{\varepsilon \to 0}m^{\mathrm{jump}}_{\mathcal{E}_{\varepsilon}} ( \MMM u^{\nu}_{x, e_1}, \EEE B_{\rho}(x))+4\beta\delta.
\end{align*}
Dividing  both sides \MMM by \EEE $\gamma_{1}(\rho+\delta)$, sending $\rho \to 0$, and using \eqref{def h_0}, \MMM \eqref{h1 = h2}, \EEE $\delta<\rho^2$, we get  $h_0(x,\zeta,\nu)\leq \hat{h}(x,\nu)$.

\noindent \emph{Step 2: $\hat{h}(x,\nu)\leq h_0(x,\zeta,\nu)$ \MMM for $\mathcal{H}^{1}$-a.e.\ $x \in J_u\cap \NNN E^0$\EEE}.  \MMM By \eqref{the e-0 energy} applied on the pair $(u,\emptyset)$ and the Radon-Nikod\'ym Theorem we have  \EEE
\begin{equation}
\label{the equation after 7.51}
 h_0(x,\zeta,\nu)=\lim\limits_{\rho \to 0}\frac{\mathcal{E}_0(u,\emptyset,B_{\rho}(x))}{\gamma_{1}\rho}\:\: \text{for \MMM $\mathcal{H}^{1}$\EEE-a.e.}\:\: x \in J_u \cap \NNN E^0.  \EEE     
\end{equation}
\MMM In view of \eqref{h1 = h2}, \EEE we can find a sequence $\rho_n\to 0$ such that $\mathcal{E}_0(u,\emptyset,\partial B_{\rho_n}(x))=0$ and
\begin{equation}
\label{hat h subsequence}
    \hat{h}(x,\nu)=\lim\limits_{n \to \infty} \limsup_{\varepsilon \to 0} \frac{m^{\mathrm{jump}}_{\mathcal{E}_{\varepsilon}}(u^{\nu}_{x,\MMM e_1 \EEE },B_{\rho_n}(x))}{\gamma_{1}\rho_{n}}.
\end{equation}
\MMM In what follows, we can without restriction fix $x \in J_u$ satisfying \eqref{the equation after 7.51}--\eqref{hat h subsequence} as these properties hold up to an $\mathcal{H}^1$-negligible set in $J_u$. \EEE

Let $(u_\varepsilon,E_\varepsilon)_\varepsilon$ be a recovery sequence for $(u,\emptyset)$ on $\Omega$. Then, \MMM by \EEE $\mathcal{E}_0(u,\emptyset,\partial B_{\rho_n}(x))=0$ \MMM and Remark~\ref{remark recovery sequences}, \EEE for every $n \in \mathbb{N}$ there exists $\varepsilon_n$ such that for $\varepsilon \leq \varepsilon_n$ it holds
\begin{equation}
\label{PC inequality voids 0}
\frac{1}{\gamma_{1}\rho_n}\mathcal{E}_0(u,\emptyset,B_{\rho_n}(x))\geq \frac{1}{\gamma_1\rho_n}\mathcal{E}_{\varepsilon}(u_{\varepsilon},E_{\varepsilon},B_{\rho_n}(x)) -\frac{1}{n}.
\end{equation}
By virtue of \eqref{hat h subsequence}, we can choose the sequence $(\varepsilon_n)_n$ such that $\varepsilon_n \to 0$ and \MMM additionally \EEE
\begin{equation}
 \label{hat h only one limit}
 \hat{h}(x,\nu)=\lim\limits_{n \to \infty}\frac{m^{\mathrm{jump}}_{\mathcal{E}_{\varepsilon_n}}(u^{\nu}_{x,\MMM e_1},B_{\rho_n}(x))}{\gamma_{1}\rho_{n}}.
\end{equation}
We define $\tilde{E}^n_{\varepsilon}=\frac{E_{\varepsilon}-x}{\rho_n}$, $\tilde{u}^n_{\varepsilon}(y)=u_{\varepsilon}(x+\rho_ny)$,  and $\MMM \tilde{u}^n(y) = u \EEE (x+\rho_ny)$ for $y\in B_1$. Notice that $\tilde{u}_{\varepsilon}^n \to \tilde{u}^n$, $\chi_{\tilde{E}^n_{\varepsilon}}\to 0$ in measure on $B_1$ \MMM as $\eps \to 0$ \EEE since $u_{\varepsilon}\to u$ in measure and $\chi_{E_{\varepsilon}}\to 0$ on $B_{\rho}(x)$. \MMM Since \EEE $x \in J_u$ is an approximate jump point, \NNN $\tilde{u}^n\to \overline{u}^{\nu}_{0,u^-(x),u^+(x)}$ \EEE in measure for $n \to \infty$. \EEE Define $v_n=\tilde{u}^n_{\varepsilon_n}$. Up to passing to a smaller $\varepsilon_n$, \MMM by a diagonal argument, \EEE  we can assume that
\begin{equation}\label{5Y}
    v_n \to  \overline{u}^{\nu}_{0,u^-(x),u^+(x)}  \:\: \text{in measure on}\:\: B_1.
\end{equation} 
Similarly, we define $F_n=\tilde{E}^n_{\varepsilon_n}$ and, up to further passing to a smaller $\varepsilon_n$, we can assume that
\begin{equation}\label{6Y}
    \chi_{F_n}\to 0 \:\: \QQQ \text{in} \EEE \:\: L^1(B_1).
\end{equation}
By a change of \MMM variables  \EEE    and \MMM \eqref{PC inequality voids 0} \EEE we get that   \EEE
\begin{align}
\label{converse inequality 1}
    \frac{1}{\gamma_1}\int_{\NNN \partial^* \EEE F_n\cap B_1}\hspace{-0.2cm}g_{\varepsilon_n}(x+\rho_n y, \nu_{F_n})\, \mathrm{d} \MMM \mathcal{H}^{1} \EEE (y) & =\frac{1}{\gamma_1\rho_n}\int_{\NNN \partial^* \EEE E_{\varepsilon_n}\cap B_{\rho_n}(x)}\hspace{-0.2cm}g_{\varepsilon_n}(z,\nu_{E_{\varepsilon_n}})\, \mathrm{d} \MMM \mathcal{H}^{1} \EEE(z)
    \\
    &
    \nonumber\leq \frac{\mathcal{E}_0(u,\emptyset,B_{\rho_n}(x))}{\gamma_1\rho_n}+\frac{1}{n}.
\end{align}
We also define $\tilde{v}_n \MMM (y) \EEE = v_n(y)-u^-(x)(1-\chi_{F_n})$ for all $y \in B_1$. We \MMM now \EEE check that $(\tilde{v}_n,F_n)$ satisfies the assumptions of Lemma \ref{Lemma 5.2}. First, we observe that, by   \MMM  \eqref{the equation after 7.51}, \eqref{converse inequality 1}, and $(g_2)$, \EEE 
\begin{equation*}
    \sup_{n \in \N} \MMM \mathcal{H}^{1} \EEE (\partial^* F_{n}\cap B_1)\leq \frac{\MMM \gamma_1}{\alpha}\sup_{n \in \N}\Big(\frac{1}{\gamma_{1}\rho_n}\mathcal{E}_0(u,\emptyset,B_{\rho_n}(x))+\frac{1}{n}\Big)<\infty.
\end{equation*}
Next, we show $\Vert e(\tilde{v}_n)\Vert_{L^p(B_1)}\to 0$. Indeed, using \NNN  $(f_2)$ \EEE and a change of variables we have
\begin{equation*}
    \int_{B_1}\vert e(\tilde{v}_n(y))\vert^p \, \mathrm{d}y \leq \frac{\rho^p_n}{\rho^2_n}\int_{B_{\rho_n}(x)}\vert e(u_{\varepsilon_n}\MMM (z) \EEE )\vert^p \, \mathrm{d}z\leq \frac{\gamma_1\rho^{p-1}_n}{\alpha}\frac{1}{\gamma_1\rho_n} \MMM \mathcal{E}_{\eps_n} \EEE \big(u_{\varepsilon_n}, \MMM E_{\eps_n}, \EEE B_{\rho_n}(x)\big)\to 0,
\end{equation*}
where we also use  the \MMM uniform energy bound provided by \eqref{the equation after 7.51} and \eqref{PC inequality voids 0}. \EEE   
Finally, $\tilde{v}_n \to u_{0,\zeta}^{\nu}$  \MMM and  $\chi_{F_n}\to 0$  \NNN in measure on $B_1$ \EEE by the definition of $\tilde{v}_n$ and \eqref{5Y}--\eqref{6Y}. \EEE   Let $\theta\in (0,\frac{1}{2})$. Thanks \EEE to Lemma \ref{Lemma 5.2}, we can define a sequence of voids $  D_{n,\theta}$ disconnecting $B_1$ in two disjoint regions $S^+_{n,\theta}$ and $S^-_{n,\theta}$. \NNN In particular, \EEE because of \QQQ \eqref{statement lemma 5.2-2} and~\eqref{statement lemma 5.2-3},   we have \EEE  
\begin{equation}
\label{it makes non intersection true}
 S^\pm_{n,\theta}\cap (B_1 \setminus B_{1-\theta})=(B^\pm_{1}\setminus  D_{n,\theta})\cap (B_1 \setminus B_{1-\theta})    
\end{equation}
and
\begin{equation}\label{no volume}
  \mathcal{H}^{1} \big(  \partial^* D_n \setminus \partial^* F_n \big) \leq C\theta, \quad \quad  \NNN \lim\limits_{n \to \infty}\mathcal{L}^2(D_{n,\theta} ) = 0. \EEE
\end{equation} 
Hence, \MMM using also $(g_3)$ \EEE we have
\begin{align}
\label{PC inequality 1}
 \int_{\partial^* D_{n,\theta}\cap B_1}g_{\varepsilon_n}(x+\rho_n y,\nu_{D_{n,\theta}})\, \mathrm{d}   \mathcal{H}^{1} (y) \EEE \leq   \int_{\partial^* F_n\cap B_1}g_{\varepsilon_n}(x+\rho_n y,\nu_{F_n}) \, \mathrm{d}  \mathcal{H}^{1}  (y)+C\beta\theta. 
\end{align}
Let $\eta>0$. We apply Proposition \ref{fundamentalestimate} to the rescaled functional $\tilde{\mathcal{E}}_{\varepsilon}$ defined in \eqref{sametildeE},   for $\eta>0$, \EEE the sets  \NNN $A=B_{1-\theta/2}$, \EEE $A^\prime=B_{1-\theta}$, $B= \NNN B_{1} \EEE  \setminus B_{1-\theta}$, and $(0,D_{n,\theta})$ (in place $(u,E)$) and $(0,E_{\varepsilon_n}^{0,\nu})$ (in place of $(v,F)$). We obtain a set of finite perimeter $\NNN \hat{D}_{n,\theta} \EEE$ such that   $\NNN \hat{D}_{n,\theta} \EEE =E_{\varepsilon_n}^{0,\nu}$ near $\partial B_{\NNN 1}$ and
\begin{align}
\label{pc inequality 2}
   \int_{B_{\NNN 1}\cap \partial^* \hat{D}_{n,\theta}}g_{\varepsilon_n}(x+\rho_{n}y,\nu_{\hat{D}_{n,\theta}})\, \mathrm{d}\mathcal{H}^{1} \NNN(y) \EEE &\leq (1+\eta) \Big(\int_{B_1\cap \partial^* D_{n,\theta}}g_{\varepsilon_n}(x+\rho_{n}y,\nu_{D_{n,\theta}})\, \mathrm{d}\mathcal{H}^{1} \NNN(y) \EEE +\NNN 2 \EEE \beta \theta \Big) \nonumber
 \\& \ \ \ +C_{\theta,\eta}\Vert \chi_{D_{n,\theta}}-\chi_{E_{\varepsilon_n}^{0,\nu}}\Vert_{L^1( \NNN B_{1-\theta/2} \EEE \setminus B_{1-\theta})} \NNN + \eta \EEE
\end{align}
for a suitable constant $C_{\theta,\eta}>0$ depending on $\theta$ and $\eta$.  In addition, we observe that, by virtue of \eqref{it makes non intersection true}, $D_{n,\theta}$ and $E_{\varepsilon_n}^{0,\nu}\cap B$ satisfy \eqref{nonintersectionassumption}. Thus, Corollary \ref{set.corollary} implies that $\hat{D}_{n,\theta}$ separates $B_{1}$ in two disjoint regions \NNN $\hat{S}^+_{n,\theta}$ and $\hat{S}^-_{n,\theta}$,  i.e., $(e_1\chi_{\hat{S}^+_{n,\theta}},\hat{D}_{n,\theta})\in \mathcal{D}^\nu_{\varepsilon_n}(B_{1})$. 

Notice that $\Vert \chi_{D_{n,\theta}}-\chi_{E_{\varepsilon_n}^{0,\nu}}\Vert_{L^1(B_1)}\to 0$ \NNN as $n \to \infty$ \MMM by \eqref{thin layer} and \eqref{statement lemma 5.2}.  Hence, letting $n \to \infty$ in  \NNN \eqref{pc inequality 2} \EEE   and using \eqref{the equation after 7.51} together with \eqref{converse inequality 1}, \NNN  \eqref{no volume}, and a change of variables, \EEE  we obtain
$$\lim\limits_{n \to \infty}\frac{m^{\mathrm{jump}}_{\mathcal{E}_{\varepsilon_n}}(u^{\nu}_{x,\MMM e_1},B_{\rho_n}(x))}{\gamma_{1}\rho_{n}} \EEE \le (1+\eta)\Big(h_0(x,  \zeta ,\nu) +\frac{C}{\gamma_1}\beta\theta \Big) \NNN + \frac{\eta}{\gamma_1}. \EEE  $$
\EEE Finally, $\hat{h}(x,\nu)\leq h_0(x,\zeta,\nu)$ follows by sending $\eta \to 0$ and $\theta \to 0$, \NNN and recalling \eqref{hat h only one limit}. \EEE
\end{proof}
\NNN We recall again that   we do not need to prove \eqref{identification gamma limit jump parts}  as it   follows directly from \eqref{h0 pc} and \eqref{2g=h}.

\subsection{Proof of Theorem \ref{relaxation}} We close with the proof of the relaxation result.  \EEE

 \begin{proof}[Proof of Theorem \ref{relaxation}]
\NNN Applying  Proposition \ref{limsupliminf} on the constant sequences $f_\eps= f$ and $g_\eps= g$ for all $\eps>0$, we get that   \EEE  \eqref{f1 = f2}--\eqref{h1 = h2} hold. Theorem \ref{first gamma convergence result} and Theorem \ref{Identification of Gamma-limit} imply that $\overline{\mathcal{E}}$ admits an integral representation with densities  \NNN $f_0$, $g_0$, and ${h}_0$, where $f_0  =\hat{f}$ and  $g_0  =\hat{g}$.  By applying \cite[Theorem 9.8]{Dacorogna2008} on the bulk density, we find that $\hat{f} = f^{\rm qc}$. Moreover, it is elementary to see that  $\hat{g} = g^{\rm BV}$. Indeed, \EEE   denoting by $\sigma$ the modulus of continuity of $g$, we have that, for every $x \in \Omega$ and $E \in \mathcal{P}(B_1)$ such that $\mathcal{H}^{d-1}( \partial^* E \cap B_1 )\leq \frac{\beta}{\alpha}\gamma_{d-1}$, it holds
\begin{equation}
\label{spatial continuity of g}
    \Big\vert\int_{\partial^* E\cap B_1}g(x+\rho y,\nu_E)\, \mathrm{d}\mathcal{H}^{d-1}(y)-\int_{\partial^* E\cap B_1}g(x,\nu_E)\,\mathrm{d}\mathcal{H}^{d-1}(y)\Big\vert\leq \frac{\beta}{\alpha}\gamma_{d-1}\sigma(\rho).
\end{equation}
Thus,  $\hat{g}= {g}^{\rm BV}$ by \eqref{g1 = g2} and \eqref{relaxationxxxxxxxx},  \NNN a change of variables,  and the fact that, due to \EEE $(g_2)$ and $(g_3)$, it is not restrictive to consider only competitors $E$ satisfying $\mathcal{H}^{d-1}(\partial^* E \cap B_1)\leq \frac{\beta}{\alpha}\gamma_{d-1}$.

 Hence, to conclude the proof, we need to show $h_0=2 {g}^{\rm BV}$. Again, we use the notation $\zeta = [u] (x)$ and $\nu = \nu_{u} (x) $ for $x \in J_{u}$.

\noindent \emph{Step 1:} $h_0(x,\zeta,\nu)\leq 2g^{\rm BV}(x,\nu)$ for $\mathcal{H}^{d-1}$-a.e.\ $x \in J_u\cap E^0$. \NNN  Repeating verbatim Step 1 in the proof of Theorem~\ref{Identification of Gamma-limit}(iii) (here, $d=2$ is not required), \EEE  we get $h_0(x,\zeta,\nu)\leq \hat{h}(x,\nu)$. Then, we recall that by Proposition~\ref{limsupliminf} we have $\hat{h}(x,\nu)=2\hat{g}(x,\nu) = 2 g^{\rm BV}(x,\nu)$.

\noindent \emph{Step 2:} $  2 \NNN {g}^{\rm BV} \EEE(x,\nu)\leq h_0(x,\zeta,\nu)$ for $\mathcal{H}^{d-1}$-a.e. $x \in J_u\cap \Omega$. We first follow the same lines of {Step 2} in the proof of Theorem \ref{Identification of Gamma-limit}(iii). In particular, we may suppose that  $x \in J_u$ \NNN satisfies \EEE 
\begin{equation}
\label{h0 relaxation}
    h_0(x,\zeta,\nu)=\lim\limits_{\rho \to 0}\frac{\overline{\mathcal{E}}(u,\emptyset,B_\rho(x))}{\gamma_{d-1}\rho^{d-1}}.
\end{equation}
\NNN 
Let $(u_n,{E}_n)_n$ be an optimal sequence for $(u,\emptyset)$, see \eqref{recovii}. \EEE In a similar fashion to \NNN \eqref{PC inequality voids 0}, \EEE \eqref{5Y}, and \eqref{6Y} we can find \NNN a sequence  \EEE  $(\rho_n)_n$, with $\rho_n \to 0$   as $n \to \infty$, such that
\begin{align}
\label{rec seq relaxation}
&  \frac{\overline{\mathcal{E}}(u,\emptyset,B_{\rho_n}(x))}{\gamma_{d-1}\rho_n^{d-1}}\geq  \frac{\NNN {\mathcal{E}}(   u_{{n}}, \EEE E_{n}   ,B_{\rho_n}(x))}{\gamma_{d-1}\rho_n^{d-1}}-\frac{1}{n},\\
\label{5Y relaxation}
& \vphantom{\int_{\Omega}} {v}_n \to \overline{u}^\nu_{0,u^-(x),u^+(x)}\:\: \text{in measure on}\:\: B_1,\\
\label{6y relaxation}
& \vphantom{\int_{\Omega}} \chi_{{F}_n}\to 0 \:\: \text{in}\:\: L^1(B_1),
\end{align}
where ${v}_n(y)=u_{ n}(x+\rho_ny)$ and $F_n =\frac{E_{n}-x}{\rho_n}$. By \eqref{rec seq relaxation} and a change of \NNN  variables \EEE we get
\begin{align*}
\frac{1}{\gamma_{d-1}\rho_n^{d-1}}\mathcal{\overline{E}}(u,\emptyset,B_{\rho_n}(x))&\geq \frac{1}{\gamma_{d-1}\rho_n^{d-1}} \int_{\partial^* E_{n} \cap B_{\rho_n}(x)}g(y,\nu_{E_{n}})\, \mathrm{d}\mathcal{H}^{d-1}(y)-\frac{1}{n}  
\\& =  \NNN \frac{1}{\gamma_{d-1} } \EEE \int_{\partial^* F_n\cap B_1}g(x+\rho_n y,\nu_{F_n})\, \mathrm{d}\mathcal{H}^{d-1}(y)-\frac{1}{n}.
\end{align*}
By continuity of $g$ we can find a sequence $\eta_n \to 0$ such that
\begin{equation*}
\frac{1}{\NNN\rho_n^{d-1}}\mathcal{\overline{E}}(u,\emptyset,B_{\rho_n}(x)) \geq \int_{\partial^* F_n\cap B_1}g(x,\nu_{F_n})\, \mathrm{d}\mathcal{H}^{d-1}(y)-\eta_n\geq \int_{\partial^* F_n\cap B_1}g^{\rm BV}(x,\nu_{F_n})\, \mathrm{d}\mathcal{H}^{d-1}(y)-\eta_n.
\end{equation*}
In the last inequality we used \NNN $g \ge g^{\rm BV}$  which immediately follows from the definition in \eqref{relaxationxxxxxxxx}. \EEE 
By virtue of \NNN \cite[Theorem 5.11]{ambrosio2000fbv} \EEE  we have that $\nu \mapsto g^{\rm BV}(y,\nu)$ is a norm, 
\NNN see also (the proof of)     \cite[Corollary~3.17]{FriPerSol20a}. \EEE Finally, we apply \cite[Theorem 5.1]{Crismale_2020} to the norm $\nu \mapsto g^{\rm BV}(\NNN x \EEE ,\nu)$ and the sequence $({v}_n,F_n)$  to obtain
\begin{equation*}
    h_0(x,\zeta,\nu)\geq \liminf\limits_{n \to \infty}  \NNN \frac{1}{\gamma_{d-1} } \EEE \int\limits_{\partial^* F_n\cap B_1} \NNN \hspace{-0.1cm} g^{\rm BV} \EEE (x,\nu_{F_n})\, \mathrm{d}\mathcal{H}^{d-1}(y)\geq  \NNN \frac{1}{\gamma_{d-1} } \EEE  \int\limits_{\Pi^\nu_0\cap B_1} \hspace{-0.1cm} 2g^{\rm BV}(x,\nu)\,\mathrm{d}\mathcal{H}^{d-1}(y)= 2g^{\rm BV}(x,\nu),
\end{equation*}
where we also used  \NNN \eqref{h0 relaxation}, \EEE \eqref{5Y relaxation}, and \eqref{6y relaxation}.  
\end{proof}

\MMM

\section*{Acknowledgements} 
This work was supported by the RTG 2339 “Interfaces, Complex Structures, and Singular Limits”
of the German Science Foundation (DFG).

\EEE

 \appendix

 \section{Remaining proofs for the bulk density}\label{otherproof-sec}

 In this section we prove \NNN Lemma \ref{LEMMA 4.2 CRIFRIESOL} \EEE and Theorem \ref{Identification of Gamma-limit}(i).   We start with   Lemma \ref{LEMMA 4.2 CRIFRIESOL} and formulate an \EEE  adaptation of \cite[Lemma 5.1, (5.5)]{SolFriCri20}.
\begin{lemma}[Blow up at points with approximate gradient]
\label{lemma5.1CRIFRIESOL}
Let $ \MMM (u,E) \EEE \in \mathcal{G}^p(\Omega)$ and let $\theta \in (0,1)$. For $\mathcal{L}^d$-a.e.\ $x \in \Omega \MMM \setminus E \EEE$ there exists a family $(u_{\varepsilon},E_{\varepsilon}) \in \mathcal{G}^p(B_{\varepsilon}(x))$ such that 
\begin{subequations}
\label{eqlemma5.1CRIFRIESOL}
\begin{align}
& \,  u_{\varepsilon}=u \MMM \text{  and } E_{\varepsilon}= E\EEE  \:\: \text{in a neighborhood of}\:\: \partial B_{\varepsilon}(x), \:\:  \label{eqlemma5.1CRIFRIESOL-1}
\\ 
 &  \MMM \lim\limits_{\varepsilon \to 0}\varepsilon^{-1-d}\mathcal{L}^d(\{ u_{\varepsilon}\neq u\})=0, \quad  \lim\limits_{\varepsilon \to 0}\varepsilon^{-1-d}\mathcal{L}^d(E\cap B_{\varepsilon}(x))=0, \label{eqlemma5.1CRIFRIESOL-2} \\  
 &  \lim\limits_{\varepsilon \to 0}\varepsilon^{-(d+p)}\int_{B_{(1-\theta)\varepsilon}(x)}\vert u_{\varepsilon}(y)-u(x)-\nabla u(x)(y-x)\vert^p\, \mathrm{d}y=0, \label{eqlemma5.1CRIFRIESOL-3}
\\ &  \lim\limits_{\varepsilon\to 0}\varepsilon^{-d} \int_{B_{\varepsilon}(x)}\vert e(u_{\varepsilon})(y)-e(u)(x)\vert^p \, \mathrm{d}y=0, \label{eqlemma5.1CRIFRIESOL-4} 
\\ & \lim\limits_{\varepsilon \to 0}\varepsilon^{-d}\mathcal{H}^{d-1}(J_{u_{\varepsilon}} \cup (\partial^*E_\eps\cap B_\varepsilon(x)) \EEE )=0, \label{eqlemma5.1CRIFRIESOL-5}
\\ & E_{\varepsilon}= \emptyset  \:\: \text{on}\:\: B_{(1-\theta)\varepsilon}(x). \label{eqlemma5.1CRIFRIESOL-6}\EEE
\end{align}
\end{subequations}
\end{lemma}

\begin{proof}
\MMM The statement essentially follows from  \cite[Lemma 5.1]{SolFriCri20}, we only indicate briefly the construction of the void set. We can choose $x$ with the additional property that $\lim_{\eps \to 0} \eps^{-d} \mathcal{H}^{d-1}(\partial^*E \cap B_\eps(x)) = 0$, as this holds for $\mathcal{L}^d$-a.e.\ $x \in \Omega \setminus E$. By the relative isoperimetric inequality this shows $\lim_{\eps \to 0} \eps^{- d^2/(d-1)}  \mathcal{L}^d(E \cap B_\eps(x)) \geq \lim_{\eps \to 0} \eps^{-d-1}  \mathcal{L}^d(E \cap B_\eps(x))=0$. Thus, by Fubini's theorem, we can choose $s_\eps \in (1-\theta, 1- \frac{\theta}{2})\eps$ such that $\lim_{\varepsilon \to 0}\varepsilon^{-d} \mathcal{H}^{d-1}(E \cap\partial B_{s_\eps}(x)) = 0$. We the define $E_{\varepsilon}=E \cap (B_{\varepsilon}(x)\setminus B_{s_\varepsilon}(x))$, and $u_\eps$ slightly differently compared to \cite[(5.5)]{SolFriCri20}, namely (in the notation therein)
$$u_\eps = u\chi_{B_\eps(x) \setminus B_{s_\eps}(x)} + (v_\eps +\bar{u}^{\rm bulk}_x )\chi_{B_{s_\eps}(x)}. $$
We note that all desired properties hold, in particular \eqref{eqlemma5.1CRIFRIESOL-1}, \eqref{eqlemma5.1CRIFRIESOL-2}, \eqref{eqlemma5.1CRIFRIESOL-5}, and \eqref{eqlemma5.1CRIFRIESOL-6} for the void set. \EEE
\end{proof}

\begin{proof}[Proof of Lemma \ref{LEMMA 4.2 CRIFRIESOL}] The proof follows the same lines of \cite[Lemma 4.2]{SolFriCri20},  up to a slight adjustment \MMM due to void set and the fundamental estimate \EEE in  Proposition \ref{fundestE0}. \EEE  It is sufficient to show \eqref{volumeequation} only for those $x \in \Omega \MMM \setminus E \EEE$ for which Lemma \ref{LEMMA 4.1 CRIFRIESOL}, Lemma \ref{lemma5.1CRIFRIESOL}, and $\lim_{\varepsilon \to 0}\varepsilon^{-d}\mu(B_{\varepsilon}(x))=\gamma_d$ hold. In fact, this is true for $\mathcal{L}^d$-a.e.\ $x \in \Omega \NNN \setminus E\EEE$.   \MMM  Then also $\lim_{\varepsilon\to 0}\varepsilon^{-d}m_{\mathcal{E}_0}(u,E,B_{\varepsilon}(x))\in \mathbb{R}$  exists,  see Lemma~\ref{LEMMA 4.1 CRIFRIESOL}. \EEE

\emph{Step 1 (Inequality ``$\leq$" in \eqref{volumeequation})}: We fix $\eta>0$ and $\theta>0$. Choose $(z_{\varepsilon},F_\varepsilon) \in\mathcal{G}^p(B_{(1-3\theta)\varepsilon}(x))$ with  \MMM  $(z_{\varepsilon},F_\varepsilon)=(\overline{u}_{x}^{\mathrm{bulk}}, \emptyset)$  \EEE in a neighborhood of $\partial B_{(1-3\theta)\varepsilon}(x)$ and
\begin{equation}
\label{volumeinequality1}
\mathcal{E}_0(z_{\varepsilon}, F_\varepsilon, B_{(1-3\theta)\varepsilon}(x))\leq m_{\mathcal{E}_0}(\overline{u}_{x}^{\mathrm{bulk}},\emptyset, B_{(1-3\theta)\varepsilon}(x))+\varepsilon^{d+1}.
\end{equation}
We extend the function $z_{\varepsilon}$ to \MMM $B_{\varepsilon}(x)$ \EEE  by setting $z_{\varepsilon}=\overline{u}_{x}^{\mathrm{bulk}}$ outside \MMM of \EEE $B_{(1-3\theta)\varepsilon}(x)$. Similarly, we extend $F_{\varepsilon}$ by setting \MMM $F_{\varepsilon} = \emptyset$ outside of \EEE $B_{(1-3\theta)\varepsilon}$. Let $(u_{\varepsilon},E_\varepsilon)_{\varepsilon}$ be the family \MMM given in \EEE Lemma \ref{lemma5.1CRIFRIESOL}. We apply Proposition \ref{fundestE0} on $(z_{\varepsilon},F_\varepsilon)$ (instead of $(u,E$)), $(u_{\varepsilon},E_\varepsilon)$ (instead of $(v,F)$), for $\eta >0$ and the sets \NNN $    A^\prime = B_{1-2\theta }(x)$, $A=B_{1-\theta }(x)$, $B=B_{1}(x)\setminus \overline{B_{1-4\theta }(x)}$, \MMM respectively their rescalings, see \eqref{sets fund est}.  \EEE Hence, we find    $\MMM (w_{\varepsilon}, D_{\varepsilon}) \EEE \in \mathcal{G}^p(\NNN B_{\varepsilon}(x)) \EEE $ such that $(w_{\varepsilon},D_\varepsilon)=(u_{\varepsilon},\MMM E_\eps)\EEE$  on $B_{\varepsilon}(x)\setminus B_{(1-\theta)\varepsilon}(x)$  and
\begin{align}
\label{applicationfundestE0-1}
\mathcal{E}_0(w_{\varepsilon},D_{\varepsilon},B_{\varepsilon}(x))\leq &(1+\eta)\big(\mathcal{E}_0(z_{\varepsilon},F_\varepsilon,A_{x,\eps})+\mathcal{E}_0(u_{\varepsilon},E_\varepsilon, B_{x,\eps})\big)   +\frac{M^p}{\varepsilon^p}\Vert z_{\varepsilon}-u_{\varepsilon}\Vert_{L^p(A_{x,\eps}\setminus A_{x,\eps}')}^p+\MMM \eps^d \EEE \eta,
\end{align}
where we used ${E_{\varepsilon}}={F_{\varepsilon}}=\MMM \emptyset \EEE$ on $A_{x,\eps}\setminus A_{x,\eps}'$, \MMM see \eqref{eqlemma5.1CRIFRIESOL-6}, \EEE and we recall that $M$ is a positive constant independent of $\varepsilon$ \MMM and $\theta$. \EEE Notice also that we have $w_{\varepsilon}=u_{\varepsilon}=u$ and ${D_{\varepsilon}}\MMM  = E_\eps \EEE =E$ in a neighborhood of $\partial B_{\varepsilon}$ by \eqref{eqlemma5.1CRIFRIESOL-1}. By \eqref{eqlemma5.1CRIFRIESOL-3} and the fact that $z_{\varepsilon}=\overline{u}_{x}^{\mathrm{bulk}}$ outside \MMM of \EEE $B_{(1-3\theta)\varepsilon}(x)$ we find
\begin{equation}
\label{applicationfundestE0-2}
\lim\limits_{\varepsilon \to 0}\varepsilon^{-(d+p)}\Vert z_{\varepsilon}-u_{\varepsilon}\Vert_{L^p(A_{x,\eps}\setminus A_{x,\eps}')}^p=\lim\limits_{\varepsilon \to 0}\varepsilon^{-(d+p)}\Vert u_{\varepsilon}-\overline{u}_{x}^{\mathrm{bulk}}\Vert^p_{L^p(A_{x,\eps})}=0.
\end{equation}
This together with \eqref{applicationfundestE0-1} shows that there exists a sequence $(\rho_{\varepsilon})_{\varepsilon} \subset (0,\infty)$ such that $\rho_{\varepsilon}\to 0$ and
\begin{align}
\label{applicationfundestE0-3}
\mathcal{E}_0(w_{\varepsilon},D_{\varepsilon},B_{\varepsilon}(x))\leq &(1+\eta)\big(\mathcal{E}_0(z_{\varepsilon},F_\varepsilon,A_{x,\eps})+\mathcal{E}_0(u_{\varepsilon},E_\varepsilon, B_{x,\eps})\big) + \varepsilon^d \rho_{\varepsilon}+ \MMM  \varepsilon^d\eta. \EEE
\end{align}
By using that $(z_{\varepsilon}, F_\varepsilon)=(\overline{u}_{x}^\mathrm{bulk}, \emptyset)$ on $B_{\varepsilon}(x)\setminus B_{(1-3\theta)\varepsilon}(x)\subset B_{x,\varepsilon}$, $(\mathrm{H1})$, $(\mathrm{H4})$, and \eqref{volumeinequality1} we have
\begin{align}
\label{applicationfundestE0-4}
\limsup\limits_{\varepsilon \to 0}\frac{\mathcal{E}_0(z_{\varepsilon},F_{\varepsilon},A_{x,\eps})}{\gamma_d\varepsilon^d}&\leq \limsup\limits_{\varepsilon \to 0}\frac{\mathcal{E}_0(z_{\varepsilon},F_\varepsilon,B_{(1-3\theta)\varepsilon}(x))}{\gamma_d\varepsilon^d}+\limsup\limits_{\varepsilon \to 0}\frac{\mathcal{E}_0(\overline{u}^{\mathrm{bulk}}_{x},\emptyset,\MMM B_{x,\varepsilon} \EEE)}{\gamma_d\varepsilon^d} \nonumber
\\ & \leq \limsup\limits_{\varepsilon \to 0}\frac{m_{\mathcal{E}_0}(\overline{u}_{x}^\mathrm{bulk},\emptyset,B_{(1-3\theta)\varepsilon}(x))}{\gamma_d\varepsilon^d}+\MMM \beta (1-(1-4\theta)^d) \EEE (1+\vert e(u)(x)\vert^p) 
\\& \leq (1-3\theta)^d \limsup\limits_{\varepsilon^\prime \to 0}\frac{m_{\mathcal{E}_0}(\overline{u}_{x}^\mathrm{bulk},\emptyset,B_{\varepsilon^\prime}(x))}{\gamma_d(\varepsilon^\prime)^d}+\MMM \beta (1-(1-4\theta)^d) \EEE (1+\vert e(u)(x)\vert^p),  \nonumber
\end{align}
where in the last step we substituted $(1-3\theta)\varepsilon$ with $\varepsilon^\prime$. By $(\mathrm{H4})$ we also find
\begin{align*}
\mathcal{E}_0(u_{\varepsilon},E_\varepsilon, B_{x,\varepsilon} )&\leq \MMM  \beta \int_{B_{x,\varepsilon}} \big( 1+|e(u_\eps)|^p \big) \, {\rm d}x  +  \beta \mathcal{H}^{d-1}(\partial^* E_\eps \cap  B_{x,\varepsilon}) + 2\beta \mathcal{H}^{d-1}(J_{u_\eps} \cap  B_{x,\varepsilon}) \EEE 
\\ & \leq  \gamma_d \varepsilon^d \beta (1-(1-4\theta)^d)(1+2^{p-1}\vert e(u)(x)\vert^p)+2^{p-1}\beta \Vert e(u_{\varepsilon})-e(u)(x)\Vert_{L^p(B_{\varepsilon}(x))}^p \\ & \ \ \ +2\beta \mathcal{H}^{d-1}(J_{u_{\varepsilon}} \MMM \cup (\partial^* E_\eps \cap B_\eps(x) )). \EEE
\end{align*}
By \eqref{eqlemma5.1CRIFRIESOL-4} and \eqref{eqlemma5.1CRIFRIESOL-5} this implies
\begin{equation}
\label{applicationfundestE0-5}
    \limsup\limits_{\varepsilon \to 0} \frac{\mathcal{E}_0(u_{\varepsilon}, E_\varepsilon,B_{x,\varepsilon})}{\gamma_d\varepsilon^d}\leq \beta (1-(1-4\theta)^d)(1+2^{p-1}\vert e(u)(x)\vert^p).
\end{equation}
Recall that $(w_{\varepsilon},D_\varepsilon)=(u,E)$ in a neighborhood of $\partial B_{\varepsilon}(x)$. This along with \eqref{applicationfundestE0-3}--\eqref{applicationfundestE0-5} and $\rho_{\varepsilon}\to 0$ yields 
\begin{align*}
    \lim\limits_{\varepsilon \to 0}\frac{m_{\mathcal{E}_0}(u,E,B_{\varepsilon}(x))}{\gamma_d \varepsilon^d}&\leq \limsup\limits_{\varepsilon \to 0}\frac{\mathcal{E}_0(w_\varepsilon,D_{\varepsilon},B_{\varepsilon}(x))}{\gamma_d \varepsilon^d}
    \\& \leq (1+\eta)(1-3\theta)^d \limsup\limits_{\varepsilon \to 0}\frac{m_{\mathcal{E}_0}(\overline{u}_{x}^{\mathrm{bulk}}, \emptyset,B_{\varepsilon}(x))}{\gamma_d \varepsilon^d}
    \\& \ \ \ + 2(1+\eta)\beta (1-(1-4\theta)^d)(1+2^{p-1}\vert e(u)(x)\vert^p)+ \NNN \frac{\eta}{\gamma_d }. \EEE
\end{align*}
Passing to $\eta,\theta \to 0$ we obtain the inequality ``$\leq$" in \eqref{volumeequation}.

\emph{Step  (Inequality ``$\geq$" in \eqref{volumeequation})}: We fix $\eta$, $\theta>0$ and let $(u_{\varepsilon},E_\varepsilon)_{\varepsilon}$ be the family \MMM given in \EEE Lemma \ref{lemma5.1CRIFRIESOL}. By \eqref{eqlemma5.1CRIFRIESOL-2} and Fubini's Theorem, for each $\varepsilon>0$ we can find $s_{\varepsilon}\in (1-4\theta,1-3\theta)\varepsilon$ such that 
\begin{subequations}
    \label{eq 5.15 CRI FRIE SOL}
\begin{align}
     & \lim\limits_{\varepsilon \to 0} \MMM \eps^{-d} \EEE \mathcal{H}^{d-1}( (\{ u_\varepsilon \neq u\} \MMM \cup  E\EEE )  \cap \partial B_{s_\varepsilon}\MMM (x) \EEE  )=0, \label{eq 5.15 CRI FRIE SOL-1}
    \\ & \mathcal{H}^{d-1}\big((J_u \cup J_{u_{\varepsilon}} \MMM \cup \partial^* E) \EEE \cap \partial B_{s_{\varepsilon}}\MMM (x) \EEE \big)=0\:\: \text{for all}\:\: \varepsilon>0. \label{eq 5.15 CRI FRIE SOL-2}
\end{align}
\end{subequations}
We consider $(z_{\varepsilon},F_\varepsilon) \in \mathcal{G}^p(B_{s_{\varepsilon}}(x))$ with  \MMM  $(z_{\varepsilon},F_\varepsilon)=(u, E)$    in a neighborhood of $\partial B_{s_{\varepsilon}}(x)$  \EEE such that
\begin{equation}
\label{eq 5.16 CRIE FRIE SOL}
\mathcal{E}_0(z_\varepsilon,F_\varepsilon,B_{s_\varepsilon}(x))\leq m_{\mathcal{E}_0}(u,E,B_{s_\varepsilon}(x))+\varepsilon^{d+1}.
\end{equation}
We extend $z_\varepsilon$ and \NNN $F_\eps$ \MMM to \EEE $B_\varepsilon(x)$ by setting 
\begin{equation}
\label{eq 5.17 CRIFRIESOL}
z_\varepsilon = u_{\varepsilon}    \:\: \text{on}\:\: B_\varepsilon (x)\setminus B_{s_\varepsilon}(x), \quad \MMM F_\eps =  E_{\varepsilon}   \:\: \text{on}\:\: B_\varepsilon (x)\setminus B_{s_\varepsilon}(x). \EEE
\end{equation}
In particular, \MMM since $(u_{\varepsilon},E_{\varepsilon}) \in \mathcal{G}^p(B_{\varepsilon}(x))$, we have \EEE $\MMM (z_{\varepsilon},F_{\varepsilon}) \EEE \in \mathcal{G}^p(B_{\varepsilon}(x))$. We apply Proposition \ref{fundestE0} on $(z_\varepsilon, F_\varepsilon)$ (in place of $(u, E)$) and $(\overline{u}_{x}^\mathrm{bulk},\emptyset)$ (in place of $(v, F)$) for \MMM $\eta >0$ and \EEE the sets indicated in \eqref{sets fund est}. Hence, there exists   $(w_\varepsilon, D_{\varepsilon}) \in \mathcal{G}^p(B_\varepsilon(x))$ such that $(w_\varepsilon,D_{\varepsilon}) = (\overline{u}_{x}^{\mathrm{bulk}},\emptyset)$  on $B_\varepsilon (x)\setminus B_{(1-\theta)\varepsilon}(x)$ and, \MMM similarly as in \eqref{applicationfundestE0-1}, \EEE
\begin{align}
\label{applicationfundestE0-6}
\mathcal{E}_0(w_{\varepsilon},D_{\varepsilon},B_{\varepsilon}(x))\leq &(1+\eta)\big(\mathcal{E}_0(z_{\varepsilon},F_{\varepsilon},A_{x,\eps} )+\mathcal{E}_0(\overline{u}_{x}^\mathrm{bulk},\emptyset, B_{x,\eps})\big)  +\frac{M^p}{\varepsilon^p}\Vert z_{\varepsilon}-\overline{u}^\mathrm{bulk}_{x}\Vert_{L^p(A_{x,\eps} \setminus A_{x,\eps}')}^p+\MMM \eps^d \EEE \eta,
\end{align}
where we used  that ${F_{\varepsilon}}={E_{\varepsilon}}=\emptyset$ on $A_{x,\eps} \setminus A_{x,\eps}'$, and we recall that $M$ depends only on $\theta$ \MMM and \EEE  $\eta>0$.
By \MMM  \eqref{eq 5.17 CRIFRIESOL}  \EEE and the choice of $s_\varepsilon$ we get $z_\varepsilon = u_\varepsilon$ outside \MMM of \EEE $B_{(1-3\theta)\varepsilon}(x)$. Thus, similarly to  {Step 1}, \MMM using \eqref{eqlemma5.1CRIFRIESOL-3} \EEE we find a sequence $(\rho_{\varepsilon})_{\varepsilon} \subset (0,\infty)$ with $\rho_\varepsilon \to 0$ such that
\begin{align}
\label{applicationfundestE0-7}
\mathcal{E}_0(w_{\varepsilon},D_{\varepsilon},B_{\varepsilon}(x))\leq &(1+\eta)\big(\mathcal{E}_0(z_{\varepsilon},F_{\varepsilon},A_{x,\eps} )+\mathcal{E}_0(\overline{u}_{x}^\mathrm{bulk},\emptyset, B_{x,\eps})\big)  +\varepsilon^d \rho_{\varepsilon} \MMM +  \varepsilon^d\eta. \EEE
\end{align}
Let us estimate the terms in \eqref{applicationfundestE0-7}. We get by $(\mathrm{H1})$, \MMM $(\mathrm{H4})$, \EEE \eqref{eq 5.16 CRIE FRIE SOL}, \eqref{eq 5.17 CRIFRIESOL}, and the choice of $s_\varepsilon$ that
\begin{align*}
 \mathcal{E}_0(z_\varepsilon,F_\varepsilon, A_{\varepsilon,x} )&\leq m_{\mathcal{E}_0}(u,E,B_{s_\varepsilon}(x))+\varepsilon^{d+1} \MMM +\beta \mathcal{H}^{d-1}\big(( E\EEE \cup \partial^* E)\cap \partial B_{s_\varepsilon}(x)\big) \nonumber
 \\&  \ \ \ + 2\beta \mathcal{H}^{d-1}\big((\{u \neq u_\varepsilon\}\cup J_{u_\varepsilon} \cup J_u)\cap \partial B_{s_\varepsilon}(x)\big)+\mathcal{E}_0(u_\varepsilon,E_\varepsilon,B_{\varepsilon,x}), \EEE
\end{align*}
\MMM where we also used that  $E_{\varepsilon}= \emptyset$ on $B_{(1-\theta)\varepsilon}(x)$, see \eqref{eqlemma5.1CRIFRIESOL-6}. \EEE
Therefore, by \eqref{applicationfundestE0-5}, \MMM \eqref{eq 5.15 CRI FRIE SOL}, and \EEE the fact that $s_\varepsilon \leq (1-3\theta)\varepsilon$ we derive 
\begin{align}
\label{eq 5.20 CRI FRIE SOL}
\limsup\limits_{\varepsilon \to 0}\frac{\mathcal{E}_0(z_{\varepsilon},F_{\varepsilon},A_{\varepsilon,x})}{\gamma_d\varepsilon^d}&\leq \MMM (1-3\theta)^d \EEE \limsup\limits_{\varepsilon \to 0}\frac{m_{\mathcal{E}_0}(u,E,B_{s_\varepsilon}(x))}{\gamma_d s^d_\varepsilon}+ \MMM \beta \EEE (1-(1-4\theta)^d)(1+2^{p-1}\vert e(u)(x)\vert^p) \nonumber
\\& \leq (1-3\theta)^d \limsup\limits_{\varepsilon \to 0}\frac{m_{\mathcal{E}_0}(u,E,\NNN B_{\varepsilon} \EEE (x))}{\gamma_d\varepsilon^d}+\MMM \beta \EEE (1-(1-4\theta)^d)(1+2^{p-1}\vert e(u)(x)\vert^p).
\end{align}
Estimating $\mathcal{E}_0(\overline{u}^\mathrm{bulk}_{x},\emptyset, B_{x,\eps})$ as in \eqref{applicationfundestE0-4}, with \eqref{applicationfundestE0-7}--\eqref{eq 5.20 CRI FRIE SOL} and $\rho_\varepsilon \to 0$  we then obtain
\begin{align*}
    \limsup\limits_{\varepsilon \to 0}\frac{\mathcal{E}_0(w_\varepsilon,D_\varepsilon,B_\varepsilon(x))}{\gamma_d\varepsilon^d}\leq &(1+\eta)(1-3\theta)^d \limsup\limits_{\varepsilon \to 0}\frac{m_{\mathcal{E}_0}(u,E,B_\varepsilon(x))}{\gamma_d\varepsilon^d}
    \\& +2(1+\eta) \MMM \beta \EEE [1-(1-4\theta)^d](1+2^{p-1}\vert e(u)(x)\vert^p)+ \NNN \frac{\eta}{\gamma_d}. \EEE
\end{align*}
Passing to $\eta$, $\theta \to 0$ and recalling $(w_\varepsilon,D_\varepsilon) = (\overline{u}_{x}^\mathrm{bulk}, \emptyset)$   in a neighborhood of $\partial B_\varepsilon(x)$ we derive
\begin{equation*}
    \limsup\limits_{\varepsilon \to 0}\frac{m_{\mathcal{E}_0}(\overline{u}_{x}^\mathrm{bulk},\emptyset,B_\varepsilon(x))}{\gamma_d \varepsilon^d}\leq \limsup\limits_{\varepsilon \to 0}\frac{\mathcal{E}_0(w_\varepsilon,D_{\varepsilon},B_\varepsilon(x))}{\gamma_d \varepsilon^d}\leq  \lim\limits_{\varepsilon \to 0}\frac{m_{\mathcal{E}_0}(u,E,B_\varepsilon(x))}{\gamma_d \varepsilon^d}.
\end{equation*}
This shows the inequality ''$\geq$" and concludes the proof of \eqref{volumeequation}.
\end{proof}

 \begin{proof}[Proof of Theorem \ref{Identification of Gamma-limit}(i)]
 The proof follows the same steps of \cite[Theorem 2.4]{FriPerSol20a}.
 
\noindent \emph{Step 1:} $f_0(x,e(u)(x))\leq \hat{f}(x,e(u)(x))$ for $\mathcal{L}^d$-a.e.\ $x \in \Omega$. 
\MMM Recalling \eqref{infimumproblem1} and \eqref{simplified_cellformula_bulk} (for $\mathcal{E}_0$ in place of $\mathcal{E}_\eps$), \EEE  we clearly get   $m_{\mathcal{E}_0}({l}_\xi,\emptyset,B_\rho(x))\leq m^{1,p}_{\mathcal{E}_0}({l}_\xi,B_\rho(x))$ \MMM for all $\xi \in \R^{d \times d}$. \EEE In addition, from \eqref{simplified_cellformula_bulk},   \eqref{f1 = f2}, \EEE and \cite[Proposition 3.13]{FriPerSol20a} we get
\begin{equation*}
    \hat{f}(x, \xi  ) \MMM =   \hat{f}(x,  {\rm sym}(\xi)  ) \EEE =\limsup\limits_{\rho \to 0}\frac{m^{1,p}_{\mathcal{E}_0}({l}_\xi,B_\rho(x))}{\gamma_d \rho^d} \quad \quad \MMM \text{ for all $\xi \in \R^{d \times d}$, $x \in \Omega$}.\EEE 
\end{equation*}
Thus, view of \eqref{def f_0}, the first inequality holds, \MMM where we also use $  {f}_0(x, \xi  )  =  {f}_0(x,  {\rm sym}(\xi)  )$ due to the fact that $\mathcal{E}_0$ satisfies $(\mathrm{H5})$. \EEE

\noindent 
\emph{Step 2:} $\hat{f}(x,e(u)(x))\leq {f}_0(x,e(u)(x))$ for $\mathcal{L}^d$-a.e.\ $x \in \Omega$.  \MMM By \eqref{the e-0 energy} applied on the pair $(u,\emptyset)$ and the Radon-Nikod\'ym Theorem we have  for $\mathcal{L}^d$-a.e.\ $x \in \Omega$ that
\begin{equation}
\label{eq7.8}
    f_0(x,e(u)(x))=   \lim\limits_{\rho \to 0}\frac{\mathcal{E}_0(u,\emptyset,B_\rho(x))}{\gamma_d\rho^d}<\infty.
\end{equation} \EEE
Let $(u_\varepsilon,E_\varepsilon)_\varepsilon$ be a recovery sequence for $(u,\emptyset)$. This along with the growth condition $(g_2)$ yields  that $(\mathcal{H}^{d-1}(\partial E_\varepsilon))_\varepsilon$ \MMM is \EEE uniformly bounded. Thus, up to \NNN a \EEE subsequence (not relabeled) there exists a Radon measure $\mu_0$ such that
\begin{equation}
\label{eq7.9}
\mu_\varepsilon \defas \mathcal{H}^{d-1}\llcorner \partial^* E_\varepsilon  \overset{\ast}{\rightharpoonup} \mu_0 \:\: \text{weakly* in the sense of measures.}  
\end{equation}
Let us notice that for $\mathcal{L}^d$-a.e.\ $x \in \Omega$ we have that
\begin{equation}
\label{eq7.10}
    \limsup\limits_{\rho \to 0}\frac{\mu_0\big(\MMM \overline{B_\rho(x)} \big) \EEE }{\gamma_{d-1}\rho^{d-1}}=0.
\end{equation}
Indeed, by contradiction we suppose that there exists a Borel set \NNN $F$ \EEE with $\mathcal{L}^d(F)>0$ and $t>0$ such that
\begin{equation}
    \limsup\limits_{\rho \to 0}\frac{\mu_0\big(\MMM \overline{B_\rho(x)} \big) \EEE }{\gamma_{d-1}\rho^{d-1}}>t
\end{equation}
for all $x \in F$. Then, as a consequence of \cite[Theorem 2.56]{ambrosio2000fbv} we would get $\mu_0 \llcorner F \geq t\mathcal{H}^{d-1}\llcorner F$ implying $\mu_0(F)=\infty$. But this is a contradiction since $\mu_0$ is finite. Due to Lemma \ref{approximate gradient} and the fact that $u \in GSBD^p(\Omega)$, the approximate gradient of $u$ exists for $\mathcal{L}^d$-a.e.\ $x \in \Omega$. \EEE Hence, it is not restrictive to take $x \in \Omega $ such that $\nabla u(x)$ exists,  and \eqref{eq7.8}, \eqref{eq7.10} hold. Since $\mathcal{E}_0(u,\emptyset,\cdot)$ is a Radon measure, there exists a subsequence $(\rho_n)_n \subset (0,\infty)$ with $\rho_n \searrow 0$ as $n \to \infty$ such that $\mathcal{E}_0(u,\emptyset,\partial B_{\rho_n}(x))=0$ for every $n \in \mathbb{N}$ and  \eqref{f1 = f2}  holds along the sequence $(\rho_n)_n$ i.e.,
\begin{equation}
\label{eq7.12}
\hat{f}(x,e(u)(x))=\lim\limits_{n \to \infty}\limsup\limits_{\varepsilon \to 0}\frac{m_{\mathcal{E}_{\varepsilon}}^{1,p}({l}_{\nabla u(x)},B_{\rho_n}(x))}{\gamma_d \rho^d_n}.
\end{equation}
\MMM (Here, we use again that $\hat{f}(x,\xi) = \hat{f}(x,{\rm sym}(\xi))$ for all $\xi \in \R^{d \times d}$, see \cite[Proposition 3.13]{FriPerSol20a}.) \EEE By Remark~\ref{remark recovery sequences}, for every $n \in \mathbb{N}$ there exists $\varepsilon_n$ such that for every $\varepsilon \leq \varepsilon_n$ it holds
\begin{equation}
\label{eq7.13}
\frac{\mathcal{E}_0(u,\emptyset,B_{\rho_n}(x))}{\gamma_d \rho^d_n}\geq \frac{\mathcal{E}_{\varepsilon}(u_{\varepsilon},E_{\varepsilon},B_{\rho_n}(x))}{\gamma_d \rho^d_n}-\frac{1}{n}.
\end{equation}
We define the functions
\begin{equation*}
 \tilde{u}_{\varepsilon}^n(y)\defas \frac{u_{\varepsilon}(x+\rho_n y)-u_{\varepsilon}(x)}{\rho_n}  
 \quad \text{and}\quad \tilde{u}^n(y)\defas \frac{u(x+\rho_n y)-u(x)}{\rho_n}  \quad \MMM \text{for } y \in B_1, \EEE
\end{equation*}
and note that $\tilde{u}^n_{\varepsilon}$ converges to $\tilde{u}^n$ in measure on $B_1$ as $\varepsilon \to  0$ since $u_{\varepsilon}$ converges in measure to $u$. Notice that $(\tilde{u}_{\varepsilon}^n, \tilde{E}^n_{\varepsilon}) \in \mathcal{W}^{1,p}(\MMM B_{1} \EEE )$, where $\tilde{E}^n_{\varepsilon}\defas \frac{E_{\varepsilon}-x}{\rho_n}$, \MMM and that \EEE  $J_{\tilde{u}_{\varepsilon}^n}\subset \MMM \partial^* \EEE \tilde{E}^n_{\varepsilon}\cap B_1$ for every $n \in \mathbb{N}$ and $\varepsilon>0$. We now show by a diagonal argument that, up to passing to a smaller $\varepsilon_n$, the sequence \MMM $v_n \defas \tilde{u}_{\varepsilon_n}^n$ \EEE satisfies
\begin{equation}
\label{eq7.14}
    v_n \to {l}_{\nabla u(x)} \:\: \text{in measure on}\:\: B_1,
\end{equation}
and 
\begin{equation}
\label{eq7.15}
  \lim\limits_{n \to \infty} \mathcal{H}^{d-1}(J_{v_n}) \MMM = \EEE  0.
\end{equation}
Since $u$ is approximately differentiable in $x$, we have $\tilde{u}^n\to {l}_{\nabla u(x)}$ in measure on $B_1$ as $n \to \infty$. Consequently, \eqref{eq7.14} can be achieved. Moreover, by a change of  \MMM variables \EEE and by recalling \eqref{eq7.9} we have
\begin{equation}
\limsup\limits_{\varepsilon \to 0}\mathcal{H}^{d-1}\big(\partial^* \tilde{E}_{\varepsilon}^n \MMM \cap B_1 \EEE \big)=\limsup\limits_{\varepsilon \to 0} \frac{\mathcal{H}^{d-1}(\partial^* E_{\varepsilon}\cap \MMM B_{\rho_n} \EEE (x))}{ \MMM  \rho^{d-1}_n}\leq \limsup\limits_{\varepsilon \to 0}\frac{\mu_{\varepsilon}(\overline{B_{\rho_n}(x)})}{\MMM \rho^{d-1}_n}\leq \frac{ \MMM \mu_0 \EEE (\overline{B_{\rho_n}(x)})}{\MMM \rho^{d-1}_n}.
\end{equation}
Thus, by \eqref{eq7.10} we get
\begin{equation}
\MMM \limsup\limits_{n \to \infty} \EEE \limsup\limits_{\varepsilon \to 0} \mathcal{H}^{d-1}(J_{\Tilde{u}^n_{\varepsilon}})\leq \MMM \limsup\limits_{n \to \infty} \EEE \limsup\limits_{\varepsilon \to 0} \mathcal{H}^{d-1}\big(\partial^* \tilde{E}_{\varepsilon}^n \MMM \cap B_1 \EEE \big)=0.   
\end{equation}
Thus, \eqref{eq7.15} can be ensured with a diagonal argument. By \eqref{eq7.12} we can choose $\varepsilon_n$ with the additional property
\begin{equation}
\label{eq7.18}
\hat{f}(x,e(u)(x))= \lim\limits_{n \to \infty}\frac{m^{1,p}_{\mathcal{E}_{\varepsilon_n}}({l}_{\nabla u(x)}, B_{\rho_n}(x))}{\gamma_d\rho^d_n}.
\end{equation}
By \eqref{eq7.13} and \MMM the \EEE change of variables $\MMM z \EEE =x+\rho_n y$ we get
\begin{equation}
\int_{B_1}f_{\varepsilon_n}\big(x+\rho_n y,e(v_n)(y)\big)\, \mathrm{d}y=\frac{\int_{B_{\rho_n}(x)}f_{\varepsilon_n}(z,e(u_{\varepsilon_n})(z))\, \mathrm{d}z}{\MMM \rho^{d}_n}\leq \frac{\mathcal{E}_0(u,\emptyset,B_{\rho_n}(x))}{\MMM \rho^{d}_n}+\frac{\MMM \gamma_d}{n}.
\end{equation}
In addition, taking into consideration \eqref{eq7.8} we get
\begin{equation}
\label{eq7.20}
    \limsup\limits_{n \to \infty}\frac{1}{\gamma_d}\int_{B_1}f_{\varepsilon_n}\big(x+\rho_n y,e(v_n)(y)\big)\, \mathrm{d}y \leq \lim\limits_{n \to \infty}\frac{\mathcal{E}_0(u,\emptyset, B_{\rho_n}(x))}{\gamma_d\rho^d_n}=f_0(x,e(u)(x)).
\end{equation}
\MMM By $(f_2)$  we also get $\sup_{n \in \N} \Vert e(v_n) \Vert_{L^2(B_1)} < \infty$. This along with  \eqref{eq7.14} and \eqref{eq7.15} allows us to apply \EEE  \cite[Lemma 5.1]{FriPerSol20a} to the sequence $(v_n)_n$. We find a sequence $(w_n)_n \subset W^{1,p}(B_1;\mathbb{R}^d)$ such that $(\vert \nabla w_n \vert^p)_n$ is equiintegrable, $\lim\limits_{n \to \infty}\Vert w_n - {l}_{\nabla u(x)}\Vert_{L^p(B_1)}= 0$ and $\lim\limits_{n \to \infty}\mathcal{L}^d(\{ w_n \neq v_n\} \cup \{e(v_n) \neq e(w_n)\} )=0$. As a consequence of these properties \NNN and $(f_3)$, \EEE it holds
\begin{equation}
\label{eq7.21}
 \limsup\limits_{n \to \infty}\int_{B_1}f_{\varepsilon_n}\big(x+\rho_n y, e(v_n)(y)\big)\, \mathrm{d}y= \limsup\limits_{n \to \infty}\int_{B_1}f_{\varepsilon_n}\big(x+\rho_n y, e(w_n)(y)\big)\, \mathrm{d}y .
\end{equation}
We now modify $(w_n)_n$ in such a way \MMM that \EEE they attain the boundary datum ${l}_{\nabla u(x)}$ in a neighborhood of $\partial B_1$ and in such a way that the energy does not increase asymptotically. This follows from a standard application of the fundamental estimate in \cite[Chapter 18]{DalMaso:93} to the functional
\begin{equation*}
    \tilde{\mathcal{F}}_n(z)\defas  \MMM \frac{1}{\gamma_d} \EEE \int_{B_1}f_{\varepsilon_n}(x+\rho_n y, e(z)(y))\, \mathrm{d}´\MMM y, \EEE \quad z \in W^{1,p}(B_1;\mathbb{R}^d),
\end{equation*}
and the fact that $\lim\limits_{n \to \infty}\Vert w_n - {l}_{\nabla u(x)}\Vert_{L^p(\Omega)}= 0$. Precisely, for any $\eta \in (0,1)$ we can find a sequence $(z^\eta_n)_n$ such that $z^\eta_n={l}_{\nabla u(x)}$ \MMM near  $\partial B_1 $ \EEE and  
\begin{equation}
\label{fund est dal maso}
\limsup\limits_{n \to \infty}\tilde{\mathcal{F}}_n(z^\eta_n)\leq (1+\eta)\limsup\limits_{n \to \infty}\tilde{\mathcal{F}}_n(w_n)+\mathrm{err}(\eta),
\end{equation}
with ${\rm err}(\eta)\to 0$ as $\eta \to 0$. From \eqref{eq7.20} \MMM and \EEE \eqref{eq7.21} we get
\begin{equation}
\label{eq7.23}
\limsup\limits_{n \to \infty}\tilde{\mathcal{F}}_n(z^\eta_n)\leq (1+\eta)f_0(x,e(u)(x))+\mathrm{err}(\eta).
\end{equation}
On the other hand, with a change of variables, we get
\begin{equation}
\label{eq7.24}
\tilde{\mathcal{F}}_n(z^\eta_n)=  \MMM \frac{1}{\gamma_d} \EEE \int_{B_1} f_{\varepsilon_n}\big(x+\rho_n y^\prime, e(z_n^\eta)(y^\prime)\big)\, \mathrm{d}y^\prime = \frac{1}{ \MMM \gamma_d \EEE \rho_n^d}\int_{B_{\rho_n}(x)}f_{\varepsilon_n}\big(y,e(\overline{z}_i^\eta)(y)\big)\, \mathrm{d}y,
\end{equation}
where $\overline{z}_n^\eta(y)\defas \rho_n z^\eta_n((y-x)/\rho_n)+{l}_{\nabla u(x)}x$ for $y \in B_{\rho_n}(x)$. In particular, observe that, since $z^\eta_n={l}_{\nabla u(x)}$ in a neighborhood of $\partial B_1$, $\overline{z}_n^\eta={l}_{\nabla u(x)}$ in a neighborhood of $\partial B_{\rho}(x)$. Hence, \eqref{eq7.18}, \eqref{eq7.23}, and \eqref{eq7.24} imply
\begin{equation*}
\hat{f}(x,e(u)(x))\leq (1+\eta)f_0(z,e(u)(x))+\mathrm{err}(\eta),
\end{equation*}
and so the desired inequality follows after sending $\eta \to 0$. 
 \end{proof}

     \MMM
\section{Definition and properties of $GBD$-functions} \EEE
\label{section preliminaries}

\subsection{G(S)BD functions}
We \NNN recall  \EEE the space of generalized functions of bounded deformation and \NNN its \EEE main properties. For further details we refer to \cite{DALMASO_GBD}.  In the following, $\Omega$ always \MMM denotes a set in $\mathcal{A}_0$.  Given a Borel set $B \subset \mathbb{R}^d$,  \EEE  for every $\xi \in \mathbb{S}^{d-1}$ \MMM and \EEE $y \in \Pi^\xi_0$ we denote by $B^\xi_y$ the set $
      B_y^{\xi}=\{ t \in \mathbb{R}: y+t\xi \in B \}.$
\begin{definition}[Generalized functions of bounded deformation]
\label{def3}
 The space $GBD(\Omega)$ of generalized functions of bounded deformation consists of \MMM all \EEE  $\mathcal{L}^d$-measurable functions $u$ such that  there exists a bounded Radon measure $\lambda_u$ satisfying the following property: for $\mathcal{H}^{d-1}$-a.e.\ $y \in \Pi_0^{\xi}$ the function $\hat{u}^{\xi}_y(t)\defas u(y+\xi t) \cdot \xi$ belongs to $BV_{\mathrm{loc}}(\Omega_y^{\xi})$ and
\begin{equation}
\label{eq2.2}
\int_{\Pi_0^{\xi}}{\big(|D\hat{u}^{\xi}_y|(B_y^{\xi}\setminus J_{\hat{u}^{\xi}_y}^1)+\mathcal{H}^0(B_y^{\xi}\cap J_{\hat{u}^{\xi}_y}^1)\big)}\, \mathrm{d}\mathcal{H}^{d-1}(y)\leq \lambda_u(B) 
\end{equation}
for every Borel set $B \subset \Omega$, where  $J_{\hat{u}^{\xi}_y}^1=\{ t \in J_{\hat{u}^{\xi}_y}: |(\hat{u}_y^{\xi})^+(t)-(\hat{u}_y^{\xi})^-(t)|\geq 1 \}$.
\end{definition}
\begin{definition}[Generalized special functions of bounded deformation]
\label{def5}
The space $GSBD(\Omega)$ of generalized special functions of bounded deformation is the space of \MMM all \EEE functions $u \in GBD(\Omega)$ such that for every $\xi \in \mathbb{S}^{d-1}$  and for $\mathcal{H}^{d-1}$-a.e.\ $y \in \Pi_0^{\xi}$ the function $\hat{u}_y^{\xi}$ belongs to $SBV_{\mathrm{loc}}(\Omega^{\xi}_{y})$.
\end{definition}
\begin{lemma}[Weak approximate symmetric differentiability]
\label{theorem4}

Let $u \in GBD(\Omega)$. Then there exists a function $e(u) \in L^1(\Omega; \mathbb{R}_{\rm sym}^{d\times d})$ such that for $\mathcal{L}^d$-a.e.\ $x \in \Omega$ it holds that
\begin{equation*}
\limsup\limits_{\varepsilon \to 0}\varepsilon^{-d}{\mathcal{L}^d\Big(\Big\{y \in B_{\varepsilon}(x): \frac{|(u(y)-u(x)-e(u)(x)(y-x))\cdot (y-x)|}{|y-x|^2}>\rho \Big\}\Big)}=0 \:\: \text{for all}\:\: \rho>0.
\end{equation*}
In addition, for every $\xi \in \mathbb{S}^{d-1}$ and for $\mathcal{H}^{d-1}$-a.e.\ $y \in \Pi_0^{\xi}$ we have
\begin{equation*}
e(u)(y+t\xi)\xi \cdot \xi = \nabla {\hat{u}}_y^{\xi}(t) \:\: \text{for  \NNN $\mathcal{L}^1$-a.e.}\:\: t \in \Omega_y^{\xi}.
\end{equation*}
\end{lemma}
Moreover, \MMM for every \EEE  function $u \in GBD(\Omega)$ the approximate jump set $J_u$ (\cite[Definition 4.30]{ambrosio2000fbv}) is $\mathcal{H}^{d-1}$-rectifiable.
We recall that $GSBD^p(\Omega)\defas \{ u \in GSBD(\Omega)\colon e(u) \in L^p(\Omega;\mathbb{R}_{\mathrm{sym}}^{d \times d}), \MMM \mathcal{H}^{d-1}(J_u) < \infty  \EEE  \}$.
\begin{lemma}[Approximate gradient, \cite{CagChaSca20, doi:10.1137/17M1129982}] 
\label{approximate gradient}
\MMM Let \EEE  $1<p<\infty$ and $u \in GSBD^p(\Omega)$. Then \MMM for \EEE $\mathcal{L}^d$-a.e.\ $x \in \Omega$ there exists a matrix in $\mathbb{R}^{d \times d}$, denoted by $\nabla u (x)$, such that
\begin{equation*}
    \lim\limits_{\varepsilon \to 0}\varepsilon^{-d} \mathcal{L}^d\Big(\Big\{x \in B_{\varepsilon}(x) \colon \frac{\vert u(y)-u(x)-\nabla u(x)(y-x)\vert}{\vert y-x \vert}>\rho\Big\}\Big)=0\:\: \text{for all}\:\: \rho>0.
\end{equation*}
\end{lemma}

\begin{theorem}[\NNN Closure \EEE in $GSBD^p$, \cite{DALMASO_GBD}]
\label{compactness in GSBD^p}
Let $(u_n)_n$ be a sequence in $GSBD^p(\Omega)$ such that
\begin{equation}
    \sup_{n \in \mathbb{N}} \big(\Vert e(u_n)\Vert^p_{L^p(\Omega)}+\mathcal{H}^{d-1}(J_{u_n})\big)<\infty,
\end{equation}
\NNN and $u_n \to u$ in measure for some $u \in L^0(\Omega;\R^d)$. \EEE 
\NNN Then, \EEE  $u \in GSBD^p(\Omega)$, \NNN and \EEE
\begin{align*}
{\rm (i)}  \ \     & u_n \to u \:\: \text{in}\:\: L^0(\Omega ;\mathbb{R}^d),\\
{\rm (ii)}  \ \     & e(u_n) \rightharpoonup e(u)\:\: \text{weakly in}\:\: L^p(\Omega  ;\mathbb{R}_{\mathrm{sym}}^{d \times d}), \\
{\rm (iii)}  \ \     & \liminf\limits_{n \to \infty}\mathcal{H}^{d-1}(J_{u_n})\geq \mathcal{H}^{d-1}(J_u ).
\end{align*}
\end{theorem}

\subsection{Korn's inequality \MMM in $GSBD$\EEE}
We say that $a \colon \Omega \to \mathbb{R}^d$ is a \emph{rigid motion} if $a(x)=Ax+b$ with $A \in \mathbb{R}_{\mathrm{skew}}^{d \times d}$   and $b \in \mathbb{R}^d$. \EEE We  recall a rigidity result for $GSBD^p$-functions due to {\sc Cagnetti, Chambolle, and Scardia} \cite[Theorem 1.1]{CagChaSca20}. 
\begin{theorem}[Korn's inequality for functions with small jump set]
\label{Korn inequality for functions with small jump set}
\MMM Let \EEE $1<p<\infty$. There exists a constant $c=c(p,d,\Omega)>0$ such that for all $u \in GSBD^p(\Omega)$ there exists a set of finite perimeter $\omega \subset \Omega$ and a rigid motion $a \colon \Omega \to \mathbb{R}^d$ such that 
\begin{equation*}
\mathcal{H}^{d-1}(\partial^* \omega) \leq c\mathcal{H}^{d-1}(J_u), \quad \quad  \mathcal{L}^d(\omega)\leq c(\mathcal{H}^{d-1}(J_u))^{\frac{d}{d-1}}   , 
\end{equation*}
\begin{equation}
\label{rigidmotioninequality}
\Vert u-a \Vert_{L^p(\Omega \setminus \omega)} + \Vert \nabla u - \nabla a \Vert_{L^p(\Omega \setminus \omega)}\leq c\Vert e(u)\Vert_{L^p(\Omega)}.  
\end{equation}
Moreover, there exists $\overline{u} \in W^{1,p}(\Omega; \mathbb{R}^d)$ such that $\overline{u}=u$ in $\Omega \setminus \omega$ and
\begin{equation}
\label{eq3.0}
\Vert e(\overline{u})\Vert_{L^p(\Omega)}\leq c \Vert e(u)\Vert_{L^p(\Omega)}.  
\end{equation}
In particular, the rigid motion $a$ can be chosen in such a way that \eqref{rigidmotioninequality} holds and
\begin{equation}
\label{sobolevreplacement}
\Vert \overline{u}-a \Vert_{L^p(\Omega)} + \Vert \nabla \overline{u} - \nabla a \Vert_{L^p(\Omega)}\leq c\Vert e(u)\Vert_{L^p(\Omega)}.
\end{equation}
\end{theorem}
\begin{remark}[Scaling invariance]
\label{scalinginvarianceoncubes} 
By standard rescaling arguments one can show that, if $\Omega=Q_{\rho}$ for $\rho>0$, then we find $\omega \subset Q_\rho$, a constant $\overline{c}>0$, and a rigid motion $a$ such that
\begin{equation*}
\mathcal{H}^{d-1}(\partial^* \omega) \leq \overline{c}\mathcal{H}^{d-1}(J_u), \:\: \mathcal{L}^d(\omega)\leq \overline{c}(\mathcal{H}^{d-1}(J_u))^{\frac{d}{d-1}}      
\end{equation*}
and
\begin{equation*}
\Vert u-a\Vert_{L^p(Q_{\rho} \setminus \omega)}^p \leq \overline{c}\rho^p \Vert e(u)\Vert _{L^p(Q_{\rho})}^p,
\end{equation*}
where $\overline{c}=\overline{c}(p,d)$ is independent of the sidelength $\rho$. 
\end{remark}

\MMM In connection with the application of Theorem \ref{Korn inequality for functions with small jump set}, we will make use of the following  elementary lemma for affine functions,  \NNN see  \EEE \cite[Lemma 3.4]{Friedrich_Solombrino}.   \EEE
\begin{lemma}
\label{lemmamatrix}
Let $G \in \mathbb{R}^{d\times d}$ and $b \in \mathbb{R}^d$. Let $\delta>0, R>0$ and let $\psi\colon \MMM  [0,\infty) \EEE \to \MMM  [0,\infty) \EEE$ be a continuous, strictly increasing function with $\psi(0)=0$. Consider a measurable, bounded set $E \subset \mathbb{R}^d$ with $E \subset \MMM B_R \EEE$ and $\mathcal{L}^d(E)\geq \delta$. Then there exists a continuous, strictly increasing function $\tau_{\psi}\colon \psi(\MMM  [0,\infty) \EEE)\to \MMM  [0,\infty) \EEE$, with $\tau_{\psi}(0)=0$, only depending on $\delta$, $R$, and $\psi$ such that
\begin{equation*}
\vert G\vert+\vert b\vert\leq \tau_{\psi}\Big(\frac{1}{\mathcal{L}^d(E)}\int_{E}{\psi(\vert Gx+b \vert) \,  \MMM{\rm d}x \EEE  }\Big).
\end{equation*}
If $\psi(t)=t^p$, $p \in [1,\infty)$, then $\tau_{\psi}$ can be chosen as $\tau_{\psi}(t)=c t^{\frac{1}{p}}$ with $c=c(p,\delta,R)>0$. Moreover, there exists a constant $c_0>0$, depending only on $p$, $\delta$ and $R$, such that
\begin{equation}
\label{eq3.2} 
\Vert Gx+b\Vert_{L^{\infty}(B_R)}\leq c_0 \Vert Gx+b \Vert_{L^1(E)}.
\end{equation}
\end{lemma}

\addcontentsline{toc}{chapter}{Bibliography}
\bibliographystyle{plain}
\bibliography{bibliography.bib}

\end{document}